\documentclass[12pt]{article}
\usepackage{fullpage}

\usepackage{graphicx}
\usepackage{caption}
\usepackage{float}

\usepackage{etoc}

\usepackage{natbib}
\usepackage{amsmath}
\usepackage{amssymb}
\usepackage{amsthm}
\usepackage{derivative}
\usepackage{mathtools}
\mathtoolsset{showonlyrefs}
\usepackage{thmtools}

\usepackage[pdfencoding=auto, psdextra]{hyperref}
\hypersetup{
	colorlinks,
	citecolor=blue,
	filecolor=black,
	linkcolor=black,
	urlcolor=black
}

\usepackage{bookmark}
\bookmarksetup{depth=2}  

\theoremstyle{plain}
\newtheorem{theorem}{Theorem}[section]
\newtheorem{lemma}{Lemma}
\newtheorem*{lemma*}{Lemma}
\newtheorem{corollary}{Corollary}
\newtheorem{proposition}{Proposition}
\newtheorem{conjecture}{Conjecture}

\DeclareMathOperator{\iid}{\stackrel{iid}{\sim}}
\DeclareMathOperator{\st}{\text{ s.t. }}
\DeclareMathOperator{\period}{\text{.}}
\DeclareMathOperator{\comma}{\text{ , }}
\DeclareMathOperator{\as}{\text{ as }}
\DeclareMathOperator{\since}{\text{since }}
\DeclareMathOperator{\Given}{\text{Given }}
\DeclareMathOperator{\Test}{\text{Test }}
\DeclareMathOperator{\vs}{\text{ versus }}
\DeclareMathOperator{\where}{\text{ where }}
\DeclareMathOperator{\by}{\text{by }}
\DeclareMathOperator{\textif}{\text{if }}
\DeclareMathOperator{\otherwise}{\text{otherwise }}
\DeclareMathOperator{\textand}{\text{ and }}
\DeclareMathOperator{\textor}{\text{ or }}
\DeclareMathOperator{\for}{\text{ for }}

\DeclarePairedDelimiter{\ceil}{\lceil}{\rceil}

\DeclareMathOperator{\supp}{supp}

\newcommand{\norm}[1]{\lVert #1 \rVert}
\newcommand{\inner}[2]{\langle #1 , #2 \rangle}

\newcommand{\Normal}{\mathcal{N}}
\DeclareMathOperator{\Poisson}{Poi}
\DeclareMathOperator{\Poi}{Poi}

\newcommand{\Rd}{\mathbb{R}^d}
\newcommand{\R}{\mathbb{R}}

\newcommand{\Zplus}{\mathbb{Z}_+}
\newcommand{\D}{\mathcal{D}}

\newcommand{\Pdist}{\mathcal{P}}
\newcommand{\Pnull}{\mathcal{P}_0}
\newcommand{\Palt}{\mathcal{P}_1}

\DeclareMathOperator{\GS}{GS}

\DeclareMathOperator{\GW}{GW}
\newcommand{\GWnull}{\mathcal{G}_0}
\newcommand{\GWalt}{\mathcal{G}_1^s}
\newcommand{\GWnullTilde}{\tilde{\mathcal{G}}_0}
\newcommand{\GWaltTilde}{\tilde{\mathcal{G}}_1^s}

\newcommand{\Multinomial}{\mathcal{M}}
\newcommand{\Multi}{\Multinomial}
\newcommand{\M}{\Multinomial}

\newcommand{\Dnull}{\mathcal{D}_0}
\newcommand{\Dalt}{\mathcal{D}_1^s}
\newcommand{\DnullTilde}{\tilde{\mathcal{D}}_0}
\newcommand{\DaltTilde}{\tilde{\mathcal{D}}_1^s}

\newcommand{\Tolerance}{\lambda}

\begin{document}

\begin{center}

{\bf{\Large{Testing Imprecise Hypotheses }}}

\vspace*{.2in}

{{
\begin{tabular}{cccc}
Lucas Kania$^{\dagger}$ & Tudor Manole$^{\diamond}$ 
& Larry Wasserman$^{\dagger, \ddagger}$
& Sivaraman Balakrishnan$^{\dagger,\ddagger}$ 
\end{tabular}
}}

\vspace{.15in}

\begin{tabular}{c}
	$^\dagger$Department of Statistics 
    and Data Science, Carnegie Mellon University\\
	$^\diamond$Statistics and Data Science Center, Massachusetts Institute of Technology \\
	$^\ddagger$Machine Learning Department, Carnegie Mellon University \\[0.12in]
\end{tabular}
\begin{tabular}{cc}
   				\texttt{lucaskania@cmu.edu}, \texttt{tmanole@mit.edu},  
   				 \texttt{\{larry,siva\}@stat.cmu.edu} 
\end{tabular}

\vspace{.15in}

\today

\end{center}

\begin{abstract}
Many scientific applications involve testing theories that are only partially specified. This task often amounts to testing the goodness-of-fit of a candidate distribution while allowing for reasonable deviations from it. The {\it tolerant testing} framework provides a systematic way of constructing such tests. Rather than testing the simple null hypothesis that data was drawn from a candidate distribution, a tolerant test assesses whether the data is consistent with any distribution that lies within a given neighborhood of the candidate. As this neighborhood grows, the tolerance to misspecification increases, while the power of the test decreases. In this work, we characterize the information-theoretic trade-off between the size of the neighborhood and the power of the test, in several canonical models. On the one hand, we characterize the optimal trade-off for tolerant testing in the Gaussian sequence model, under deviations measured in both smooth and non-smooth norms. On the other hand, we study nonparametric analogues of this problem in smooth regression and density models. Along the way, we establish the sub-optimality of the classical $\chi^2$-statistic for tolerant testing, and study simple alternative hypothesis tests.
\end{abstract}

\etocdepthtag.toc{mtchapter}
\etocsettagdepth{mtchapter}{section}
\etocsettagdepth{mtappendix}{none}
\etocsettagdepth{mtreferences}{section}
{
\renewcommand{\baselinestretch}{0}
\normalsize
\parskip=0em
\renewcommand{\contentsname}{\normalsize Table of contents}
\tableofcontents
}

\pagebreak

\section{Introduction}\label{sec:introduction}

Evaluating whether data support or refute a hypothesis is a fundamental task in
the sciences. Statistical hypothesis testing provides a formal framework for this problem. Typically, a researcher states a null hypothesis
which represents the absence of a scientific discovery, and designs a test to assess it.
A rejection of the null hypothesis thus suggests a discovery of the scientific phenomenon of interest.

This process conflates statistical and scientific discovery. To test a hypothesis, one specifies a model representing \textit{no scientific discovery} and checks whether the observed data aligns with its predictions. If the model accurately reflects the underlying process, statistical rejection implies scientific discovery. In practice, however, models are imperfect, and rejection simply indicates a discrepancy between the model and the data, not necessarily a discovery.

To mitigate this issue,  practitioners attempt to account for acceptable deviations from the \textit{no discovery} model, thus requiring stronger evidence for rejection.
A prominent example is the field of experimental high-energy physics, in which
datasets are collected from large-scale particle colliders with the aim of probing the {\it Standard Model} of particle physics~\citep{van2014role}.
The Standard Model is a classification of elementary particles and the forces which act upon them,
which leads to predictions of the probability distribution $P_0$ governing the kinematic behavior of particles within a collider. Given a set of experimental observations $X_1,\dots,X_n  \sim P$ of such particles, a sought-after goal  is to test
the {goodness-of-fit} hypothesis
\begin{align}\label{eq:GOF}
H_0: P = P_0, \vs H_1: d(P,P_0) > \epsilon,
\end{align}
where $d$ is a divergence between probability distributions, and $\epsilon > 0$
is a separation parameter. A rejection of the null indicates a departure
from the Standard Model, and thus a discovery of new physics.  One of the difficulties in testing this hypothesis lies in the fact that $P_0$ is unavailable in closed form, and is typically approximated through simulators. These simulators themselves rely on parameters which are either inferred from past analyses, or are computed mathematically, typically using approximations. As a result of these various imperfections, physicists acknowledge that the putative null distribution $P_0$ may be slightly misspecified.
The degree of misspecification is quantified through the notion of {\it systematic uncertainties}~\citep{heinrich2007systematic}, and an important goal is to construct goodness-of-fit tests which are tolerant to deviations of the null hypothesis within the confines of a systematic uncertainty set.

Motivated by such applications---which we return
to in Section~\ref{sec:null_approximation}---the goal of our work is to develop  methods
for testing
{\it imprecise} goodness-of-fit hypotheses, and to sharply
characterize the information-theoretic trade-off between statistical power
and  imprecision.
We   study these questions for several of
canonical divergence
functionals and model classes, such as high-dimensional location families and smoothness classes.
Although  each problem will involve different optimal tests, and different
information theoretic trade-offs, we identify a variety of common
phenomena across these problems which allow us to reason more broadly about
null misspecification in goodness-of-fit testing.

To formalize our problem, we adopt a framework known
as {\it tolerant testing}~\citep{ingsterTestingHypothesisWhich2001,parnasTolerantPropertyTesting2006}.
Unlike the goodness-of-fit problem~\eqref{eq:GOF}, the tolerant testing problem consists of the hypotheses
\begin{align}\label{eq:generic_tolerant_testing}
H_0: d(P,P_0)\leq \epsilon_0, \vs H_1: d(P,P_0) \geq \epsilon_1,
\end{align}
where $\epsilon_0 > 0$ is a {\it tolerance} or {\it imprecision} parameter, and $\epsilon_1-\epsilon_0>0$
is a {\it separation} parameter.
By analogy with the physical applications described previously, one
can think of $\epsilon_0$ as defining a {\it systematic} uncertainty
ball $\{\tilde P_0: d(\tilde P_0,P_0) \leq \epsilon_0\}$
centered at $P_0$, and the goal
is to detect deviations not only from $P_0$, but from any distribution
which lies in this ball.

We study tolerant testing  via the {\it minimax framework} for hypothesis testing,
which has its roots in the foundational works of \citet{mannChoiceNumberClass1942}, \citet{ingsterMinimaxNonparametricDetection1982,ingsterAsymptoticallyMinimaxHypothesisI1993}, \citet{ermakovAsymptoticallyMinimaxTests1990,ermakovMinimaxDetectionSignal1991}, and \citet{lepskiMinimaxNonparametricHypothesis1999}.
Given a tolerance parameter $\epsilon_0$,
the main figure of merit in the minimax framework  is
the notion of {critical separation}, also called the critical separation radius,
defined as the smallest
distance $\epsilon_1-\epsilon_0$
for which the hypotheses can be tested with nontrivial power, uniformly over
all elements of the null and alternative
classes; see~ Section \ref{sec:minimax_framework} below for
precise definitions. Any test which
achieves this condition is said to be
{minimax optimal}.

The minimax criterion stands in contrast
to more traditional frameworks
to quantifying the difficulty
of hypothesis testing problems,
which typically quantify the relative efficiency of
tests along contiguous alternatives.
These two viewpoints
both have their strengths and drawbacks~\citep{bickelTailormadeTestsGoodness2006}.
The contiguous framework is well-suited
to problems where the set of alternatives
is believed to lie in a small set of directions
that can accurately be captured by one-dimensional
 paths of alternatives.
On the other hand,
analyzing power against such paths
can be misleading in  genuinely
nonparametric problems,
as they typically cannot capture the full dimensionality
of the space of alternatives, and can therefore
lead to conclusions which obscure the curse of
dimensionality~\citep{arias-castroRememberCurseDimensionality2018}.
We find the minimax framework to be particularly
illuminating for analyzing the tolerant
testing problem, since it allows
us to quantify the interplay
between the geometry and dimensionality
of the composite null and alternative classes
as a function of the parameter $\epsilon_0$.



From a more technical lens, the minimax
perspective on tolerant
testing provides a common framework
for analyzing
two seemingly-distinct but well-studied
problems in the minimax theory literature:
\begin{enumerate}
\item[(a)] {\it Testing} the
  goodness-of-fit hypothesis~\eqref{eq:GOF}.
\item[(b)] {\it Estimating} the divergence functional
$d(P,P_0)$.
\end{enumerate}
It is clear that problem (a)
is a special case of tolerant testing.
To see the connection to problem (b), notice
that any estimator $\hat{d}_n$ of the divergence
functional $d(P,P_0)$   induces
a tolerant test,
which rejects the null hypothesis
when $\hat d_n$ takes on values
which are  appreciably larger than
$\epsilon_0$. In particular,
if the minimax
rate of
{\it estimating} $d(P,P_0)$ is $\delta_n$, then
there exists  a
tolerant {\it test} with nontrivial
power whenever
\begin{equation}
\epsilon_1 - \epsilon_0 \geq C \cdot \delta_n,
\end{equation}
for a sufficiently large constant $C > 0$.
This reduction is described in further
detail in  Section \ref{sec:functional_estimation_rates}, though
is well-known, and forms
a common strategy for certifying
lower bounds on the minimax
estimation risk~\citep{tsybakovIntroductionNonparametricEstimation2009}.
In our context, this reduction
is typically sharp when $\epsilon_0$ takes
on sufficiently large values,
thus  one can
generally interpret
the critical
separation
as interpolating between the minimax testing rate
of problem (a) when $\epsilon_0$ is small, and the minimax
estimation rate (b) when $\epsilon_0$ is large.
In between these two regimes, there
sometimes appears
an {\it interpolation} regime,
which is unique to tolerant
testing. In what follows,
we provide a more detailed
view into these phenomena, before
drawing further comparisons to prior work.

\paragraph*{Outline}  Section \ref{sec:overview} provides an overview of our main results, while  Section \ref{sec:minimax_framework} introduces the minimax framework for tolerant testing. In  Section \ref{sec:gaussian_sequece}, we characterize the critical separation in the Gaussian sequence model. In  Sections \ref{sec:gaussian_white_noise} and \ref{sec:testing_densities}, we characterize the critical separation for the Gaussian white noise model and densities separated in the $L_1$ norm. Finally, in  Section \ref{sec:null_approximation}, we comment on the application of tolerant testing for handling systematic uncertainties in high-energy physics.

\paragraph*{Notation} We write $a_n \lesssim b_n$ if there exists a positive constant $C$ such that $a_n \leq C \cdot b_n$ for all $n$ large enough. Analogously, $a_n \asymp b_n$ denotes that $a_n \lesssim b_n$ and $b_n \lesssim a_n$. Furthermore, given a distribution $\pi$, we use $m_l(\pi)=E_{p\sim \pi}[p^l]$ to denote its $l$-th moment. Additionally, let $\delta_A(a)$ equal $1$ if $a \in A$, and $0$ otherwise. Similarly, we use the following notation $I(\text{condition})$ to denote the indicator function that evaluates to one whenever the condition is true and returns zero otherwise. Whenever clear from context, we omit stating the random variable in expectations and variances: $E_{P}[T] =  E_{X\sim P}[T(X)]$ and $V_{P}[T] = V_{X\sim P}[T(X)]$.

\subsection{Overview of main results}\label{sec:overview}
We study the minimax tolerant testing problem
in the following three settings:
\begin{enumerate}
\item[(i)] The Gaussian sequence model, under the $\ell_p$-metrics ( Section \ref{sec:gaussian_sequece}).
\item[(ii)]
A smooth
Gaussian white noise model,
under the $L_1$ metric ( Section \ref{sec:gaussian_white_noise}).
\item[(iii)] A smooth density testing model,
under the $L_1$ metric ( Section \ref{sec:testing_densities}).
\end{enumerate}
We derive minimax upper bounds across
these various problems, and
we complement them
with matching minimax lower bounds in many
cases. Along the way, we
  identify the suboptimality
of certain classical goodness-of-fit tests
for these  problems, and provide
alternative tests which are both
practical and minimax-optimal.

Let us illustrate our results
in the simplest case of the Gaussian sequence
model (i), which is a canonical
model for the study of minimax
hypothesis testing problems, and serves
as a convenient benchmark for more general nonparametric
regression problems~\citep{ingsterNonparametricGoodnessofFitTesting2003,johnstone2019}. In this setting, one
is given a $d$-dimensional, homoscedastic  observation of the form
\begin{equation}\label{eq:gsn_seq_model_expo}
X \sim \mathcal{N}(v,I_d/n),
\end{equation}
for some $v \in \mathbb{R}^d$
and $n,d \geq 1$.
Let us first consider the $\ell_1$-tolerant
testing problem:
\begin{equation}\label{eq:gsn_seq_expo}
H_0: \norm{v}_1 \leq \epsilon_0,
\vs H_1: \norm{v}_1 \geq \epsilon_1.
\end{equation}
It has been known since the work
of~\cite{ingsterAsymptoticallyMinimaxHypothesisI1993} that the classical $\chi^2$-test
is minimax optimal for this problem
when $\epsilon_0=0$, with corresponding
critical radius scaling as $d^{3/4}/\sqrt n$.
Surprisingly, however,
we will show that the $\chi^2$-test
is generally suboptimal
for the tolerant testing problem,
and we will instead show that
a different test achieves the following
minimax hypothesis testing rate.

\begin{theorem}[Informal version of Theorem \ref{thm:gaussian_testing_l1} in  Section \ref{sec:simple_null_rates}]\label{thm:l1_informal}
Given $\epsilon_0 \geq 0$, let $\Tolerance=\sqrt{n}\epsilon_0$.
Then, the
hypotheses~\eqref{eq:gsn_seq_expo}  can be tested with nontrivial power if and only if
\\[0.02in]
$$\epsilon_1 - \epsilon_0 \gtrsim
\frac{1}{\sqrt n}\cdot
\left\{
\begin{array}{lll}
\displaystyle d^{3/4} ,
& \textif 0\leq \Tolerance \lesssim \sqrt d,
 &~~ \text{(Free tolerance regime)}\\[2ex]
\displaystyle d^{1/2} \cdot \Tolerance^{1/2},
& \textif \sqrt d  \lesssim  \Tolerance \lesssim d,
& ~~\text{(Interpolation regime)} \\[2ex]
\displaystyle d ,
& \textif \Tolerance\asymp d,
&~~ \text{(Functional estimation regime)}
\end{array}
\right.
$$  where we hide polylogarithmic factors in $d$.
\end{theorem}

A striking feature of the critical separation under the $\ell_1$-norm  is the existence of
three regimes.
When $\epsilon_0$
takes on small values, the critical
separation remains constant,
and scales as $d^{3/4} /\sqrt n$.
Once again, this matches
the well-known critical separation   for
goodness-of-fit testing under the $\ell_1$ norm~\citep{ingsterAsymptoticallyMinimaxHypothesisI1993}. We refer
to this regime as the {\it free tolerance
regime},  since the goodness-of-fit
rate persists  even as $\epsilon_0$ increases.
Let us emphasize that the existence of
this regime was already known
from the work of~\cite{ingsterTestingHypothesisWhich2001}.
When $\epsilon_0$ crosses the
critical rate $\sqrt {d / n}$,
we enter the {\it interpolation regime},
which exhibits a new rate that
is specific to tolerant testing,
and decays polynomially in~$\epsilon_0$.
Finally, a third
regime occurs when $\epsilon_0 \asymp d / \sqrt n$,
in which case, for the first time,
 the critical separation
$\epsilon_1-\epsilon_0$ coincides with $\epsilon_0$.
The convergence rate of this
regime coincides with the
minimax
rate of estimating the functional $\|v\|_1$, up to polylogarithmic
factors~\citep{caiTestingCompositeHypotheses2011a}.  Although
we do not characterize the $\ell_1$ tolerant
testing rate for larger
values of $\epsilon_0$,
we shall do so for
some other norms discussed below.

The existence of a free tolerance regime
provides a means of quantifying
  the robustness
of goodness-of-fit tests
to null misspecification.
Indeed, one can heuristically
define the {\it tolerance}
of a goodness-of-fit
test as the largest value of $\epsilon_0$
for which the
critical separation of the test
stays constant---namely, the largest
$\epsilon_0$ for which the test remains
powerful when $\epsilon_1$ equals
the goodness-of-fit minimax critical
separation.
In the Gaussian sequence model, Theorem~\ref{thm:l1_informal} shows that the maximal achievable tolerance is $d^{1/2}/\sqrt n$.
Perhaps surprisingly, this tolerance is
not achieved by certain well-known
goodness-of-fit tests:
In Section~\ref{sec:gaussian_sequece},
we will see that the test statistics
$$\chi^2 = \sum_{j=1}^d X_j^2, ~~~\text{and}~~~
 T = \sum_{j=1}^d |X_j|,$$
are both minimax-optimal for
goodness-of-fit testing
(i.e. for problem~\eqref{eq:gsn_seq_expo} with $\epsilon_0=0$), however the
  $\chi^2$-test
merely has tolerance $d^{1/4}/\sqrt n$,
whereas the plugin test $T$
has optimal tolerance $d^{1/2}/\sqrt n$.
Despite its popularity,
the $\chi^2$-test thus has suboptimal
tolerance, and can be improved
by   a simple alternative test
without sacrificing
minimax optimality, at least
when measuring deviations
under the $\ell_1$ norm.

Theorem~\ref{thm:l1_informal} builds upon recent work by \cite{canonnePriceToleranceDistribution2021}, which established closely-related phenomena for tolerant testing in the setting of discrete distributions. Concretely, given $n$ observations from a discrete probability
distribution $P$ supported on $[d]=\{1,\dots,d\}$, and a reference distribution $P_0$,
\cite{canonnePriceToleranceDistribution2021} study the tolerant testing problem
\begin{equation}\label{eq:tolerant_testing_multinomial_l1}
H_0: V(P,P_0) \leq \epsilon_0 \vs H_1 : V(P,P_0) \geq \epsilon_1,
\end{equation}
where $V$ denotes the total variation distance.
They prove that for
$ d\lesssim n $, the hypotheses \eqref{eq:tolerant_testing_multinomial_l1} can be tested with non-trivial power, up to polylogarithmic factors, if and only if
\begin{equation}\label{eq:critical_separation_multinomial_l1}
\epsilon_1-\epsilon_0 \gtrsim
\frac 1 {\sqrt n} \cdot
\left\{
\begin{array}{lll}
 d^{1/4},
& \textif 0\leq \lambda \lesssim  1,
&~~\text{(\textit{Free tolerance regime})} \\[2ex]
 d^{1/4} \cdot \lambda^{1/2},
& \textif 1 \lesssim \lambda \lesssim \sqrt d,
&~~\text{(\textit{Interpolation regime})} \\[2ex]
\displaystyle d^{1/2},
& \textif \lambda \asymp \sqrt{{d}}.
&~~\text{(\textit{Functional estimation regime})}
\end{array}
\right.
\end{equation}
Here again, we can observe that the critical separation interpolates between the goodness-of-fit rate \citep{paninskiCoincidenceBasedTestUniformity2008}, which persists over
a free tolerance region, and the risk of estimating the total variation distance between the distributions \citep{jiaoMinimaxEstimationL_12018a}.
We will later see that this same pattern
also occurs in the
infinite-dimensional models~(ii) and~(iii).

Nevertheless, the interpolation regime does not always need to exist. To illustrate this point, we now turn to a second
setting considered in this paper,
consisting of the hypotheses
\begin{equation}\label{eq:gsn_seq_expo_even_p}
H_0: \norm{v}_p \leq \epsilon_0,
\vs H_1: \norm{v}_p \geq \epsilon_1,
\quad \text{with } p  \text{ an even integer},
\end{equation}
where we continue to work
under the Gaussian sequence model~\eqref{eq:gsn_seq_model_expo}.
Unlike the $\ell_1$ problem
studied previously,
the $\ell_p$ norms arising in
problem~\eqref{eq:gsn_seq_expo_even_p}
are smooth.
It is well-known that
functional estimation
and goodness-of-fit
testing rates tend to coincide
for smooth functionals~\citep{gineMathematicalFoundationsInfinitedimensional2016};
for example, when $p=2$,
it is is known that both the goodness-of-fit testing rate and the functional estimation rate are of the same order $d^{1/4}/\sqrt{n}$, in the Gaussian
sequence model.
It is therefore natural to
expect the critical separation to remain constant for all $\epsilon_0\leq d^{1/4}/\sqrt{n}$, absent any interpolation
regime. The following
result proves that
this is indeed the case, and
 additionally
characterizes the
tolerant testing rate when $\epsilon_0$
exceeds the functional estimation rate.

\begin{theorem}[Informal version of Lemma \ref{lemma:upper_bound_lp_even} in  Section \ref{sec:smooth_lp_norm}]
\label{thm:smooth_informal}
Given $\epsilon_0 \geq 0$, let $\Tolerance=\sqrt{n}\epsilon_0$.
Then, the
hypotheses~\eqref{eq:gsn_seq_expo_even_p}
can be tested with nontrivial power if\begin{equation}
\epsilon_1 - \epsilon_0 \gtrsim
\frac{1}{\sqrt{n}} \cdot
\left\{
\begin{array}{lll}
\displaystyle
d^{1/2p},
& \textif 0\leq \Tolerance \lesssim d^{1/2p},
&~~\text{(Free tolerance regime)}
\\[2ex]
\displaystyle
d^{1/2}\cdot \lambda^{1-p}, &\textif d^{1/2p} \lesssim \Tolerance \lesssim d^{1/p},
\\[2ex]
\displaystyle
d^{1/p - 1/2}, &\textif \Tolerance \asymp d^{1/p}.
\end{array}
\right.
\end{equation}
\end{theorem}

Theorem~\ref{thm:smooth_informal}
confirms the absence of an interpolation regime for tolerant
testing under the smooth $\ell_p$ norms,
and shows that the traditional
goodness-of-fit critical separation radius $d^{1/2p} / \sqrt n$~\citep{ingsterAsymptoticallyMinimaxHypothesisI1993} persists until $\epsilon_0$
coincides with this rate.
Once again, the existence of this free
tolerance regime was already known from the work of~\cite{ingsterTestingHypothesisWhich2001}.
On the other hand, when $\epsilon_0$
exceeds this rate, we find
that the critical separation begins to shrink,
indicating that the tolerant
testing problem becomes easier. Nevertheless, this phenomenon is still linked to functional estimation, which we explain in  Section \ref{sec:smooth_lp_norm}.

Since only moderate tolerance levels are relevant in practice—when testing provides a clear advantage over estimation—we focus on the case where the tolerance parameter $\epsilon_0$ lies between zero and the functional estimation rate, throughout what follows. Beyond this point, an estimation-based test becomes preferable.

\subsection{Related Work}

\begin{figure}
    \centering
    \includegraphics[width=\linewidth]{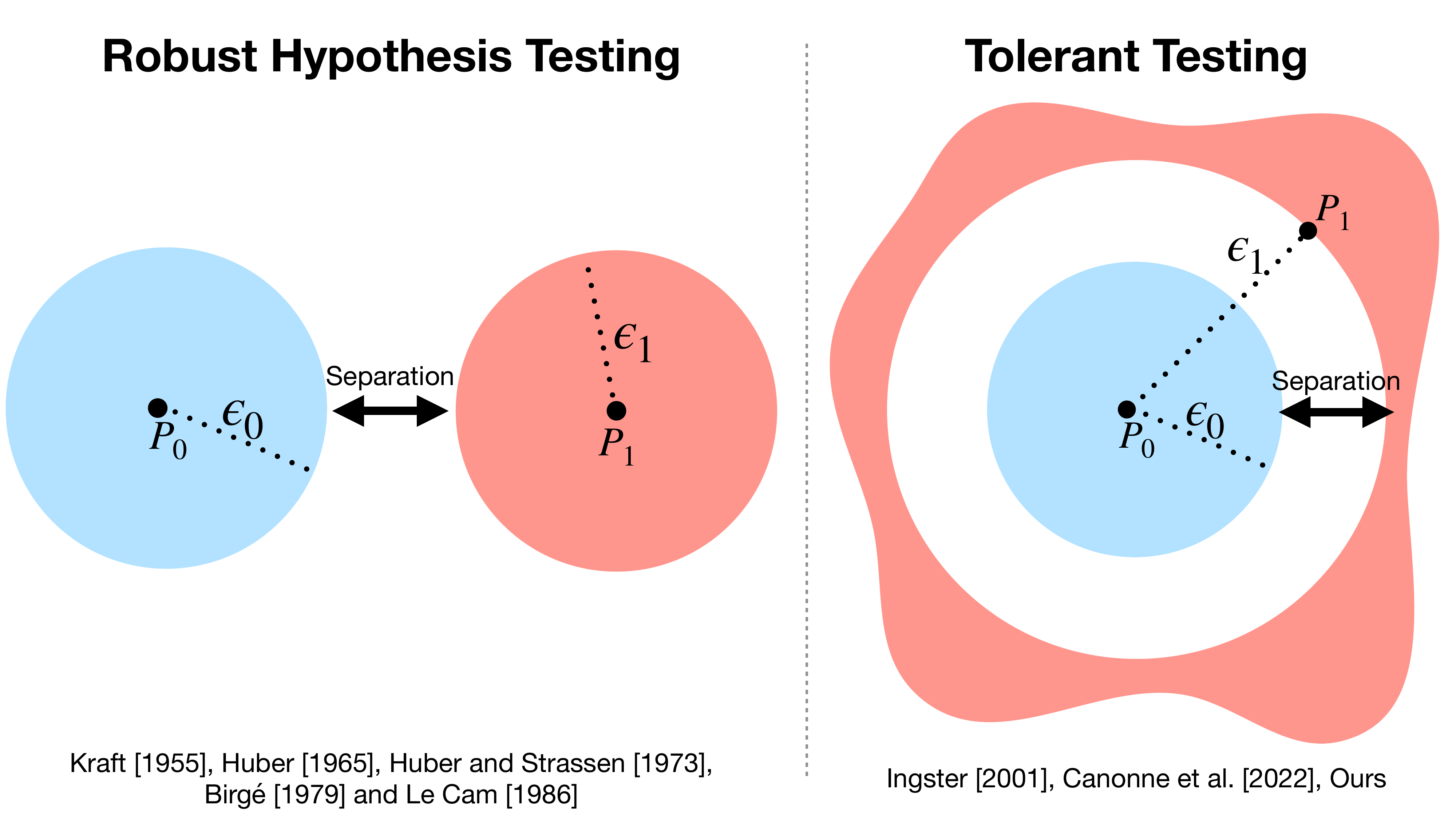}
    \caption{Comparison between robust hypothesis testing and tolerant testing.}
    \label{fig:comparison}
\end{figure}

The study of robust hypothesis testing within the minimax framework has a long history. \citet{LeCamAsymptoticMethodsStatistical1986} and \citet{kraftConditionsConsistencyUniform1955} studied the conditions for the distinguishability of disjoint convex sets of distributions, see the left panel of Figure \ref{fig:comparison}. \citet{huberRobustVersionProbability1965} derived minimax optimal tests when the distributions under the null and alternative hypotheses form disjoint balls in the total variation metric. \citet{huberMinimaxTestsNeymanPearson1973,riederLeastFavorablePairs1977,bednarskiBinaryExperimentsMinimax1982} extended these results to general disjoint sets dominated by capacities. Later works by \citet{birge1979theoreme,birgeRobustTestingIndependent1983,birgeApproximationDansEspaces1983,birgeDiscussionHypothesesTesting2015} and \citet{LeCamAsymptoticMethodsStatistical1986} developed analogous tests for Hellinger balls, with the aim of studying the convergence of estimators  \citep{lecamConvergenceEstimatesDimensionality1973}. In all cases, the disjointness of the sets ensures that the optimal testing procedure is a likelihood ratio test between a pair of \textit{closest} distributions lying on the boundaries of each set. Together, these developments provide a strong generalization of binary hypothesis testing. However, they do not apply to tolerant testing \eqref{eq:generic_tolerant_testing}, where distributions under the alternative hypothesis \textit{surround} the set of distributions under the null hypothesis \citep{ingsterNonparametricGoodnessofFitTesting2003}, see the right panel of Figure \ref{fig:comparison}.

To the best of our knowledge, the first study of minimax tolerant testing is due to \citet{ingsterTestingHypothesisWhich2001} for the Gaussian sequence and white noise models under general $\ell_p$ metrics, and by \cite{batuTestingThatDistributions2000} for discrete distributions under the $\ell_1$ distance. Both of these works focused on the case where $\epsilon_0$ is small. In the Gaussian sequence model, \citet{carpentierAdaptiveEstimationSparsity2019} subsequently
analyzed tolerant testing under
the $\ell_0$ divergence.
The discrete Multinomial
model~\eqref{eq:tolerant_testing_multinomial_l1} has received comparatively
more attention, including
further refinements
in the small-$\epsilon_0$ regime under the $\ell_1$ norm~\citep{paninskiCoincidenceBasedTestUniformity2008,valiantTestingSymmetricProperties2011,valiantAutomaticInequalityProver2014,diakonikolas2015}, mild tolerance under $\ell_1$ and $\ell_2$ \citep{batu2013,chanOptimalAlgorithmsTesting2014,acharyaOptimalTestingProperties2015a,diakonikolas2016new,daskalakis2018distribution,bhattacharyyaLearningMultivariateGaussians2025}, intermediate tolerance under $\ell_1$ \citep{canonnePriceToleranceDistribution2021}, and more general metrics \citep{chakrabortyExploringGapTolerant2022}.

As mentioned in the previous section, when $\epsilon_0$ is of the same order as the risk of
estimating $d(P,P_0)$,
testing becomes equally hard to estimating this divergence~\citep{parnasTolerantPropertyTesting2006}.
For smooth and non-smooth divergences, this estimation problem
has   been studied extensively,   in the discrete model~\citep{valiantEstimatingUnseenLog2011,hanMinimaxEstimationDiscrete2015,jiaoMinimaxEstimationL_12018a},
the Gaussian sequence model~\citep{caiTestingCompositeHypotheses2011a,koltchinskiiEfficientEstimationSmooth2021,koltchinskiiEstimationSmoothFunctionals2021}, and the Gaussian white noise model \citep{ibragimovProblemsNonparametricEstimation1987,lepskiEstimationLrnormRegression1999,hanEstimationL_Norms2020}.
In many of these works, lower bounds on the critical separation are established by creating pairs of distributions that share many moments but remain significantly different in terms of the divergence of interest \citep{wuMinimaxRatesEntropy2016}. Explicit constructions are a central focus of the work by \cite{ingsterTestingHypothesisWhich2001}, while \cite{canonnePriceToleranceDistribution2021} relies on implicit constructions that are based on the duality between moment-matching and polynomial approximation. There is also a wealth of work on minimax estimation for smooth divergences in the various model classes studied in this paper (e.g.~\cite{bickel1988,gine2008,kerkyacharian1996,robins2009,birge1995,laurent1996,tchetgen2008minimax}), for which minimax lower bounds are typically obtained through reductions to goodness-of-fit testing problems~\citep{ingsterAsymptoticallyMinimaxHypothesisI1993}. In the special case of linear functionals, we emphasize that the celebrated work of \citet{donohoGeometrizingRatesConvergence1991} established a tight link between estimation and tolerant testing.

As discussed in  Section \ref{sec:introduction}, one key motivation for studying the tolerant testing problem comes from scientific applications that require testing goodness-of-fit hypotheses under systematic or epistemic uncertainty. A common source of systematic uncertainty occurs when the null distribution $P_0$ is only approximately known due to simulation error. This setting is often referred to as \textit{likelihood-free hypothesis testing}, a problem recently investigated within the minimax framework~\citep{gerber2024likelihood,gerber2023minimax,jiaGaussianSequenceModel2025b}. A related line of research in \textit{imprecise probability} \citep{walley1991statistical} models epistemic uncertainty arising from a practitioner's partial knowledge. In this case, $P_0$ is assumed to belong to \textit{Credal set}, a convex set of distributions, that captures this lack of complete information~\citep{levi1980enterprise,sale2023volume}. Across these frameworks and related ones, tests based on the maximum mean discrepancy~\citep{gretton2006kernel} have become prominent for their strong empirical performance and theoretical tractability~\citep{gerber2023kernel,chauCredalTwoSampleTests2025,zhouKernelDistributionCloseness2025}.

Finally, we note that tolerant testing subsumes several other robustness frameworks in the literature. These include the semiparametric framework of~\cite{liuBuildingUsingSemiparametric2009}, which examines deviations from a parametric model through the Kullback–Leibler divergence; and the equivalence testing framework~\citep{wellekTestingStatisticalHypotheses2002,romanoOptimalTestingEquivalence2005}, which reverses the roles of the hypotheses in~\eqref{eq:generic_tolerant_testing}. Therefore, general results in tolerant testing provide unified tools for reasoning about these frameworks. In Appendix \ref{sec:equivalence_testing} of the supplementary material, we illustrate this connection by examining how the critical separation in tolerant testing relates to that in equivalence testing. 

\section{The minimax framework for tolerant testing}\label{sec:minimax_framework}

Our aim is to distinguish hypotheses \eqref{eq:generic_tolerant_testing} while controlling the probability of choosing the wrong hypothesis. In this section, we define the smallest distance between the hypotheses, called the critical separation, such that a test is able to reliably distinguish them.

Given a set of distributions $\Pdist$, define the set of distributions that belong to the null and alternative hypotheses: \begin{equation}
\Pnull = \left\{P \in \Pdist : d(P,P_0)\leq \epsilon_0\right\} \textand \Palt = \left\{P \in \Pdist : d(P,P_0)\geq \epsilon_1\right\}.
\end{equation} A test $\psi$ maps the sample space to $\{0,1\}$. It returns $0$ if it considers that the data supports the null hypothesis and $1$ otherwise. The type-I error corresponds to the probability of choosing the alternative hypothesis when the null hypothesis is true. Henceforth, let $\Psi$ denote the set of all tests that control the type-I error by $\alpha$, known as valid tests, \begin{equation}
\Psi(\Pnull) = \left\{\psi: \sup_{P \in \Pnull}P(\psi(X)=1) \leq \alpha \right\}
\end{equation} where $X=(X_1,\dots,X_n)$ is a vector of $n$ observations.  The type-II error is the probability of selecting the null when the alternative is true. The risk of a valid test is the maximum type-II error under the alternative: \begin{equation}\label{eq:test_risk_def}
R(\epsilon_0,\epsilon_1,\Pdist,\psi) = \sup_{P\in\Palt} P(\psi(X) = 0) \for \psi \in \Psi(\Pnull).
\end{equation} Throughout the paper, we call a test powerful if its risk is bounded by $\beta \in (0,1-\alpha)$. Finally, the minimax risk quantifies the best possible performance over all valid tests: \begin{equation}\label{eq:minimax_risk}
R_*(\epsilon_0,\epsilon_1,\Pdist) = \inf_{\psi \in \Psi(\Pnull)}R(\epsilon_0,\epsilon_1,\Pdist,\psi).
\end{equation} We say that a valid test $\psi$ is optimal if its risk $R(\epsilon_0,\epsilon_1,\psi,\Pdist)$ equals the minimax risk $R_*(\epsilon_0,\epsilon_1,\Pdist)$ up to constants and logarithmic factors. For a fixed sample size $n$, a powerful valid test may not exist when the gap between $\epsilon_0$ and $\epsilon_1$ is too small. Thus, we define the critical separation as the smallest distance that admits a powerful valid test: \begin{equation}\label{eq:critical_separation}
\epsilon_1^*(\epsilon_0,\Pdist)-\epsilon_0 = \inf\{\epsilon_1-\epsilon_0: \epsilon_1 \geq \epsilon_0 \textand R_*(\epsilon_0,\epsilon_1,\Pdist) < \beta\}.
\end{equation} The critical separation is the detection threshold for the problem class $\Pdist$ at sample size $n$: below this threshold, no valid test can reliably distinguish the hypotheses in \eqref{eq:generic_tolerant_testing}.

\section{The Gaussian sequence model}\label{sec:gaussian_sequece}

As a starting point, we study the $d$-dimensional isotropic Gaussian sequence model, since it provides the right intuition for later models studied in the paper: \begin{equation}\label{eq:gaussian_model}
X \sim P_v \where P_v = \Normal\left(v, \sigma^2 \cdot I_d\right) \comma \sigma \in (0,1] \textand \theta \in \mathbb{R}^d.
\end{equation}
Given a real number $p\geq 1$ and a tolerance parameter
$\epsilon_0 > 0$, we consider the problem of testing whether the mean of $P_v$ lies near the origin under the $\ell_p$ norm:
\begin{align}\label{eq:gaussian_testing_lp}
H_0: \norm{v}_p \leq \epsilon_0 \vs H_1: \norm{v}_p \geq \epsilon_1\period
\end{align} Let $\GS = \left\{P_v : v \in \Rd\right\}$ denote the set of Gaussian distributions \eqref{eq:gaussian_model}. Our goal is to sharply characterize the critical separation $s_1(\epsilon_0,\mathrm{GS})$ for this problem. We begin by discussing two special cases---where $p=1$ or $p=2$---which capture many of the essential phenomena of our problem.

\subsection{Tolerant Testing under the \texorpdfstring{$\ell_1$}{l1} norm}\label{sec:simple_null_rates}

Consider testing the simple null hypothesis problem, $\epsilon_0=0$ in \eqref{eq:gaussian_testing_lp} with $p=1$: \begin{equation}\label{eq:simple_null_testing_l1}
H_0: v=0 \vs H_1: \norm{v}_1 \geq \epsilon_1.
\end{equation} A minimax test for \eqref{eq:simple_null_testing_l1} is the well-known Pearson chi-squared test: \begin{equation}
\psi_2(X,0) = I\left(T_2(X) \geq q_{1-\alpha}\left(T_2,P_0\right)\right) \where T_2(X)=\norm{X}_2^2 - E_{P_0}\norm{X}_2^2,
\end{equation} where $q_{1-\alpha}\left(T_2,P_0\right)$ is the $1-\alpha$ quantile of $T_2(X)$ for $X\sim P_0$. The test is known \citep{ingsterAsymptoticallyMinimaxHypothesisI1993} to consistently distinguish the hypotheses \eqref{eq:simple_null_testing_l1} as soon as $\epsilon_1 \gtrsim \sigma \cdot d^{3/4}$. We refer to this rate as the simple null hypothesis testing rate. To understand its significance, note that the testing rate is faster than the rate at which the functional $\norm{v}_1$ can be estimated, which is $\sigma \cdot d/\sqrt{\log d}$ and called the functional estimation rate \citep{caiTestingCompositeHypotheses2011a}. If we choose
the natural scaling $\sigma = n^{-1/2}$, this means that for $\sqrt{n} \lesssim d \lesssim n^{2/3}$, we can test the size of $\norm{v}_1$ even though we cannot reliably estimate it.

Given that the chi-squared test is valid, it cannot reject distributions that are
statistically indistinguishable
from $P_0$. Thus, one would expect that the null hypothesis can be \textit{enlarged}  without   affecting the
test's power. By slightly modifying the decision threshold, it can be shown that the chi-squared test maintains the same power guarantee (up to constants) for larger null hypotheses. The proof is deferred to Appendix \ref{sec:failure_l2_test} of the supplementary material. 

\begin{lemma}[Suboptimality of the chi-squared test]\label{lemma:chi_squared_test_under_l1}
For hypotheses \eqref{eq:gaussian_testing_lp} with $p=1$, there exists a decision threshold $t_*$, such that the chi-squared test $\psi(X)=I(T_2(X)\geq t_*)$ is valid. Futhermore, there exists positive constants $C_1$ and $C_2$ such that, the test is powerful whenever \begin{equation}\label{eq:chi2_free_tolerance}
\epsilon_1 - \epsilon_0 \gtrsim \sigma \cdot d^{3/4} \quad \for \epsilon_0 \leq C_1\cdot \sigma \cdot d^{1/4}.
\end{equation} However, for $\epsilon_1 - \epsilon_0 \asymp \sigma \cdot d^{3/4}$ and $C_2\cdot \sigma \cdot d^{1/4} \leq \epsilon_0  \lesssim \sigma \cdot d^{3/4}$, there exists no decision threshold $t$ such that the chi-squared test $\psi(X)=I(T_2(X)\geq t)$ is both valid and powerful.
\end{lemma}

In other words, the chi-squared test tolerates some deviations from $P_0$ for free, in the sense that it maintains the same power guarantees as in goodness-of-fit testing, although the null has size $\epsilon_0\asymp d^{1/4}\cdot \sigma$. However, the second part  Lemma \ref{lemma:chi_squared_test_under_l1} states that as soon as $\epsilon_0$ is larger than $d^{1/4}\cdot \sigma$, this power guarantee is lost.

It is natural to wonder whether the tolerance stated in equation~\eqref{eq:chi2_free_tolerance} is optimal: Is it possible to tolerate larger deviations from $P_0$ while maintaining the same power as when testing a simple null? In a follow-up work, \citet{ingsterTestingHypothesisWhich2001} proved that the
following simple plug-in test achieves a better guarantee:
\begin{equation}\label{eq:plugin_test_l1}
\psi_1(X,\epsilon_0) = I\left(T_1(X) \geq \sup_{\norm{v}_1 \leq \epsilon_0} q_{1-\alpha}\left(T_1,P_v\right) \right) \where T_1(X)=\norm{X}_1-E_{P_0}\norm{X}_1.
\end{equation} The plug-in test is usually not employed for testing $v=0$ due to the $\ell_1$ norm being non-smooth, i.e., non-differentiable, which complicates studying its asymptotic distribution. However, the $\ell_1$ norm naturally provides more robustness than the plug-in test. This can be anticipated
from the fact that $V_{P_v}[T_1]\asymp \sigma^{2}\cdot d$ while $V_{P_v}[T_2]\asymp \sigma^{4}\cdot d + \sigma^{2}\cdot \norm{v}_2^2$. Thus, large deviations from $P_0$, under the null hypothesis, affect the variance of $T_2$ more than $T_1$, which is detrimental to the performance of the chi-squared test.

\begin{proposition}[\citet{ingsterTestingHypothesisWhich2001}, Optimal tolerant testing for small deviations]\label{prop:free_tolerance_l1} For hypotheses \eqref{eq:gaussian_testing_lp} with $p=1$, the plug-in test \eqref{eq:plugin_test_l1} is powerful whenever  \begin{equation}\label{eq:free_tolerance_l1}
\epsilon_1 - \epsilon_0 \gtrsim \sigma \cdot d^{3/4} \quad \for \epsilon_0 \lesssim \sigma \cdot \sqrt{d}. \qquad \text{(Free tolerance regime)}
\end{equation} Furthermore, the above result is optimal. That is, for $\epsilon_1 - \epsilon_0 \lesssim \sigma \cdot d^{3/4}$ and $\epsilon_0 \lesssim \sigma \cdot \sqrt{d}$, there is no test that is both valid and powerful.
\end{proposition}


In what follows, we prove that the plug-in test remains optimal even when we allow for larger values of $\epsilon_0$. Most interestingly, as soon as we leave the free tolerance regime, the performance of the plug-in test depends on both the size of the null hypothesis, dictated by the amount of tolerance allowed, and the cost of estimating the underlying $\ell_1$ norm.

\begin{theorem}\label{thm:gaussian_testing_l1} Given $\epsilon_0 \geq 0$, let $\Tolerance=\epsilon_0/\sigma$.
Then, for hypotheses \eqref{eq:gaussian_testing_lp} with $p=1$, the critical separation is characterized up to logarithmic factors by \begin{equation}
\epsilon_1^*(\epsilon_0,\GS)-\epsilon_0 \asymp
\sigma \cdot \left\{
\begin{array}{lll}
\displaystyle d^{3/4} ,
& \textif 0\leq \Tolerance \lesssim \sqrt d,
 &~~\text{(Free tolerance regime)}\\[2ex]
\displaystyle d^{1/2}\cdot \lambda^{1/2},
& \textif d^{1/2} \lesssim \lambda \lesssim d,
&~~\text{(Interpolation regime)}\\[2ex]
\displaystyle d,
& \textif \lambda \asymp d.
&~~\text{(Functional estimation regime)}
\end{array}
\right.\end{equation}
Furthermore, the plug-in test \eqref{eq:plugin_test_l1} attains it.
\end{theorem}

The free tolerance regime in Theorem \ref{thm:gaussian_testing_l1} simply restates Proposition \ref{prop:free_tolerance_l1}.  In the following two subsections, we discuss upper bounds on the critical separation for the interpolation and functional estimation regimes of  Theorem \ref{thm:gaussian_testing_l1}, beginning with the second. Then, in  Section \ref{sec:lower_bounds_l1}, we provide lower bounds on the critical separation for both regimes. Together, these sections provide a proof of Theorem \ref{thm:gaussian_testing_l1}.

\subsubsection{Reduction to functional estimation under the \texorpdfstring{$\ell_1$}{l1} norm}\label{sec:functional_estimation_rates}

Using an estimator to construct a test often leads to suboptimal performance for simple null hypotheses. However, estimation becomes optimal for tolerant testing when the required tolerance is larger than the estimator's accuracy. The following lemma formalizes this intuition. It shows that any estimator of $\norm{v}_1$ might be used to construct a tolerant test insofar as a bound of its accuracy is available \citep{parnasTolerantPropertyTesting2006}. Similar statements appear in Proposition 2.17 of \citet{ingsterNonparametricGoodnessofFitTesting2003} and Proposition 6.2.2 of \citet{gineMathematicalFoundationsInfinitedimensional2016}; the proof is deferred to Appendix \ref{sec:general_upperbounds} of the supplementary material. 

\begin{lemma}[Testing by learning]\label{lemma:testing_by_learning_1} For any estimator $T$ such that \begin{equation}\label{eq:estimation_rate}
\sup_{v\in \R^d}E_{X\sim P_v}[\ T(X)-\norm{v}_1\ ]^2 \leq \phi,
\end{equation} the following estimation-based test is valid
for hypotheses~\eqref{eq:gaussian_testing_lp} with $p=1$
\begin{equation}\label{eq:estimation_based_test}
\psi(X) = I\left(T(X) > \epsilon_0 + \sqrt{\phi/\alpha}\right).
\end{equation} Furthermore, it is powerful whenever $\epsilon_1 - \epsilon_0 \geq C_\alpha\cdot \sqrt{\phi}$ where $C_\alpha = \alpha^{-1/2} + \beta^{-1/2}$.
\end{lemma}

Intuitively, the above result states that for tolerant testing, we cannot do worse than estimating the underlying $\ell_1$ norm. For the plug-in statistic \eqref{eq:plugin_test_l1}, the bias and variance satisfy
\begin{equation}\label{eq:bias_variance_l1}
\forall v \in \R^d \quad |E_{P_v}[T_1]-\norm{v}_1|\leq \norm{v}_1 \wedge \mu_1 \cdot \sigma \cdot d \textand   V_{P_v}\left[T_1\right]\leq \sigma^2 \cdot d,
\end{equation} where $\mu_1=\sqrt{2/\pi}$. Hence \eqref{eq:estimation_rate} holds with $\phi = (\mu_1^2+1) \cdot \sigma^2d^2$, since the bias dominates the variance when $v$ is unbounded. Consequently, the test \eqref{eq:estimation_based_test} is valid and powerful whenever:
\begin{equation}\label{eq:maximum_tolerance_l1}
\epsilon_1 - \epsilon_0 \gtrsim \sigma \cdot d\quad \for \epsilon_0 \geq 0.
\end{equation} Proposition \ref{prop:free_tolerance_l1} proves that the above result is suboptimal when the tolerance parameter is small. However, we should expect it to be optimal whenever the distance between the hypotheses is larger than the approximation error of estimating the $\ell_1$ norm. In  Section \ref{sec:functional_estimation_lb}, we argue that this is the case when $\epsilon_0 \asymp \sigma \cdot d$. We can foresee that result by noting that at $\epsilon_0 = C_\alpha \cdot \sqrt{\mu_1^2+1} \sigma\cdot d$, \eqref{eq:estimation_based_test} consistently distinguishes the hypotheses:
$H_0: \norm{v}_1 \leq \epsilon_0$ versus $H_1: \norm{v}_1 \geq 2\epsilon_0$, which we cannot expect to improve rate-wise.

Finally, we note that it is possible to improve over the plug-in estimator by approximating the absolute value function with a polynomial. Based on that strategy, \citet{caiTestingCompositeHypotheses2011a} propose an estimator that guarantees \eqref{eq:estimation_rate} with $\phi=\sigma^2 \frac{d^2}{\log d}$. Thus, \eqref{eq:maximum_tolerance_l1} can be improved by a polylogarithmic factor.

\subsubsection{Interpolation regime under the \texorpdfstring{$\ell_1$}{l1} norm}\label{sec:interpolation_rates}

From the discussion in the previous section, we would expect that being able to consistently test $H_0: \norm{v}_1\leq \epsilon_0$ versus $H_1: \norm{v}_1\geq \epsilon_1$ for $\epsilon_1=2\epsilon_0$ implies that we can estimate $\norm{v}_1$ up to error $\epsilon_0$. This is indeed the case for $\epsilon_0 \asymp d\cdot \sigma$. However, we know that this cannot be the case for $\epsilon_0 \asymp d^{3/4}\cdot \sigma$, since \citet{caiTestingCompositeHypotheses2011a} proved that it is not possible to estimate $\norm{v}_1$ faster than $d \cdot \sigma$. Thus, whenever $\epsilon_1 \asymp d^{3/4} \cdot \sigma$, it must be that $\epsilon_0$ is smaller. Indeed, in  Section \ref{sec:simple_null_rates}, we show that $\epsilon_0$ can be at most $d^{1/2}\cdot \sigma$. Consequently, for $\epsilon_0$ between $d^{1/2}\cdot \sigma$ and $d\cdot \sigma$, we expect $\epsilon_1$ to interpolate between $d^{3/4} \cdot \sigma$ and $d\cdot \sigma$, i.e. between the goodness-of-fit testing and functional estimation rates.

\begin{lemma} Let $\lambda=\epsilon_0/\sigma$. For hypotheses \eqref{eq:gaussian_testing_lp} with $p=1$, the plug-in test \eqref{eq:plugin_test_l1} is powerful whenever \begin{equation}\label{eq:interpolation_regime_l1}
\epsilon_1 - \epsilon_0  \gtrsim \sigma \cdot \sqrt{d\cdot  \lambda}  \quad \for \sqrt{d} \lesssim \lambda \lesssim d.
\end{equation}\end{lemma} The proof is deferred to Appendix \ref{sec:upper_bound_lp_odd_less_2} of the supplementary material. Note that the upper bound on critical separation in this regime, called the interpolation regime, depends on the functional estimation rate and the required tolerance. This mirrors the aforementioned findings of \cite{canonnePriceToleranceDistribution2021} in the Multinomial model under the $\ell_1$ metric. 

Looking at \eqref{eq:bias_variance_l1}, note that in the free tolerance regime ($\epsilon_0 \lesssim \sigma \cdot \sqrt{d}$), the variance of $T_1$ dominates its bias. In the functional estimation regime ($\epsilon_0 \asymp \sigma \cdot d$), the bias is of the same order as the functional estimation rate $\sigma \cdot d$. However, the interpolation regime arises due to the bias dominating the variance but remaining smaller than the functional estimation rate,
thus allowing us to test the null hypothesis faster than we can estimate the $\ell_1$ functional.

Intuitively, to derive \eqref{eq:interpolation_regime_l1}, consider the case where $\epsilon_0\gg \sqrt{d}\cdot \sigma$ and $\epsilon_0 \leq \epsilon_1 \ll d\cdot \sigma$. In this intermediate regime, we can establish that $\norm{v}_1^2 \cdot (d\cdot \sigma)^{-1} \lesssim E_{P_v}[T_1] \leq \norm{v}_1
$. Since the variance of the statistic is negligible relative to its mean under both hypotheses, the minimum separation required for the plug-in test \eqref{eq:plugin_test_l1} to distinguish them is driven by the largest mean under the null hypothesis and the smallest mean under the alternative hypothesis. Therefore, we must require that $\epsilon_1^2 \cdot (d\cdot \sigma)^{-1} \gtrsim \epsilon_0$, which leads to~\eqref{eq:interpolation_regime_l1}.

\subsection{Lower-bounds via moment-matching distributions}\label{sec:lower_bounds_l1}

To demonstrate the sharpness of the upper bounds on the critical separation in  Sections \ref{sec:simple_null_rates},\ref{sec:functional_estimation_rates} and \ref{sec:interpolation_rates}, we next
derive minimax lower bounds,
showing that the plug-in test
is
optimal up to logarithmic factors. In  Sections \ref{sec:simple_null_lb} and \ref{sec:functional_estimation_lb}, we derive lower bounds for the free and functional estimation regimes by constructing mixtures under the null and alternative that no test can reliably distinguish. In  Section \ref{sec:interpolation_lb}, we adapt this approach to establish lower bounds in the interpolation regime.

\subsubsection{Lower bound for the free tolerance regime}\label{sec:simple_null_lb}

We introduce a variant of Le Cam's two-point argument \citep{tsybakovIntroductionNonparametricEstimation2009}, and match the critical separation rate in the free tolerance regime \eqref{eq:free_tolerance_l1}. Henceforth, given a distribution $\pi$, we let $P_\pi$ be the corresponding mixture distribution: $P_\pi(B) = \int P_v(B) \ d\pi(v)$. Let $\pi_0^d$ and $\pi_1^d$ be corresponding product measures, with marginals $\pi_0$ and $\pi_1$, supported on the hypothesis sets: \begin{equation}\label{eq:hypotheses_sets}
V_0 = \{v : \norm{v}_1\leq \epsilon_0\} \textand V_1 = \{v : \norm{v}_1\geq \epsilon_1\}.
\end{equation} If the total variation distance between the induced mixtures $P_{\pi_0}$ and $P_{\pi_1}$ is small, then no test can reliably distinguish them. Therefore, the separation between the hypothesis sets yields a lower bound on the critical separation.

\begin{lemma}\label{lemma:exact_support_lb}
Consider distributions $\pi_0^d$ and $\pi_1^d$ supported on \eqref{eq:hypotheses_sets}, i.e. $\pi_{0}^d(V_0)=\pi_{1}^d(V_1) = 1$. If the total variation between their induced mixtures satisfies $V\left(P_{\pi_0}^d, P_{\pi_1}^d\right)\leq C_\alpha$ where $C_\alpha = 1-(\alpha+\beta)$. Then, the critical separation is lower-bounded by: \begin{equation}
\epsilon_1^*(\epsilon_0,\GS)-\epsilon_0 \geq  \epsilon_1 - \epsilon_0   \period
\end{equation}
\end{lemma}

To construct a lower-bound for the free tolerance regime \eqref{eq:free_tolerance_l1}, consider the following symmetric mixing distributions: $
\pi_0 = \delta_0$ and $\pi_1 = \frac{1}{2}\left(\delta_{\epsilon/d}+\delta_{-\epsilon/ d}\right)$ where $\epsilon = C \sigma\cdot d^{3/4}$ and $C=\left[2\log\left(1+(C_\alpha)^2\right)\right]^{1/4}$. These distributions share their first moment, and their separation is given by $\epsilon$ since \begin{equation}
\norm{v}_1 \overset{a.s.}{=} 0 \text{ under } \pi_0\quad   \textand \quad \norm{v}_1\overset{a.s.}{=}\epsilon \text{ under } \pi_1,
\end{equation} where the equalities hold almost surely. A simple computation shows that $V\left(P_{\pi_1}^d,P_{\pi_0}^d\right)\leq C_\alpha$, see Theorem C.10 in Appendix \ref{sec:chi2_bound} of the supplementary material. Thus, by Lemma \ref{lemma:exact_support_lb}, it follows that $
\epsilon_1^*(0,\GS) \geq \epsilon$. Furthermore, note that we can expand the null set $V_0$ until it nearly touches $V_1$, up to constants, without invalidating the construction. Consequently, we have proved that the critical separation is lower-bounded by \begin{equation}
\epsilon_1^*(\epsilon_0,\GS)-\epsilon_0 \geq \frac{C}{2}\sigma \cdot d^{3/4} \quad \for  0 \leq \epsilon_0 \leq \frac{C}{2}\sigma \cdot d^{3/4}.
\end{equation} Since $d^{3/4} \geq d^{1/2}$, we match \eqref{eq:free_tolerance_l1}, which shows that the plug-in test is optimal in the free tolerance region. A detailed version of this argument can be found in Appendix \ref{lb_p_less_than_2_small_e0} of the supplementary material.

\subsubsection{Lower bound for the functional estimation regime}\label{sec:functional_estimation_lb}

The construction in the previous section relies on the fact that the mixing distributions are centered and share the first moment. This structure facilitates bounding their total variation. The following lemma extends this idea by providing general conditions under which the total variation of Gaussian mixtures with moment-matching mixing distributions can be bounded, a technique that has been employed to great success across the estimation and testing literature \citep{lepskiEstimationLrnormRegression1999,ingsterTestingHypothesisWhich2001,caiTestingCompositeHypotheses2011a,wuMinimaxRatesEntropy2016,hanEstimationL_Norms2020}. A proof of the statement can be found in Appendix \ref{sec:chi2_bound} of the supplementary material. 

\begin{lemma}[\cite{wuMinimaxRatesEntropy2016}]
There exists positive constant $C$ depending only on $\alpha$ and $\beta$ such that for any $\delta\geq 0$ and $L\geq 1$ that satisfy \begin{equation}\label{eq:chi2_condition}
\delta^2 \leq C \cdot \sigma^2 \cdot \frac{L}{d^{1/(L+1)}}\comma
\end{equation} it holds that for any $\pi_0$ and $\pi_1$ centered distributions supported on $[-\delta,\delta]$ that share the first $L$ moments, the total variation distance between the corresponding mixture distributions is bounded by $
V\left(P_{\pi_1}^d,P_{\pi_0}^d\right)\leq C_\alpha=1-(\alpha+\beta)$.
\end{lemma}

Therefore, if we are given two mixing distribution $\pi_0$ and $\pi_1$ that match $L$ moments, are supported on $[-\delta,\delta]$ and their respective product measures are supported on the sets \begin{equation}\label{eq:support_sets}
V_0 = \{v : \norm{v}_1\leq   E_{v\sim\pi^d_0}\norm{v}_1\} \textand V_1 = \{v : \norm{v}_1\geq E_{v\sim\pi^d_1}\norm{v}_1\}\comma
\end{equation} then by Lemma \ref{lemma:exact_support_lb}, the critical separation is lower-bounded by \begin{equation}
\epsilon_1^*(\epsilon_0,\GS)-\epsilon_0 \geq E_{v\sim\pi^d_1}\norm{v}_1-E_{v\sim\pi^d_0}\norm{v}_1\period
\end{equation} Hence, a sharp lower bound can be obtained if we optimize over all mixing distributions satisfying the constraints. Theorem \ref{thm:lower_bound_function_estimation_l1} formalizes this notion by lower-bounding the critical separation by the maximum separation among all moment-matching distributions.

Let $M_p(L)$ denote the maximum functional separation achieved among distributions that share their first $L$ moments \begin{align}\label{eq:Mp}
M_p(L) = &\sup_{\pi_0,\pi_1} E_{v\sim \pi_1}|v|^p-E_{v\sim \pi_0}|v|^p\\
&\st  m_l(\pi_0)=m_l(\pi_1) \for 1\leq l\leq L
\end{align} where the supremum is taken over all probability measures supported on $[-1,1]$. Then, Theorem \ref{thm:lower_bound_function_estimation_l1} states that the critical separation can be lower-bounded by the moment-matching problem insofar as the indistinguishability condition \eqref{eq:chi2_condition} is satisfied. The proof is deferred to Appendix \ref{sec:moment_matching_lower_bounds}  of the supplementary material. 

\begin{theorem}\label{thm:lower_bound_function_estimation_l1}
Choose $\delta\geq 0$ and $L\geq1$ satisfying \eqref{eq:chi2_condition}. If $M_1(L)>0$, the critical separation for hypotheses \eqref{eq:gaussian_testing_lp} with $p=1$ is lower bounded by \begin{equation}
\epsilon_1^*(\epsilon_0,\GS)-\epsilon_0  \gtrsim \epsilon_0
\quad\for
\epsilon_0 \asymp \delta\cdot d \cdot M_1(L) \textand  d^{1/2} \gtrsim M^{-1}_1(L).
\end{equation}\end{theorem}

Moment-matching problems are challenging to solve directly. An effective approach is to study their duals, which is the well-researched best polynomial approximation problem.

\begin{lemma}[Duality of the moment-matching problem  \citep{wuMinimaxRatesEntropy2016}]\label{lemma:moment_matching_duality} Let $P_L$ be the set of all $L$-order polynomials supported on $[-1,1]$, then it holds that the moment-matching problem and the best polynomial approximation problem are equivalent up-to-constants: $M_p(L) = 2 \cdot A_p(L)$ where $A_p(L) =\inf_{f \in P_L}\sup_{|x|\leq 1}||x|^p-f(x)|.$
\end{lemma}

Using this connection, we can derive a matching lower bound for the functional tolerance regime \eqref{eq:maximum_tolerance_l1}. \citet{bernsteinOrdreMeilleureApproximation1912} proved that the best polynomial approximation of the absolute value satisfies: \begin{equation}\label{eq:bernstein_best_poly_approx}
A_1(L) = \beta_1 \cdot (1+C_L)/L
\end{equation} where $\beta_1 \approx 0.28$ and $C_L\to0 \as L\to\infty$. Consequently, combining Theorem \ref{thm:lower_bound_function_estimation_l1}, Lemma \ref{lemma:moment_matching_duality}, and \eqref{eq:bernstein_best_poly_approx}, we obtain a sharp lower-bound by choosing $\delta \asymp \sigma \cdot \sqrt{L}$ and $L\asymp \log d$. This argument yields the following lower bound on the critical separation.
\begin{corollary} For hypotheses \eqref{eq:gaussian_testing_lp} with $p=1$, the critical separation is lower bounded by
\begin{equation}
\epsilon_1^*(\epsilon_0,\GS)-\epsilon_0 \gtrsim \epsilon_0 \quad\for \epsilon_0\asymp \sigma \cdot \frac{d}{\sqrt{\log d}}.
\end{equation}
\end{corollary} The above rate matches \eqref{eq:maximum_tolerance_l1} up to polylogarithmic factors. Therefore, the plug-in test \eqref{eq:plugin_test_l1} is optimal in the functional estimation regime.

\subsubsection{Lower bound for the interpolation regime}\label{sec:interpolation_lb}

The lower bounds in  Section \ref{sec:simple_null_lb} and  Section \ref{sec:functional_estimation_lb} are driven by moment-matching arguments. It is therefore natural to expect these arguments to extend to the interpolation setting. Note that in \eqref{eq:Mp}, we do not control how large $E_{v\sim \pi_0}|v|$ can be. In order to study the interpolation regime, we must guarantee that it does not exceed $\epsilon_0/(\delta \cdot d)$, while making sure that $E_{v\sim \pi_1}|v|$ is above that threshold. Let $M_p(\epsilon,L)$ denote the constrained moment-matching problem where we maximize the functional distance among moment-matching distributions while constraining their mean \begin{align}\label{eq:Mpepsilon}
M_p(\epsilon, L) &= \sup_{\pi_0,\pi_1} E_{v\sim \pi_1}|v|^p \\&\st  E_{v\sim \pi_0}|v|^p \leq \epsilon^p
\textand m_l(\pi_0)=m_l(\pi_1) \for 1\leq l\leq L
\end{align} where the supremum is taken over all probability measures supported on $[-1,1]$. Then, analogously to Theorem \ref{thm:lower_bound_function_estimation_l1}, the constrained moment-matching problem can be used to lower bound the critical separation. The proof is deferred to  Appendix \ref{sec:moment_matching_lower_bounds}  of the supplementary material. 

\begin{theorem}\label{thm:interpolation_lower_bound_l1}
Choose $\delta\geq 0$ and $L\geq1$ satisfying \eqref{eq:chi2_condition}. The critical separation for hypotheses \eqref{eq:gaussian_testing_lp} with $p=1$ is lower bounded by \begin{equation}
\epsilon_1^*(\epsilon_0,\GS)-\epsilon_0 \gtrsim b_1(\epsilon_0) \quad \where b_1(\epsilon_0)=d\cdot \delta\cdot M_1(\tilde{\epsilon}_0,L) \textand \tilde{\epsilon}_0 \asymp \frac{\epsilon_0}{d\delta}
\end{equation} for $
0 < \epsilon_0 \lesssim b_1(\epsilon_0)$ and $d^{1/2} \gtrsim M^{-1}_1(\tilde{\epsilon}_0,L)$.
\end{theorem}

Thus, any lower bound on the constrained moment-matching problem \eqref{eq:Mpepsilon} leads to a lower bound of the critical separation in the interpolation regime. The major contribution from \cite{canonnePriceToleranceDistribution2021} is proving such a lower bound.

\begin{lemma}[\citet{canonnePriceToleranceDistribution2021}]{MOneEpsilon}\label{lemma:M1Epsilon} For $L$ large enough, it holds that
\begin{equation}
M_1(\epsilon,L) \gtrsim \sqrt{\frac{\epsilon}{L}} \quad \for 0<\epsilon \lesssim 1/L \period
\end{equation}
\end{lemma}

The proof is related to Lemma \ref{lemma:moment_matching_duality}. Rather than directly solving the optimization problem, \citet{canonnePriceToleranceDistribution2021} lower bounded the dual of $M_1(\epsilon,L)$, which is a constrained best polynomial approximation problem. A presentation of their proof can be found in Appendix \ref{sec:constrained_moment_matching} of the supplementary material. 

Finally, to match the upper bound on interpolation regime \eqref{eq:interpolation_regime_l1} up to constants and logarithmic factors, it is enough to combine Theorem \ref{thm:interpolation_lower_bound_l1}, Lemma \ref{lemma:M1Epsilon} and choose $L\asymp\log d$ and $\delta \asymp \sigma \cdot \sqrt{L}$. This leads to the desired result.
\begin{corollary}
Let $\lambda = \epsilon/\sigma$. For $d \gtrsim 1$ and hypotheses \eqref{eq:gaussian_testing_lp} with $p=1$, the critical separation is lower bounded by
\begin{equation}
\epsilon_1^*(\epsilon_0,\GS)-\epsilon_0 \gtrsim \sigma \cdot \sqrt{\frac{d}{\sqrt{\log d}} \cdot \lambda} \quad \for \sqrt{d} \lesssim \lambda \lesssim \frac{d}{\sqrt{\log d}}.
\end{equation}
\end{corollary}

\subsection{Testing under general smooth  \texorpdfstring{$\ell_p$}{lp} norms} \label{sec:smooth_lp_norm}

We next consider the tolerant testing problem~\eqref{eq:gaussian_testing_lp} under
the $\ell_p$ norms when $p$
is an even integer. As discussed in  Section \ref{sec:overview},
these norms are analytic, and in this case the testing problem is
closely related to the corresponding functional estimation problem for all choices of $\epsilon_0$.

In order to construct upper bounds, we consider a debiased plug-in statistic which was first introduced by \cite{ingsterTestingHypothesisWhich2001}. Concretely, let
$\tilde{T}_p = \norm{X}_p^p - E_{P_0}\norm{X}_p^p$,
and notice that for
all even $p>2$, its bias can be quantified exactly  as \begin{align}\label{eq:plugin_bias_at_even_p}
E_{P_v}\left[\tilde{T}_p\right]&=\norm{v}_p^p + E_{P_v}\left[r_{p}(X)\right]\\
\text{with }  r_{p}(X)&= \sum_{j=1}^{k-1}\ \binom{p}{2j} \cdot \mu_{p-2j} \cdot \sigma^{p} \cdot \sum_{i=1}^d H_{2j}\left(\frac{X_i}{\sigma}\right),
\end{align}
  where $k=\lfloor{p/2}\rfloor$, $\binom{p}{k} = p \cdot (p-1)  \dots  (p-k+1) / k!$, $\mu_{p}=E_{Z\sim \Normal(0,1)}|Z|^p$, and $H_{j}$ is the $j$-th probabilistic Hermite polynomial. Thus, we construct a debiased plug-in test by subtracting $r_p(X)$ from the plug-in test \begin{equation}
T_p(X) = \begin{cases}\label{eq:Tp}
\tilde{T}_p(X) &\textif p\leq 2\\
\tilde{T}_p(X) - r_{p}(X) &\otherwise
\end{cases},
\end{equation}and calibrating the resulting statistic \begin{equation}\label{eq:debiased_plugin_test_lp}
\psi_p(X,\epsilon_0) = I\left(T_p(X) \geq \sup_{\norm{v}_p \leq \epsilon_0} q_{1-\alpha}\left(T_p,P_v\right)\right).
\end{equation}

Since there is no bias–variance trade-off for any smooth $\ell_p$ norm, the debiased plug-in test \citep{ingsterTestingHypothesisWhich2001} achieves fast goodness-of-fit testing when $\epsilon_0 \lesssim \sigma \cdot d^{1/2p}$. For $\epsilon_0 \gtrsim \sigma \cdot d^{1/2p}$, even faster rates can be achieved. The proof of the following lemma appears in Appendix \ref{sec:upper_bound_lp_even}  of the supplementary material.

\begin{theorem}\label{lemma:upper_bound_lp_even}
Let $\Tolerance = \epsilon_0/\sigma$. For hypotheses \eqref{eq:gaussian_testing_lp} with fixed $p$ even integer,
there exist universal constants $C_1,C_2>0$ such that the debiased plug-in test \eqref{eq:debiased_plugin_test_lp} is powerful whenever \begin{equation}
\epsilon_1 - \epsilon_0 \gtrsim
\sigma \cdot
\left\{
\begin{array}{lll}
\displaystyle
d^{1/2p},
& \textif \Tolerance \leq C_1 \cdot d^{1/2p},
&\text{(Free tolerance regime)}
\\[2ex]
\displaystyle
d^{1/2}\cdot \lambda^{1-p}, &\textif C_1 \cdot d^{1/2p} \leq \Tolerance \leq C_2\cdot d^{1/p},
\\[2ex]
\displaystyle
d^{1/p - 1/2}, &\textif \Tolerance \geq C_2\cdot  d^{1/p}.
\end{array}
\right.
\end{equation}
\end{theorem}

Another pattern that can be observed above is that the dependence on the dimension diminishes as $p$ increases, resulting in faster testing rates. Intuitively, under $\ell_p$, large values of $p$ diminish the influence of the tail entries of $v$ because raising small values to the $p$-th power reduces their contribution. Consequently, only the largest entries of $v$ significantly impact the test, effectively lowering the problem’s dimensionality.

From a lower bound perspective, consider testing the simple null hypothesis $H_0: v = 0$ versus $H_1: \norm{v}_p \geq \epsilon$. Under the alternative hypothesis, an adversary allocates mass across the $d$ coordinates of $v$ to maximize $\norm{v}_p$ while staying close to the origin. Initially, the mass is spread uniformly. As $p$ increases, each entry decays rapidly at a $d^{-p}$ rate. To maintain some distance from the origin, the adversary must concentrate most of the mass in a few entries. This shift enables faster testing rates because the test can focus on deviations in the largest coordinates of $v$.

\paragraph*{Relationship to functional estimation} The rates observed in Lemma \ref{lemma:upper_bound_lp_even} are fundamentally linked to the performance of the debiased plug-in statistic \eqref{eq:Tp} as an estimator, meaning that estimation also becomes easier as the norm of the mean increases if measured under the appropriate scale. The proof is deferred to Appendix \ref{sec:EstimationLpEven} of the supplementary material.

\begin{proposition}\label{lemma:EstimationLpEven} Let $\lambda = \epsilon_0/\sigma$, $p$ be an even integer, and $T_p$ be the debiased plug-in statistic \eqref{eq:Tp}. It holds that \begin{equation}
\sup_{\norm{v}_p \leq \epsilon_0}E_{X\sim P_v}|T^{1/p}_p(X)-\norm{v}_p| \lesssim \sigma \cdot \begin{cases}
d^{1/2p}, &\textif \lambda \leq d^{1/2p},\\
d^{1/2}\cdot \lambda^{1-p}, &\textif     d^{1/2p} \leq \lambda \leq d^{1/p},\\
d^{1/p-1/2}, &\textif \lambda \geq d^{1/p}.
\end{cases}
\end{equation}
\end{proposition} Therefore, for smooth norms, tolerant testing and functional estimation coincide across the whole tolerance spectrum.

\paragraph*{Lower bound on the critical separation in the free tolerance regime} We recall that for the free tolerance regime, \citet{ingsterTestingHypothesisWhich2001} proved that the upper bound in Lemma \ref{lemma:upper_bound_lp_even} is sharp up to constants. While \citet{ingsterTestingHypothesisWhich2001} constructed lower bounds based on an explicit moment-matching distribution, we can use a similar argument as in  Section \ref{sec:functional_estimation_lb}, to exploit implicitly moment-matching distributions and verify the sharpness of the free tolerance region when $\epsilon_0 \asymp \sigma \cdot d^{1/2p}$.

First, we note that Theorem \ref{thm:lower_bound_function_estimation_l1} can be easily generalized for any $\ell_p$ norm. That is, any lower bound on the unconstrained moment-matching problem \eqref{eq:Mp} leads to a lower bound on the critical separation. The proof is deferred to Appendix \ref{sec:moment_matching_lower_bounds}  of the supplementary material. 

\begin{theorem}\label{thm:unconstrained_moment_matching_lowerbound}
Choose $\delta\geq 0$ and $L\geq1$ satisfying \eqref{eq:chi2_condition}. If $M_p(L)>0$, the critical separation for hypotheses \eqref{eq:gaussian_testing_lp} is lower bounded by  \begin{equation}
\epsilon_1^*(\epsilon_0,\GS)-\epsilon_0 \gtrsim \epsilon_0 \quad\textif \epsilon_0 \asymp \delta\cdot d^{1/p} \cdot M^{1/p}_p(L)  \textand d^{1/2} \gtrsim M^{-1}_p(L).
\end{equation}
\end{theorem}

As noted, $\norm{v}_p^p$ can be estimated without bias when $p$ is even, which leads to faster testing rates. From the point of view of moment-matching distributions, this is reflected by the fact that we can construct two distributions that match at most $p-1$ moments rather than some multiple of $\log d$, as we did for the $\ell_1$ norm in  Section \ref{sec:functional_estimation_lb}. Meaning that under smooth norms, there is a fundamental limit regarding how close two indistinguishable distributions can be.

\begin{lemma}[\citet{newmanApproximationMonomialsLower1976}, Equation (1) of \citet{saibabaApproximatingMonomialsUsing2021a}]\label{lemma:MpEven}
Let $p$ be an even integer. For $1\leq L < p$, it holds that $(2e)^{-1}\cdot g(p,L)  \leq M_p(L)\leq 2 \cdot g(p,L)$ where $g(p,L)=2^{-(p-1)} \cdot \sum_{j=\lfloor{(p+L)/2}\rfloor}^p \binom{p}{j}$. Furthermore, for $L\geq p$, $M_p(L)=0$.
\end{lemma}

Using Theorem \ref{thm:unconstrained_moment_matching_lowerbound} and Lemma \ref{lemma:MpEven}, we can choose $\delta \asymp \sigma \cdot d^{-1/2p}$ and $L=p-1$, in which case $M_p(L)$ behaves like a constant for fixed $p$. Thus, the following bound holds for $d$ large enough.

\begin{corollary}
For $d \gtrsim 1$, and hypotheses \eqref{eq:gaussian_testing_lp} with fixed even $p$, the critical separation is lower bounded by \begin{equation}
\epsilon_1^*(\epsilon_0,\GS)-\epsilon_0 \gtrsim \epsilon_0 \quad \textif \epsilon_0 \asymp \sigma \cdot d^{1/2p}.\end{equation}\end{corollary}

\subsection{Testing under general non-smooth  \texorpdfstring{$\ell_p$}{lp} norms}\label{sec:nonsmooth_lp_norm}

Let us finally comment on testing hypotheses \eqref{eq:gaussian_testing_lp}
when $p > 1$ is not an even integer. In this case, there is no unbiased estimator of $\norm{v}_p^p$, and the bias of the the  plug-in statistic \eqref{eq:debiased_plugin_test_lp} grows with the magnitude of $\norm{v}_p$. Thus, we can expect   the testing rates to follow the same pattern as for $\ell_1$ norm in Theorem \ref{thm:gaussian_testing_l1}.

The next result establishes that, for $p \in [1,2)$, the smallest separation between hypotheses at which the debiased plug-in test can reliably distinguish them interpolates between the known testing rates for a simple null hypothesis \citep{ingsterTestingHypothesisWhich2001} and the functional estimation rates \citep{lepskiEstimationLrnormRegression1999,hanEstimationL_Norms2020}. The proof appears in  Section \ref{sec:upper_bound_lp_odd_less_2}  of the supplementary material.

\begin{lemma}\label{lemma:upper_bound_lp_odd_less_2}
Let $\Tolerance = \epsilon_0/\sigma$. For hypotheses \eqref{eq:gaussian_testing_lp} with fixed $p\in(1,2)$, the debiased plug-in test \eqref{eq:debiased_plugin_test_lp} is powerful whenever: \begin{equation}
\epsilon_1-\epsilon_0 \gtrsim\sigma \cdot
\left\{
\begin{array}{lll}
\displaystyle
d^{1/p - 1/4},
& \textif \Tolerance \lesssim  d^{1/2p},
& \text{(Free tolerance regime)}
\\[2ex]
\displaystyle
\left(d^{1/p} \right)^{1 - p/2} \cdot \Tolerance^{p/2},
& \textif d^{1/2p} \lesssim \Tolerance\lesssim d^{1/p},
& \text{(Interpolation regime)}
\\[2ex]
\displaystyle
d^{1/p} \cdot \sigma,
& \textif \Tolerance \asymp d^{1/p}.
& \text{(Functional estimation regime)}
\end{array}
\right.
\end{equation}
\end{lemma}

However, for odd $p>2$, the rates must change. This can be foreseen by analyzing the mean of $T_p$ \eqref{eq:debiased_plugin_test_lp}. For $1\leq p<2$, $E_{P_v}\left[T_p\right]$ behaves like $\norm{v}_p^p$ as $\norm{v}_p\to 0$. However, for $p$ odd such that $2k<p<2(k+1)$, the  $E_{P_v}\left[T_p\right]$ behaves like $\norm{v}_{2k}^{2k} \cdot \sigma^{p-2k}$ as $\norm{v}_p\to 0$. Thus, for $p>2$, testing $H_0: \norm{v}_p \leq \epsilon_0$ for small values of $\epsilon_0$ should lead to a free tolerance regime that is similar to the one observed for smooth norms in  Lemma \ref{lemma:upper_bound_lp_even}. This observation leads to the following lemma, whose proof appears in Appendix \ref{sec:upper_bound_lp_odd_gtr_2} of the supplementary material.

\begin{lemma}\label{lemma:upper_bound_lp_odd_gtr_2} Let $\Tolerance = \epsilon_0/\sigma$.
For hypotheses \eqref{eq:gaussian_testing_lp} with fixed $p$ such that $2k < p < 2(k+1)$ where $k \in \Zplus$, the debiased plug-in test \eqref{eq:debiased_plugin_test_lp} is powerful whenever
\begin{equation}
\epsilon_1-\epsilon_0 \gtrsim \sigma \cdot
\left\{
\begin{array}{lll}
\displaystyle
d^{1/2p},
& \textif \Tolerance \lesssim d^{1/p - 1/4k},
& \text{(Free tolerance regime)}
\\[2ex]
\displaystyle
\left(d^{1/p} \right)^{1 - 2k/p} \cdot \Tolerance^{2k/p},
& \textif d^{1/p - 1/4k} \lesssim \Tolerance \lesssim d^{1/p},
& \text{(Interpolation regime)}
\\[2ex]
\displaystyle
d^{1/p},
& \textif \Tolerance \asymp d^{1/p}.
& \text{(Functional estimation regime)}
\end{array}
\right.
\end{equation}
\end{lemma}

Comparing the rates obtained in Lemma \ref{lemma:upper_bound_lp_odd_gtr_2} with those in Lemma \ref{lemma:upper_bound_lp_even} for smooth $\ell_p$ norms, we find that they are discontinuous (from the left) at every even $p>2$. This discontinuity, observed by \citet{lepskiMinimaxNonparametricHypothesis1999} and \citet{ingsterTestingHypothesisWhich2001}, arises because the analysis treats $p$ as fixed, and the debiased plug-in statistic only introduces an additional debiasing term, through $r_p(X)$ in \eqref{eq:plugin_bias_at_even_p}, at even values of $p$.

\paragraph*{Lower bound on the critical separation in the functional estimation regime} We recall that \cite{ingsterTestingHypothesisWhich2001} proved that the upper bounds in the free tolerance regime in Lemmas \ref{lemma:upper_bound_lp_odd_less_2} and \ref{lemma:upper_bound_lp_odd_gtr_2} are tight up to constants. Regarding the functional estimation regime of Lemmas \ref{lemma:upper_bound_lp_odd_less_2} and \ref{lemma:upper_bound_lp_odd_gtr_2}, they can be matched up to logarithmic factors using the duality between moment-matching and polynomial approximation as done in  Section \ref{sec:functional_estimation_rates}.

\citet{hanEstimationL_Norms2020} analyzed the minimax estimation of the $L^p$ norm in the Gaussian white noise model using moment-matching distributions. From their results, it is possible to conclude that $M_p^{1/p}(L)$ must decay linearly with the number of matched moments. The proof is deferred to Appendix \ref{sec:unconstrained_moment_matching} of the supplementary material. 

\begin{lemma}\label{lemma:MpOdd} For fixed $p\geq 1$ non-even and $L \geq 1$, it holds that $M_p^{1/p}(L) \asymp L^{-1}$.\end{lemma}
Using Theorem \ref{thm:unconstrained_moment_matching_lowerbound} and Lemma \ref{lemma:MpOdd}, we can choose $L\asymp \log d$ and $\delta = \sigma \cdot \sqrt{L}$, and match the functional estimation regime up to polylogarithmic factors in Lemmas \ref{lemma:upper_bound_lp_odd_less_2} and \ref{lemma:upper_bound_lp_odd_gtr_2} for $d$ large enough.

\begin{corollary}
For $d \gtrsim1 $, the critical separation for hypotheses \eqref{eq:gaussian_testing_lp} with fixed odd $p\geq 1$  is lower-bounded by \begin{equation}
\epsilon_1^*(\epsilon_0,\GS)-\epsilon_0 \gtrsim \epsilon_0 \quad \textif \epsilon_0 \asymp \sigma \cdot \frac{d^{1/p}}{\sqrt{\log d}}.
\end{equation}\end{corollary}

Finally, we conjecture that the interpolation rates in Lemmas \ref{lemma:upper_bound_lp_odd_less_2} and \ref{lemma:upper_bound_lp_odd_gtr_2} are also tight up to polylogarithmic factors, but defer further comments to Appendix \ref{sec:interpolation_lb_lp} of the supplementary material.

\section{The smooth Gaussian white noise model}\label{sec:gaussian_white_noise}

Let $(X(t))_{t\in [0,1]}$ denote a
realization from the
the   Gaussian white noise  model
\begin{equation}\label{eq:gaussian_white_noise}
dX(t) = f(t)\ dt + \sigma \cdot dW(t) \quad \for t \in [0,1],
\end{equation} with respect to a signal $f \in L_2[0,1]$. In this section, our goal is to characterize the minimax tolerant
 testing problem under deviations of $f$ under the $L_p[0,1]$
 norms: $$\|f\|_p = \left(\int_0^1 |f(t)|^pdt\right)^{\frac 1 p}
,\quad 1 \leq p < \infty.$$

Ideally, we would like to test the $L_p$ norm of the signal wihtout additional assumptions: $H_0: \norm{f}_p\leq \epsilon_0$ versus $H_1:\norm{f}_p \geq \epsilon_1$. However, without further restrictions under the alternative, distinguishing these hypotheses becomes impossible since the set of distribution under the null hypothesis is dense in the set of distribution under the alternative \citep{janssenGlobalPowerFunctions2000}. In other words, the critical separation remains constant; see Theorem 2.3 of \citet{ingsterAsymptoticallyMinimaxHypothesisI1993}.

We will do so under the assumption that $f$ lies in a Besov body with smoothness parameter $s > 0$ and integrability parameters $1 \leq p,q \leq \infty$; we recall the definition of the   Besov norm $\norm{\cdot}_{s,p,q}$, in  Appendix \ref{sec:besov_norm} of the supplementary material. Our goal is to distinguish between a small signal under the null hypothesis, versus a large smooth signal under the alternative: \begin{equation}\label{eq:testing_gaussian_white_noise}
H_0: f \in \GWnull \vs H_1: f \in \GWalt
\end{equation} where the null set is given by those signals that are close to the origin, and the alternative set contains smooth signals that are far from the origin: \begin{align}
\GWnull &= \left\{f \in L_2[0,1]: \norm{f}_p \leq \epsilon_0\right\},\\
\textand \GWalt &= \left\{f \in  L_2[0,1]: \norm{f}_p \geq \epsilon_1 \textand \norm{f}_{s,p,q} \leq L\right\},
\end{align}
for a fixed radius $L > 0$. Henceforth, let $P_f$ denote the probability measure corresponding to the Gaussian white noise model \eqref{eq:gaussian_white_noise},
and fix
$\GW = \left\{P_f : f \in L_p[0,1]\cap L_2[0,1]\right\}$ is the collection of all such processes. The next theorem shows that we can use results for testing the Gaussian sequence model, to derive the critical separation for the smooth Gaussian white noise model. The proof appears in Appendix \ref{sec:GaussianWhiteNoiseEquivalence}  of the supplementary material.

\begin{theorem}[Equivalence between Gaussian white noise and sequence models]\label{lemma:GaussianWhiteNoiseEquivalence} Let $\epsilon_1^*(\epsilon_0,\GW)$ denote the critical separation for hypotheses \eqref{eq:testing_gaussian_white_noise}, and $\epsilon_1^*(\tilde{\epsilon}_0,\GS)$ the critical separation for hypotheses \eqref{eq:gaussian_testing_lp}. It holds that \begin{equation}
\epsilon_1^*(\epsilon_0,\GW)-\epsilon_0 \asymp \left(\epsilon_1^*(\tilde{\epsilon}_0,\GS)-\tilde{\epsilon}_0\right) \cdot d^{1/2-1/p} \where \tilde{\epsilon}_0 \asymp \epsilon_0 \cdot  d^{1/p-1/2}
\end{equation} and $d^{-s} \asymp
\epsilon_1^*(\tilde{\epsilon}_0,\GS) \cdot d^{1/2-1/p}$, where the hidden constants depend on $L$, $p$ and $q$. \end{theorem}

The proof of the theorem leverages the fact that, under sufficient smoothness, a Gaussian white noise model can be appropriately discretized into a Gaussian sequence model without losing valuable information. The original argument for the free tolerance region goes back to \citet{ingsterTestingHypothesisWhich2001}.

We illustrate Lemma \ref{lemma:GaussianWhiteNoiseEquivalence} by considering the $p=1$. Henceforth, assume that $s > 0$ and $\sigma =n^{-1/2}$.  Lemma \ref{lemma:GaussianWhiteNoiseEquivalence} implies that the critical separation continuously interpolates between known rates for testing simple null hypotheses \citep{ingsterMinimaxNonparametricDetection1982,ingsterMinimaxDetectionSignal1994} and functional estimation \citep{lepskiEstimationLrnormRegression1999}.

\begin{corollary} The critical separation for hypotheses \eqref{eq:gaussian_white_noise} with $p=1$ is characterized up to polylogarithmic factors by
\begin{equation}
\epsilon_1^*(\epsilon_0,\GW)-\epsilon_0 \asymp \begin{cases}
n^{-\frac{2s}{4s+1}}, &\textif 0\leq \epsilon_0 \lesssim n^{-1/2},\\
\left[n^{-1/2}\cdot \epsilon_0\right]^{\frac{2s}{4s+1}}, &\textif n^{-1/2} \lesssim \epsilon_0 \lesssim  n^{-\frac{s}{2s+1}}, \\
n^{-\frac{s}{2s+1}}, &\textif \epsilon_0 \asymp n^{-\frac{s}{2s+1}}.
\end{cases}\end{equation}
\end{corollary}

\section{The smooth density model}\label{sec:testing_densities}

Consider the case where we have access to $n$ observation from an unknown density $f$: \begin{equation}\label{eq:density_model}
X_1,\dots,X_n \iid f \in \mathcal{D}= \left\{f \in L_2[0,1]: \int f = 1 \textand f \geq 0 \text{ almost everywhere } \right\}.
\end{equation} We aim to test whether $f$ is close to a reference density $g \in \mathcal{D}$ under the $L_p$ norm. As in  Section \ref{sec:gaussian_white_noise}, testing requires smoothness assumptions under the alternative to ensure a non-trivial critical separation. Given a reference density $g$, define the null set as those distributions that are close to $g$ \begin{equation}
\Dnull = \left\{f \in \mathcal{D}: \norm{f-g}_p \leq \epsilon_0\right\},
\end{equation} and the alternative set, as those densities that are far from $g$ but which differ in a smooth way \begin{equation}
\Dalt = \left\{f \in \mathcal{D}:   \norm{f-g}_p \geq \epsilon_1 \textand \norm{f-g}_{s,p,q} \leq L\right\}
\end{equation} where $\norm{\cdot}_{s,p,q}$ denotes the Besov norm, defined in Appendix \ref{sec:besov_norm} of the supplementary material. Then, we aim to test: \begin{equation}\label{eq:density_testing}
H_0: f \in \Dnull \vs H_1 : f \in \Dalt.
\end{equation}

Analogously to Lemma \ref{lemma:GaussianWhiteNoiseEquivalence}, density testing \eqref{eq:density_testing} can be reduced to Multinomial testing \eqref{eq:tolerant_testing_multinomial_l1} by using discretization arguments introduced by \cite{ingsterTestingHypothesisWhich2001} and \cite{arias-castroRememberCurseDimensionality2018}, see also \cite{,BalakrishnanWasserman2019}. Henceforth, let $P_F$ denote the $d$-dimensional Multinomial distribution whose probability mass function is given by \begin{equation}\label{eq:simplex}
F \in \Delta^d = \left\{G \in [0,1]^d: \sum_{i=1}^d G_i=1 \textand \min_{1\leq i\leq d}G_i \geq 0 \right\},
\end{equation} Given observations $X_1,\dots,X_n\sim P_F$ and reference distribution $P_G$ where $F\in \Delta^d$, we define the tolerant testing for Multinomials under $\ell_p$ as: \begin{equation}\label{eq:testing_mutinomials_lp}
H_0: \norm{F-G}_p \leq \epsilon_0 \vs H_1 : \norm{F-G}_p \geq \epsilon_1.
\end{equation} The next theorem proves the equivalence between \eqref{eq:density_testing} and \eqref{eq:testing_mutinomials_lp}. The proof can be found in Appendix \ref{sec:DensityEquivalence} of the supplementary material. 

\begin{theorem}[Equivalence between Density and Multinomial models]\label{lemma:DensityEquivalence}
Let $\epsilon_1^*(\tilde{\epsilon}_0,\D)$ denote the the critical separation for hypotheses \eqref{eq:density_testing}, and let $\epsilon_1^*(\epsilon_0,\Multinomial_d)$ denote the critical separation for hypotheses \eqref{eq:testing_mutinomials_lp}. Then, the following equivalence holds \begin{equation}
\epsilon_1^*(\epsilon_0,\D)-\epsilon_0 \asymp \left(\epsilon_1^*(\tilde{\epsilon}_0,\Multinomial_d)-\tilde{\epsilon}_0\right) \cdot d^{1-1/p} \where \tilde{\epsilon}_0 \asymp \epsilon_0 \cdot d^{1/p-1}
\end{equation} and $d^{-s} \asymp
\epsilon_1^*(\tilde{\epsilon}_0,\Multinomial_d) \cdot d^{1-1/p}$, where the hidden constants depend on $L$, $p$ and $q$.
\end{theorem}

To illustrate Lemma \ref{lemma:DensityEquivalence}, we characterize the critical separation for hypotheses \eqref{eq:density_testing} with $p=1$. Note that rate-wise, the following equivalence between the critical separation under the Gaussian sequence, see Theorem \ref{thm:gaussian_testing_l1}, and the Multinomial, see \eqref{eq:critical_separation_multinomial_l1}, holds: \begin{equation}
\epsilon_1^*(\epsilon_0 \cdot d^{-1/2},\M_d) \cdot d^{1/2} \asymp \epsilon_1^*(\epsilon_0,\GS)\quad \for \sigma=n^{-1/2} \textand 0\leq \epsilon_0 \lesssim d\cdot \sigma.
\end{equation} Together with Lemma \ref{lemma:DensityEquivalence}, this implies that tolerant testing under the $L_1$ norm is equally hard in the density and Gaussian white noise models.

\begin{corollary} The critical separation for hypotheses \eqref{eq:density_testing} with $p=1$ is characterized, up to polylogarithmic factors, by
\begin{equation}\label{eq:equivalence_s_1}
\epsilon_1^*(\epsilon_0,\D) \asymp \epsilon_1^*(\epsilon_0,\GW)\quad \for \sigma=n^{-1/2} \textand 0\leq \epsilon_0 \lesssim n^{-\frac{s}{2s+1}}.
\end{equation}
\end{corollary}

Usually, equivalence between density and Gaussian white noise models is established by proving that the Le Cam deficiency between the models vanishes asymptotically under the assumption that the density is uniformly lower bounded \citep{nussbaumAsymptoticEquivalenceDensity1996,carterDeficiencyDistanceMultinomial2002,mariucciCamTheoryComparison2016,rayCamDistanceDensity2018a}. However, the equivalence \eqref{eq:equivalence_s_1} does not require a lower bound on the density. This is because, proving that the Le Cam deficiency vanishes is a stronger requirement that implies the indistinguishability of the models under any divergence measure, while \eqref{eq:equivalence_s_1} only makes a claim regarding distinguishability under the $L_1$ distance.

\section{Systematic uncertainties in high-energy physics}\label{sec:null_approximation}

As mentioned in the introduction, scientific applications routinely give rise to goodness-of-fit testing problems with imprecise null hypotheses~\citep{gerber2023kernel,liuRobustnessKernelGoodnessFit2024,bailloAlmostGoodnessfitTests2024}. In this section, we explore implications of the tolerant testing framework for goodness-of-fit testing in high-energy physics.

In collider experiments, such as the Large Hadron Collider (LHC;~\cite{aad2008atlas}), data often arise from an
inhomogeneous Poisson point process with an unknown intensity function~\citep{van2014role,kuusela2015statistical}. A canonical problem in these applications is that of assessing whether data are consistent with the {\it Standard Model} (SM)---a theory describing the elementary particles and the forces which act upon them.  To test this goodness-of-fit hypothesis, physicists typically discretize the observed Poisson point process into a Poisson sequence model, and test whether the observed rates match those predicted by the SM. From a minimax perspective, the Poisson sequence model and the Multinomial model are equivalent: tests developed in one setting can be adapted to the other---we refer to Appendix \ref{sec:PoissonSequenceEquivalence} of the supplementary material for a review of this well-known result \citep{neykovMinimaxOptimalConditional2021,canonneTopicsTechniquesDistribution2022,kimConditionalIndependenceTesting2023}.
We can therefore think of an LHC experiment as giving rise to independent observations $X_1,\dots,X_n$ sampled from an unknown discrete distribution $F$ supported on $d$ points. The goal is to assess if the data follows the Standard Model, which defines a discrete distribution $B$: \begin{equation}\label{eq:physics_target}
H_0: F=B \vs H_1: V(F,B)\geq \epsilon.
\end{equation} In practice, however, the distribution $B$
predicted by the SM is only available through Monte Carlo simulation, and is only known up to nuisance parameters. Some of these nuisance parameters are fully unspecified, while others are estimated from past studies, or approximated
mathematically, and thus carry statistical or systematic uncertainties~\citep{heinrich2007systematic}. Taking these various sources of imprecision into account, physicists
define a plausible set of discrete distributions $\mathcal{B}$, which is assumed to contain the true SM distribution $B$. With this set in hand, their goal is to test the following composite analogue of the hypotheses~\eqref{eq:physics_target}:
\begin{equation}\label{eq:physics_target_composite}
H_0: F\in \mathcal{B} \vs H_1: \inf_{\widetilde B\in \mathcal{B}} V(F,\widetilde B)\geq \epsilon.
\end{equation} We argue that the tolerant testing framework
provides a rigorous framework for determining efficient testing procedures in composite hypothesis testing problems of this type.
As an example, suppose that $\mathcal{B}$ is a total variation ball, centered at a prior estimate $\hat B$ of the SM distribution: \begin{align} \label{eq:calB_r}
\mathcal{B} = \big\{ \widetilde B : V(\widetilde B, \hat B) \leq r\big\},
\end{align} where $r > 0$ is an approximation error. In this case, problem~\eqref{eq:physics_target_composite} coincides with the tolerant testing problem for discrete distributions~\citep{canonnePriceToleranceDistribution2021}, discussed in  Section \ref{sec:overview}. The optimal test that attains
the rate in equation~\eqref{eq:critical_separation_multinomial_l1} is valid and powerful for testing hypotheses \eqref{eq:physics_target_composite}
whenever \begin{equation}\label{eq:separation_with_null_approximation}
\epsilon \gtrsim \max\left(\frac{d^{1/4}}{\sqrt{n}}\ ,\ \frac{d^{1/4}}{n^{1/4}}\cdot \sqrt{r}\right)  \quad \for 0 \leq r \lesssim \sqrt{d/n}\period
\end{equation} We can immediately conclude that if the uncertainty ball $\mathcal{B}$ shrinks at the parametric rate, i.e. $r \lesssim 1/\sqrt{n}$, replacing $B$ by its approximation does not imply any loss of power. This is expected, as no test can distinguish $\hat{B}$ from $B$ under that condition. However, whenever $r \gtrsim 1/\sqrt{n}$, the smallest detectable deviation from the Standard Model grows with $\sqrt{r}$.

One can also imagine variants of this problem in which it is difficult to perfectly specify the uncertainty set $\mathcal{B}$, in which case the parameter $r$ in equation~\eqref{eq:calB_r} is unknown. Even in such contexts, the tolerant testing framework provides a way to quantify the tolerance of a goodness-of-fit when naively testing the null hypothesis $H_0: F=\hat{B}$. In particular, we argue that inverting tolerant tests helps assess the sensitivity of
goodness-of-fit p-values to  the accuracy of $\hat{B}$.

To elaborate, let $\psi_{\epsilon}^\alpha$ be a level-$\alpha$ test for $H_0(\epsilon): d(F,\hat{B}) \leq \epsilon$ where $d$ is some distance. A p-value is the smallest $\alpha$ at which we  reject the null hypothesis $\alpha_*(X,\epsilon) = \inf\left\{\alpha: \psi_{\epsilon}^\alpha(X) = 1\right\}$.
Consider the case where we reject the simple null hypothesis $\psi_0^\alpha(X)=1$, and we want to understand how sensitive our conclusion is with respect to $\hat{B}$. The natural idea is to find the largest null hypothesis $H_0(\epsilon)$ that we can reject: \begin{equation}\label{eq:tolerance_factor}
\epsilon_*(X,\alpha) = \sup\left\{\epsilon : \psi_{\epsilon}^\alpha(X) = 1\right\}.
\end{equation} We call $\epsilon_*(X,\alpha)$ the tolerance factor, but note that it has appeared in related works under other names, see the \textit{minimum distance} in \citet{bailloAlmostGoodnessfitTests2024} and the \textit{minimum contamination level} in \citet{barrioApproximateValidationModels2020}.
In fact, under mild assumptions, the tolerance factor serves as a lower confidence bound for the distance between $B$ and $\hat{B}$ \citep{liuBuildingUsingSemiparametric2009}. A larger tolerance factor suggests less concern about the accuracy of $\hat{B}$, i.e., the observed deviations from $\hat{B}$ are large. Conversely, if $\epsilon_*(X)=0$ the results are highly dependent on $\hat{B}$. Thus, one might use
the p-value to make a decision and $\epsilon_*(X,\alpha)$ to analyze the sensitivity of such a decision.

\section{Discussion}\label{sec:discussion}

In this work, we study tolerant testing in the Gaussian sequence model. We characterize the critical separation under the $\ell_1$ norm up to polylogarithmic factors, building on lower bound techniques from \citet{ingsterTestingHypothesisWhich2001} and \citet{canonnePriceToleranceDistribution2021}. We provide partial results for general smooth and non-smooth $\ell_p$ norms that reveal a trade-off between the size of the null hypothesis and the difficulty of estimating the underlying norm. We also explore extensions to smooth Gaussian white noise and density models.

We conclude by noting that tolerant testing opens the possibility of generically testing composite hypotheses \citep{balakrishnanHypothesisTestingHighdimensional2018}. Consider the problem of deciding whether $P$ belongs to a class of distributions $\mathcal{C}$, that is, testing the null $H_0: P \in \mathcal{C}$. A practical testing strategy splits the sample: use the first half to choose a candidate, denoted by $\hat{P}_0$, under the null hypothesis, then use the second half of the data to test the candidate's proximity to the data distribution via tolerant testing $H_0: d(P,\hat{P}_0)\leq \epsilon_0$, where $\epsilon_0$ accounts for the cost of choosing the candidate. A naive implementation offers no advantage over directly estimating the distance of $P$ to the class, denoted by $d(P,\mathcal{C})$. However, in structured settings where the $\mathcal{C}$ is not too large, this approach has yielded tests that require fewer observations to make a reliable decision than estimating $d(P,\mathcal{C})$ \citep{acharyaOptimalTestingProperties2015a,canonneTestingShapeRestrictions2018}. This motivates a central question: how large can $\mathcal{C}$ be so that this practical sample splitting procedure outperforms directly estimating $d(P,\mathcal{C})$? Our work gives a prerequisite for this approach to be successful in a few canonical models: the cost of choosing a candidate in $\mathcal{C}$ should not exceed the cost of estimating $d(P,\mathcal{C})$.

\paragraph*{Acknowledgments} The authors thank Alexandra Carpentier, Ryan Tibshirani, Edward H.~Kennedy, Kai Hormann, Arun Kumar Kuchibhotla, Dominik Rothenh\"ausler, Mikael Kuusela, Kenta Takatsu, and Siddhaarth Sarkar for useful comments and discussions during the development of this work. The authors gratefully acknowledge funding from the National Science Foundation (DMS-2310632). TM gratefully
acknowledges the support
of a Norbert Wiener fellowship.

\pagebreak

\appendix

\begin{center}

{\large\bf SUPPLEMENTARY MATERIAL}

\end{center}

\addcontentsline{toc}{section}{Supplementary material}
\etocdepthtag.toc{mtappendix}
\etocsettagdepth{mtchapter}{none}
\etocsettagdepth{mtappendix}{subsection}
\etocsettagdepth{mtreferences}{section}
\renewcommand{\contentsname}{Supplementary material}
{
\normalsize
\parskip=0em
\renewcommand{\contentsname}{\normalsize Table of contents}
\tableofcontents
}

\pagebreak

\renewcommand{\theequation}{\thesection.\arabic{equation}}
\renewcommand{\thetheorem}{\thesection.\arabic{theorem}}
\renewcommand{\thelemma}{\thesection.\arabic{lemma}}
\renewcommand{\theproposition}{\thesection.\arabic{proposition}}
\renewcommand{\thecorollary}{\thesection.\arabic{corollary}}
\setcounter{equation}{0}

\section{Connection between tolerant and equivalence testing}\label{sec:equivalence_testing}

Bioequivalence (or equivalence) studies aim to show that two drugs yield similar effects \citep{wellekTestingStatisticalHypotheses2002}. The standard design is a two-phase crossover trial: participants are randomly assigned to two groups. In phase one, one group receives the standard treatment, and the other receives an alternative treatment. After a washout period, the groups switch the assigned treatment. This design produces bivariate observations that capture each individual's response across both phases:
\begin{equation}
X_g=(\ X_{g,1}\ ,\ X_{g,2}\ )\ |\ G=g \sim P_g\comma
\end{equation} where $P_1$ and $P_2$ are the response distributions for the two groups. Bioequivalence is established if $P_1$ and $P_2$ are  indistinguishable under the alternative: \begin{equation}\label{eq:equivalence_testing}
H_0: d(P_1,P_2)\geq \epsilon_0 \vs H_1: d(P_1,P_2) \leq \epsilon_1.
\end{equation} Many studies assume $P_1$ and $P_2$ are Gaussian  \citep{dragalinKullbackLeiblerDivergence2003,chowBioavailabilityBioequivalenceDrug2014} and choose distances based on comparing means, variances, or combinations of both \citep{chowBioavailabilityBioequivalenceDrug2014}. Others propose tests based on asymptotic normality \citep{romanoOptimalTestingEquivalence2005}. In practice, however, normality often fails, and nonparametric approaches are of interest~\citep{freitagNonparametricTestSimilarity2007a}.  For instance, \citet{freitagNonparametricTestSimilarity2007a} consider the 2-Wasserstein distance since it approximates metrics recommended by regulatory guidelines \citep{guidance2001statistical} under Gaussianity and extends to arbitrary distributions. 

Rather than tackling equivalence testing \eqref{eq:equivalence_testing} directly, we can connect it to tolerant testing and leverage known results. To simplify the discussion, assume one of the distributions in \eqref{eq:equivalence_testing} is known. The next lemma shows that the critical separation in equivalence and tolerant testing is the same up to constants. 

\begin{proposition}[Correspondence between tolerant and equivalence testing]\label{prop:connection_equivalence_tolerant_testing}
Define the hypotheses sets $H_0(\epsilon_0) = \left\{P : d(P,P_0)\leq \epsilon_0\right\}$ and $H_1(\epsilon_1) = \left\{P : d(P,P_0)\geq \epsilon_1\right\}$. Furthermore, define a set of valid tests $\psi_\alpha(H) = \left\{\psi : \sup_{P \in H} P(\psi(X)=1)\leq \alpha\right\}$. For tolerant testing, the closest detectable alternative is
\begin{equation}
\epsilon_1(\epsilon_0,\alpha,\beta) = \inf\left\{\epsilon_1 : \epsilon_1 \geq \epsilon_0 \textand \inf_{\psi \in \Psi_\alpha(H_0(\epsilon_0))}\sup_{P \in H_1(\epsilon_1)}P(\psi(X)=0)\leq \beta \right\}\comma
\end{equation} and the critical separation is defined as $s(\epsilon_0,T) = \epsilon_1(\epsilon_0,\alpha,\beta) - \epsilon_0$. Analogously, for equivalence testing, the closest detectable alternative is \begin{equation}
\epsilon_0(\epsilon_1,\alpha,\beta) = \sup\left\{\epsilon_0 : \epsilon_1 \geq \epsilon_0 \textand \inf_{\psi \in \Psi_\alpha(H_1(\epsilon_1))}\sup_{P \in H_0(\epsilon_0)}P(\psi(X)=0)\leq \beta \right\}\comma
\end{equation} and the critical separation is defined as : $s(\epsilon_1,E) = \epsilon_1 - \epsilon_0(\epsilon_1,\alpha,\beta)$.

It follows that for any $\epsilon_1^*>0$, there exists $\epsilon_0^*\leq \epsilon_1^*$ such that $s(\epsilon_1^*,E) \lesssim s(\epsilon_0^*,T)$. Additionally, for any $\epsilon_0^*>0$, there exists $\epsilon_1^*\geq \epsilon_0^*$ such that $s(\epsilon_0^*,T) \lesssim s(\epsilon_1^*,E)$.
\end{proposition}
\begin{proof}
We prove only the first claim, since the second claim can be proved analogously. Let $\psi^\alpha_{\epsilon_0}$ be an optimal tolerant test, that is: \begin{equation}
\sup_{P \in H_0(\epsilon_0)} P(\psi(X)=1)\leq \alpha \textand \sup_{P \in H_1(\epsilon_1(\epsilon_0,\alpha,\beta))} P(\psi(X)=0)\leq \beta.
\end{equation} Define the test \begin{equation}
\tilde{\psi}_{\alpha}^{\epsilon_1^*} = 1 - \psi_{\beta}^{\epsilon_0^*} \where \epsilon_0^* = \sup\{\epsilon_0 : \epsilon_1^*=\epsilon_1(\epsilon_0,\beta,\alpha)\}.
\end{equation} Then, the test satisfies: \begin{equation}
\sup_{P \in H_1(\epsilon_1^*)} P(\psi(X)=1)\leq \alpha \textand \sup_{P \in H_0(\epsilon_0^*)} P(\psi(X)=0)\leq \beta.
\end{equation} Consequently, we have derived an upper-bound for the equivalence testing critical separation: \begin{equation}
s(\epsilon_1^*,E) \leq \epsilon_1^*-\epsilon_0^* = \epsilon_1(\epsilon_0^*,\beta,\alpha)-\epsilon_0^* \asymp \epsilon_1(\epsilon_0^*,\alpha,\beta)-\epsilon_0^* = s(\epsilon_0^*,T).
\end{equation}
The claim follows.
\end{proof}

To apply Proposition \ref{prop:connection_equivalence_tolerant_testing}, one must express $\epsilon_1$ as a function of $\epsilon_0$ or vice versa to convert a test for tolerant testing into one for equivalence testing. Since these functions are typically known only up to constants, a more practical approach uses a test statistic that is optimal for tolerant testing and recalibrates it for equivalence testing.

\begin{proposition}\label{prop:equivalence_testing_chebyshev_upper_bound} Let $V_0$ and $V_1$ be two sets of distributions. Consider a statistic $T(X)$ such that $E_{X\sim P}\left[\ T(X)\ -\ \mu(P)\ \right]^2\ \leq\ \phi(P)$. If the following inequality is satisfied \begin{equation}
\sup_{P \in V_0} \mu(P) + \sqrt{\frac{\phi(P)}{\alpha}} \leq \inf_{P \in V_1} \mu(P) - \sqrt{\frac{\phi(P)}{\alpha}}.
\end{equation} Then the test \begin{equation}
\psi(X) = I(T(X) \geq \sup_{P \in V_0} \mu(P) + \sqrt{\frac{\phi(P)}{\alpha}})
\end{equation} controls the type-I and type-II errors by $\alpha$ for $H_0: P \in V_0$ versus $H_1: P \in V_1$. Furthermore, the test \begin{equation}
\psi(X) = I(T(X) \leq \inf_{P \in V_1} \mu(P) - \sqrt{\frac{\phi(P)}{\alpha}})
\end{equation} controls the type-I and type-II errors by $\alpha$ for $H_0: P \in V_1$ versus $H_1: P \in V_0$.

\end{proposition}\begin{proof}
The proof is analogous to that of Lemma \ref{lemma:cheby_bound}.
\end{proof}

\section{General arguments for upper bounding the critical separation}\label{sec:general_upperbounds}

In this section, we recap a series of results used for upper-bounding the critical separation.

\begin{lemma}\label{lemma:cheby_bound} Let the hypotheses be \begin{equation}
H_0: P \in V_0 \vs H_1: P \in V_1\period
\end{equation} Consider the statistic $T$ such that \begin{equation}\label{eq:approx_cond}
E_{X\sim P}\left[\ T(X)\ -\ \mu(P)\ \right]^2\ \leq\ \phi(P)\period
\end{equation} Then the test $\psi(X) = I(T(X) > t) $ where \begin{equation}\label{eq:threshold_req}
t \geq t_* = \sup_{P\in V_0} q_{1-\alpha}\left(T,P\right)
\end{equation} controls the type-I error by $\alpha$. In particular, any $t$ such that \begin{equation}\label{eq:threshold_def}
t \geq t_{\sup}=\sup_{P \in V_0} \mu(P) + \sqrt{\frac{\phi(P)}{\alpha}}
\end{equation} satisfies \eqref{eq:threshold_req}. Furthermore, if \begin{equation}\label{eq:power_condition_for_threshold}
t \leq t_{\inf}=\inf_{P \in V_1} \mu(P) - \sqrt{\frac{\phi(P)}{\beta}}
\end{equation} then the test controls type-II error by $\beta$.
\end{lemma}

\begin{corollary}\label{cor:chebyshev_upperbound} Let the hypotheses be \begin{equation}
H_0: P \in V_0 \vs H_1: P \in V_1\period
\end{equation} Then the test $\psi(X) = I(T(X) > t) $ where \begin{equation}
t \geq t_* = \sup_{P\in V_0} q_{1-\alpha}\left(T,P\right)
\end{equation} controls the type-I error by $\alpha$. In particular, the above condition is satisfied by any $t$ such that \begin{equation}
t \geq t_{\sup}=\sup_{P \in V_0} E_{X\sim P}[T(X)] + \sqrt{\frac{V_{X\sim P}[T(X)]}{\alpha}}.
\end{equation} Furthermore, if \begin{equation}
t \leq t_{\inf}=\inf_{P \in V_1} E_{X\sim P}[T(X)] - \sqrt{\frac{V_{X\sim P}[T(X)]}{\beta}}
\end{equation} then the test controls the type-II error by $\beta$.
\end{corollary}

\begin{corollary}[Testing by learning]\label{lemma:testing_by_learning} Let the hypotheses be \begin{equation}
H_0: P \in V_0 \vs H_1: P \in V_1\period
\end{equation} Furthermore, let $V = V_0 \cup V_1$ and consider a statistic $T$ such that \begin{equation}
\sup_{P \in V}E_{X\sim P}[\ T(X)-\mu(P)\ ]^2 \leq \phi
\end{equation} then the test \begin{equation}
\psi(X) = I(T(X) > t)  \where t = \sup_{P\in V_0}\mu(P) + \sqrt{\frac{\phi}{\alpha}}
\end{equation} controls the type-I error by $\alpha$. Furthermore, it controls the type-II error by $\beta$ whenever \begin{equation}
\inf_{P \in V_1}\mu(P) - \sup_{P \in V_0}\mu(P) \geq C\cdot \phi^{1/2}
\end{equation} where $C = \alpha^{-1/2} + \beta^{-1/2}$.
\end{corollary}

\begin{proof}[Proof of Lemma \ref{lemma:cheby_bound}]

By definition of $t_*$, it holds that for any $t \geq t_*$ the type-I error is bounded by $\alpha$: \begin{equation}
\sup_{P \in V_0} P(T(X) > t)\leq \sup_{P \in V_0} P(T(X) > t_*)\leq \alpha.
\end{equation}

By Markov's inequality, for any $t \geq t_{\sup}$, it holds that the type-I error is bounded by $\alpha$:\begin{align}
\sup_{P \in V_0} P(T(X) > t) &=  \sup_{P \in V_0} P(T(X)-\mu(P) > t-\mu(P))\\
&\leq \sup_{P \in V_0} P(|T(X)-\mu(P)| \geq t-\mu(P))\\
&\leq \sup_{P \in V_0}\frac{E\left[T(X)-\mu(P)\right]^2}{\left[t-\mu(P)\right]^2} &&\since t \geq \sup_{P \in V_0}\mu(P)\\
&\leq \sup_{P \in V_0}\frac{\phi(P)}{\left[t-\mu(P)\right]^2} &&\by \eqref{eq:approx_cond}\\
&\leq \alpha &&\text{by \eqref{eq:threshold_def}}
\end{align} Therefore, it must hold that \begin{equation}
t \geq t_{\sup} \geq t_*.
\end{equation}

Finally, by Markov's inequality and condition \eqref{eq:power_condition_for_threshold}, the type-II error is bounded by $\beta$:
\begin{align}
\sup_{P \in V_1} P(T(X) \leq t) &=  \sup_{P \in V_1} P(T(X)-\mu(P) \leq t-\mu(P))\\
&\leq \sup_{P \in V_1} P(|T(X)-\mu(P)| \geq \mu(P)-t)\\
&\leq \sup_{P \in V_1}\frac{E\left[T(X)-\mu(P)\right]^2}{\left[t-\mu(P)\right]^2} &&\since \inf_{P \in V_1}\mu(P) \geq t\\
&\leq \sup_{P \in V_1}\frac{\phi(P)}{\left[t-\mu(P)\right]^2}&&\by \eqref{eq:approx_cond}\\
&\leq \beta &&\by \eqref{eq:power_condition_for_threshold}
\end{align}

\end{proof}

\section{Critical separation for the Gaussian sequence model}

\subsection{Suboptimality of the chi-squared test for testing hypotheses separated under the \texorpdfstring{$\ell_1$}{l1} norm}\label{sec:failure_l2_test}

In the following, we provide a proof of Lemma \ref{lemma:chi_squared_test_under_l1}. Next, we provide a proof of the auxiliary result Proposition \ref{prop:quantile_bound}.

\begin{proof}[Proof of Lemma \ref{lemma:chi_squared_test_under_l1}]

\textbf{Existence of a valid and powerful chi-squared test.} Under the null hypothesis: $H_0: \norm{v}_1\leq \epsilon_0$, it holds that $\norm{v}_2 \leq \epsilon_0$ by norm monotonicity. Thus, the chi-squared test $\psi_2(X,\epsilon_0)$ \eqref{eq:debiased_plugin_test_lp} is valid under the null hypothesis. By \eqref{eq:free_tolerance_lp_even} in the proof of Lemma \ref{lemma:upper_bound_lp_even}, it holds that the chi-squared test rejects with probability at least $1-\beta$ whenever \begin{equation}\label{eq:upper_bound_lp_even_alt}
\norm{v}_2 - \epsilon_0 \geq \frac{C_1}{2}\cdot \sigma \cdot
d^{1/4} \quad\textif \epsilon_0 \leq \frac{C_1}{2}\cdot \sigma\cdot d^{1/4}
\end{equation} where $C_1$ is a positive constant that depends on $\alpha$ and $\beta$. Under the alternative hypothesis $H_1: \norm{v}_1\geq \epsilon_1$, it holds that $\norm{v}_2\geq d^{-1/2}\cdot \epsilon_1$. Thus, by \eqref{eq:upper_bound_lp_even_alt}, under the alternative hypothesis, the chi-squared test rejects with probability at least $1-\beta$ whenever \begin{equation}\label{eq:upper_bound_lp_even_alt2}
\epsilon_1 - \epsilon_0 \geq (\sqrt{d}-1)\cdot \epsilon_0+\frac{C_1}{2}\cdot \sigma \cdot
d^{3/4} \quad\textif \epsilon_0 \leq \frac{C_1}{2}\cdot \sigma\cdot d^{1/4}
\end{equation} Consequently, under the alternative hypothesis, the test rejects with probability at least $1-\beta$ whenever \begin{equation}
\epsilon_1 - \epsilon_0 \geq C_1\cdot \sigma \cdot
d^{3/4} \quad\textif \epsilon_0 \leq \frac{C_1}{2}\cdot \sigma\cdot d^{1/4}.
\end{equation}

\textbf{Limit on the power of any chi-squared test.} In the following, we prove that for hypotheses \begin{equation}
H_0: \norm{v}_1 \leq \epsilon_0 \vs H_1: \norm{v}_1 \geq \epsilon_1 \asymp d^{3/4} \cdot \sigma \quad \where d^{1/4}\cdot \sigma \lesssim \epsilon_0 \leq \epsilon_1,
\end{equation} no valid chi-squared test uniformly controls the type-II error under the alternative hypothesis.

Henceforth, let $\epsilon_0$ and $\epsilon_1$ satisfy \begin{equation}\label{eq:e0_condition}
C_0 \cdot d^{1/4}\cdot \sigma \leq \epsilon_0 \leq \frac{\epsilon_1}{2} \textand \epsilon_1 =  C_1 \cdot d^{3/4} \cdot \sigma
\end{equation} where $C_0$ and $C_1$ are positive constants that will be defined later. Consider any valid chi-squared test \begin{equation}\label{eq:generic_chi2_test}
\psi(X) = \left(T_2(X)\geq t\right).
\end{equation} where $t$ is called the decision threshold. Since the test is valid, it holds that \begin{equation}
t \geq \sup_{\norm{v}_1\leq \epsilon_0} q_{1-\alpha}\left(T_2,P_v\right)
\end{equation} where $q_{1-\alpha}\left(T_2,P_v\right)$ is the $1-\alpha$ quantile of $T_2(X)$ when $X\sim P_v = \Normal(v,\sigma^2\cdot I_d)$. Furthermore, by Proposition \ref{prop:quantile_bound}, it holds that \begin{equation}
q_{1-\alpha}\left(T_2,P_v\right) \geq E_{P_v}\left[T_2\right] - \sqrt{\frac{1-\alpha}{\alpha}}\cdot \sqrt{V_{P_v}\left[T_2\right]}.
\end{equation} Thus, we have the following lower bound on the decision threshold \begin{equation}
t \geq \sup_{C_0 \cdot d^{1/4}\cdot \sigma \leq \norm{v}_1\leq \epsilon_0} E_{P_v}\left[T_2\right] - \sqrt{\frac{1-\alpha}{\alpha}}\cdot \sqrt{V_{P_v}\left[T_2\right]}.
\end{equation} For any $v\in \R^d$, it holds that \begin{equation}
E_{P_v}\left[T_2\right] = \norm{v}_2^2 \textand V_{P_v}\left[T_2\right] = 2\cdot d\cdot \sigma^4 + 4 \cdot \norm{v}_2^2 \cdot \sigma^2.
\end{equation} Consequently, the following lower bound on the decision threshold holds \begin{equation}\label{eq:threshold_lb}
t \geq \sup_{\leq \norm{v}_1\leq \epsilon_0} \frac{\norm{v}_2^2}{2} = C \cdot \epsilon_0^2 \where C=2^{-1}\quad \textif C_0 \geq 4 \left(\sqrt{\frac{1-\alpha}{\alpha}}\lor \sqrt{2}\right).
\end{equation} Henceforth, assume that $C_0$ satisfies the inequality above.

Under the alternative hypothesis, let \begin{equation}
v_i = \frac{\epsilon_1}{d} \quad \for 1\leq i \leq d,
\end{equation} so that $\norm{v}_1=\epsilon_1$ and  $\norm{v}_2=\epsilon_1^2/d$. Our goal is to bound the probability that the chi-squared test \eqref{eq:generic_chi2_test} rejects under the alternative hypothesis by an arbitrary constant $\beta$. The probability of rejecting is: \begin{equation}\label{eq:rejeciton_under_h1}
P_v\left(T_2(X)\geq t\right)=P_v\left(T_2(X)-\norm{v}_2^2\geq t-\norm{v}_2^2\right)=P_v\left(T_2(X)-\norm{v}_2^2\geq t-\frac{\epsilon_1^2}{d}\right)
\end{equation} Henceforth, we require that \begin{equation}
\frac{\epsilon_1^2}{d} \leq \frac{t}{2}
\end{equation} which, by \eqref{eq:e0_condition} and \eqref{eq:threshold_lb} is satisfied whenever \begin{equation}
C_1 \leq (C/2)^{-1/2}\cdot C_0.
\end{equation} By Chebyshev's inequality and \eqref{eq:rejeciton_under_h1}, it follows that the probability of rejecting under the alternative hypothesis is bounded by \begin{equation}
P_v\left(T_2(X)\geq t\right) \leq \frac{V_{P_v}T_2(X)}{(t/2)^2} \leq \beta
\end{equation} whenever \begin{equation}\label{eq:threshold_lb2}
2\beta^{-1/2} \cdot \sqrt{V_{P_v}T_2(X)} \leq t.
\end{equation} By \eqref{eq:threshold_lb}, condition \eqref{eq:threshold_lb2} is implied by \begin{equation}
C^{-1}2\cdot \beta^{-1/2} \cdot \sqrt{2\cdot d\cdot \sigma^4 + 4 \cdot \frac{\epsilon_1^2}{d} \cdot \sigma^2} \leq \epsilon_0^2,
\end{equation} which, by \eqref{eq:e0_condition}, is satisfied whenever \begin{equation}
C_1 \leq \sqrt{\frac{(C^2\cdot \beta/4)\cdot C_0^4-2}{4}}\cdot d^{1/2} \textand C_0 > \left[\frac{8}{C^2\cdot \beta}\right]^{1/4}.
\end{equation} In summary, for any $C_1$ and $C_0$ that satisfy \begin{align}
&0 < C_1 \leq \left(\sqrt{\frac{(C^2\cdot \beta/4)\cdot C_0^4-2}{4}}\cdot d^{1/2}\right) \wedge \left((C/2)^{-1/2}\cdot C_0\right)\\
&\textand C_0 > \left[\frac{8}{C^2\cdot \beta}\right]^{1/4} \lor 4 \left(\sqrt{\frac{1-\alpha}{\alpha}}\lor \sqrt{2}\right),
\end{align} it holds that the probability that any chi-squared test \eqref{eq:generic_chi2_test} rejects under the alternative hypothesis is upper bounded by $\beta$.

\end{proof}

We conclude the section by proving a proof of Proposition \ref{prop:quantile_bound}, which is a consequence of Cantelli's inequality.
\begin{proposition}\label{prop:quantile_bound}
Let $P$ be a probability distribution on $\mathbb{R}$, and
let $q_{1-\alpha}(T,P)$ be the $1-\alpha$ quantile of $T(X)$ when $X\sim P$. Then, it holds that \begin{equation}
\left|\ q_{1-\alpha}(T,P)-E_{X\sim P}[T(X)]\ \right| \leq \sqrt{\frac{1-\alpha}{\alpha}}\cdot \sqrt{V_{X\sim P}[T(X)]}
\end{equation}
\end{proposition}\begin{proof}
Cantelli's inequality states that \begin{equation}
P\left(T \geq E_{X\sim P}[T(X)] + \sqrt{\frac{1-\alpha}{\alpha}}\cdot \sqrt{V_{X\sim P}[T(X)]} \right) \leq \alpha.
\end{equation} Consequently, it must holds that \begin{equation}
q_{1-\alpha}(T,P) \leq E_{X\sim P}[T(X)] + \sqrt{\frac{1-\alpha}{\alpha}}\cdot \sqrt{V_{X\sim P}[T(X)]}.
\end{equation} Analogously, Cantelli's inequality states that \begin{equation}
P\left(T \leq E_{X\sim P_{v}}[T(X)] - \sqrt{\frac{1-\alpha}{\alpha}}\cdot \sqrt{V_{X\sim P}[T(X)]} \right) \leq \alpha.
\end{equation} Thus, it must hold that \begin{equation}
q_{1-\alpha}(T,P) \geq E_{X\sim P}[T(X)] - \sqrt{\frac{1-\alpha}{\alpha}}\cdot \sqrt{V_{X\sim P}[T(X)]}.
\end{equation}
The claim follows.
\end{proof}

\subsection{Upper bound on the critical separation under \texorpdfstring{$\ell_p$}{lp} for \texorpdfstring{$p\in[1,2)$}{p in [1,2)}}\label{sec:upper_bound_lp_odd_less_2}

In the following, we provide a proof of Lemma \ref{lemma:upper_bound_lp_odd_less_2}. After it, we provide a proof of the auxiliary result Corollary \ref{cor:bound_p_less_2}.

\begin{proof}[Proof of Lemma \ref{lemma:upper_bound_lp_odd_less_2}]

\textbf{Free tolerance and interpolation regimes}

It follows from Corollary \ref{cor:chebyshev_upperbound} that the debiased plug-in test \eqref{eq:debiased_plugin_test_lp} controls the type-II by $\beta$ error if \begin{equation}\label{eq:chebyshev_control_p}
\sup_{\norm{v}_p \leq \epsilon_0}E[T_p] + \sqrt{\frac{V[T_p]}{\alpha}} \leq \inf_{\norm{v}_p \geq \epsilon_1} E[T_p] - \sqrt{\frac{V[T_p]}{\beta}}.
\end{equation} Using the bounds on the expectation and variance in Corollary \ref{cor:bound_p_less_2}, \eqref{eq:chebyshev_control_p} is implied by  \begin{equation}
C_1 \cdot \left( \frac{\sigma^{p-2}}{d^{2/p-1}}\cdot \epsilon_1^2 \wedge \epsilon_1^p \right) \geq \epsilon_0^p + C_2 \cdot  \left[\
\sigma^p \cdot d^{1/2}\ +\ \sigma\cdot d^{1/p-1/2} \cdot \left(\epsilon_1^{p-1} +  \epsilon_0^{p-1}\right) \cdot I(p>1)\ \right]
\end{equation} where $C_1 \in(0,1/8)$ and $C_2$ is a positive constant that depends on $p$, $\alpha$ and $\beta$. There are two cases, depending on which term on the LHS attains the minimum.

\textbf{Case I.} The following condition must be satisfied \begin{align}
C_1 \cdot \epsilon_1^p  \geq\ \epsilon_0^p + C_2 \cdot  \left[\
\sigma^p \cdot d^{1/2}\ +\ \sigma\cdot d^{1/p-1/2} \cdot \left(\epsilon_1^{p-1} +  \epsilon_0^{p-1}\right) \cdot I(p>1)\ \right]
\end{align} which is implied by \begin{equation}
\epsilon_1^p \geq C \cdot \left( \epsilon_0^p + \sigma^p\cdot d^{1/2} +  \sigma^p\cdot d^{1-p/2} \cdot I(p>1) \right)
\end{equation} where $C> 1$. Note that for $p\geq1$, it holds that $1-p/2 \leq 1/2$. Thus, the above condition reduces to \begin{equation}
\epsilon_1^p \geq C \cdot \left(\epsilon_0^p + \sigma^p \cdot d^{1/2} \right)
\end{equation}

\textbf{Case II.} The following condition must be satisfied \begin{align}
C_1 \cdot \frac{\sigma^{p-2}}{d^{2/p-1}}\cdot \epsilon_1^2   \geq\ \epsilon_0^p + C_2 \cdot  \left[\
\sigma^p \cdot d^{1/2}\ +\ \sigma\cdot d^{1/p-1/2} \cdot \left(\epsilon_1^{p-1} +  \epsilon_0^{p-1}\right) \cdot I(p>1)\ \right]
\end{align} which is implied by \begin{equation}
\epsilon_1^p  \geq C \left(  \left[\left[\sigma^p\cdot d\right]^{2/p-1}\ \cdot \epsilon_0^p\right]^{p/2} + \sigma^p \cdot d^{1-p/4} + \sigma^p \cdot d^{\frac{3p}{3-p}\cdot(1/p-1/2)} \cdot I(1<p)\right)
\end{equation} where $C> 1$. Note that $1 - p/4 \geq (3p/(3 - p)) (1/p - 1/2)$ for $1<p\leq2$; hence the above equation is implied by \begin{equation}
\epsilon_1^p  \geq   \left[\left[\sigma^p\cdot d\right]^{2/p-1}\ \cdot \epsilon_0^p\right]^{p/2} + \sigma^p \cdot d^{1-p/4}
\end{equation}

\textbf{Case I and II.} Consider cases I and II together, \eqref{eq:chebyshev_control_p} is satisfied whenever \begin{equation}
\epsilon_1^p \geq 2C\left(\epsilon_0^p + \sigma^p \cdot d^{1/2} +   \left[\left[\sigma^p \cdot d\right]^{2/p-1}\ \cdot \epsilon_0^p\right]^{p/2} + \sigma^p \cdot d^{1-p/4}\right).
\end{equation} Noting that $1-p/4\geq 1/2$ for $p\leq2$, the analysis above equation reduces to \begin{equation}\label{eq:testing_regimes}
\epsilon_1^p \geq 2C\left(\epsilon_0^p + \sigma^p \cdot d^{1/2} +  \left[\left[\sigma^p \cdot d\right]^{2/p-1}\ \cdot \epsilon_0^p\right]^{p/2}\right).
\end{equation}

\textbf{Summary.} By \eqref{eq:testing_regimes}, we have that the debiased plug-in test rejects with probability at least $1-\beta$ whenever \begin{equation}
\epsilon_1^p - \epsilon_0^p \gtrsim  \begin{cases}
\sigma^p\cdot d^{1-p/4} &\textif \epsilon_0^p \lesssim \sigma^p\cdot d^{1/2}\\[1ex]
\left[\left[\sigma^p\cdot d^{1/p}\right]^{2-p}\ \cdot \epsilon_0^p\right]^{p/2} &\textif \sigma^p\cdot d^{1/2}  \lesssim \epsilon_0^p \lesssim \sigma^p\cdot d \\[1ex]
\sigma^p\cdot d &\textif \epsilon_0^p \asymp \sigma^p\cdot d
\end{cases}
\end{equation} The claim of Lemma \ref{lemma:upper_bound_lp_odd_less_2} follows by noting that
\begin{equation}
(\epsilon_1^p-\epsilon_0^p)^{1/p} \asymp \epsilon_1 \asymp \epsilon_1 - \epsilon_0 \quad \for  \epsilon_0\leq \frac{\epsilon_1}{2}.
\end{equation}

\end{proof}

We conclude the section by recaping Lemma 3.1 of \citet{ingsterTestingHypothesisWhich2001}, and proving the corresponding Corollary \ref{cor:bound_p_less_2} used in the proof of Lemma \ref{lemma:upper_bound_lp_odd_less_2}. They allows us to control the mean and variance of the debiased plug-in statistic \eqref{eq:debiased_plugin_test_lp} for $p\in(0,2]$. Before providing a proof, we state some useful auxiliary results.

\begin{proposition}[$C_r$ inequality, see theorem 2 of \cite{clarksonUniformlyConvexSpaces1936}]\label{prop:trig_ineq} For $p>0$
\begin{equation}
|a+b|^p \leq A_p \cdot \left(|a|^p+|b|^p\right) \quad \where A_p=2^{(p-1)_+}
\end{equation}
\end{proposition}
\begin{corollary}\label{cor:cr_lb}
Let $\tilde{a}=a+b$ and $\tilde{b}=-b$, it follows that $
|a+b|^p \geq \frac{|a|^p}{A_p} -  |b|^p$.
\end{corollary}
\begin{corollary}\label{cor:cr_up}
\begin{equation}
\left||Z+u|^p - |u|^p\right| \leq \begin{cases}
|Z|^p &\textif 0<p\leq 1\\
p\cdot A_{p-1} \cdot \left[|Z|^p + |Z| \cdot |u|^{p-1}\right] &\textif 1<p
\end{cases}
\end{equation}
\end{corollary}\begin{proof}
For $p \leq 1$, the bound follows from Proposition \ref{prop:trig_ineq} \begin{equation}
|Z+u|^p - |u|^p \leq (A_p-1) \cdot |u|^p + A_p \cdot |Z|^p
\end{equation} For $p > 1$, by Taylor expansion we have that \begin{align}
\left||Z+u|^p-|u|^p\right| &= p \cdot |\lambda Z+ u|^{p-1} \cdot |Z| &&\where 0 \leq \lambda  \leq 1\\
&\leq p \cdot A_{p-1} \cdot \left(\lambda^p |Z|^p + |Z| \cdot |u|^{p-1}\right)
 &&\text{by Proposition \ref{prop:trig_ineq}}\\
&\leq p \cdot A_{p-1} \cdot \left(|Z|^p + |Z| \cdot |u|^{p-1}\right) &&\since  \lambda \leq 1.
\end{align}
\end{proof}

Using the above auxiliary results, we prove Lemma 3.1 of \citet{ingsterTestingHypothesisWhich2001} and Corollary \ref{cor:bound_p_less_2}.

\begin{lemma}[Lemma 3.1 of \citet{ingsterTestingHypothesisWhich2001}]\label{lemma:Ingsterlemma3.1} Let $0<p\leq 2$, let \begin{equation}
f_p(u) = E|Z+u|^p \textand g_p(u) = |u|^p - f_p(0)
\end{equation} and consider the following expectation and variance \begin{equation}
h(u)=E[g_p(Z+u)] \textand H(u)=V[g_p(Z+u)] \where Z \sim \mathcal{N}(0,1)
\end{equation} It holds that \begin{equation}
C_1 \cdot \left(\ |u|^p \wedge u^2\ \right) \leq h(u)\leq C_p \cdot |u|^p
\end{equation} and
\begin{equation}
H(u) \leq  \begin{cases}
\mu_{2p} &\textif 0<p\leq 1\\
2p \cdot \left[\mu_{2p} + \mu_2\cdot|u|^{2(p-1)}\right] &\textif 1<p\leq 2
\end{cases}
\end{equation} where \begin{equation}
\mu_p = E|Z|^p = \sqrt{\frac{2^p}{\pi}} \textand A_p=2^{(p-1)_+} \textand 0<C_1\leq\frac{1}{2A_p} \textand 0 < C_p \leq 1.
\end{equation} 
\end{lemma}\begin{proof}

\textbf{Lower bound on the expectation.} $h$ is infinitely differentiable at zero, it is convex, and it achieves a minimum at zero. Thus, $h(0)=0, h'(0)=0 \textand h''(0)>0$. Therefore, by Taylor expansion we get that \begin{equation}\label{eq:taylor}
h(u) =  \frac{h''(\lambda \cdot u)}{2} \cdot u^2 \quad \where 0\leq\lambda \leq1
\end{equation} Since $h''$ is continuous and $h''(0)>0$, for any $b$, there exists $a\leq b$ such that $$\min_{|u|\leq a} \frac{h''(u)}{2} >0$$ and consequently, \begin{equation}
h(u) \geq C_{a} \cdot u^2  \quad \for |u|\leq a \where C_{a} = \left[\min_{|u|\leq a} \frac{h''(u)}{2}\right].
\end{equation} Using the convexity of $h$, it follows that \begin{equation}
h(u) \geq \frac{b}{a}\cdot h\left(\frac{a}{b}\cdot u\right) \geq C_b\cdot u^2 \quad \for |u|\leq b \where C_b = C_{a} \cdot \frac{a}{b}
\end{equation} Furthermore, by Corollary \ref{cor:cr_lb}, it holds that \begin{equation}
h(u) \geq \frac{|u|^p}{A_p} - 2 \mu_p \geq \left[\frac{1}{A_p}-\frac{2\mu_p}{b^p}\right] \cdot |u|^p \for |u| \geq b \textand b > \left[2\mu_p A_p\right]^{1/p}
\end{equation} Consequently, choosing $b_\gamma = \gamma \cdot \left[2\mu_p A_p\right]^{1/p}$ for $\gamma > 1$, one of the two lower-bounds must always hold \begin{equation}
h(u) \geq C_{b_\gamma} \cdot u^2 \wedge \frac{1}{A_p}\cdot \left[1-\frac{1}{\gamma}\right]\cdot |u|^p \quad \for \gamma > 1
\end{equation} Note that $C_{b_\gamma}\to0$ and $1-\frac{1}{\gamma}\to 1$ as $\gamma \to \infty$. Hence, choosing $\gamma=2$, we get \begin{equation}
h(u) \geq  \left[\frac{1}{2A_p} \wedge C_{b_2} \right]\cdot \left(\ |u|^p \wedge u^2\ \right)
\end{equation}

\textbf{Upper bound on the expectation.} By Proposition \ref{prop:trig_ineq}, it holds that \begin{align}
h(u) &\leq (A_p-1) \cdot \mu_p + A_p \cdot |u|^p \\
&\leq \begin{cases}
|u|^p &\textif 0<p\leq 1\\
2p\cdot |u|^p  &\textif 1<p<2 \textand |u| \geq [(1-A_p^{-1})\mu_p]^{1/p}
\end{cases} \\
&\leq \begin{cases}
|u|^p &\textif 0<p\leq 1\\
2p\cdot |u|^p  &\textif 1<p<2 \textand |u| \geq 1
\end{cases}
\end{align} since $(1-1/A_p)\mu_p \leq 1 \for 1<p<2$. Additionally, for $1<p<2$, by the Taylor expansion in \eqref{eq:taylor} it holds that \begin{equation}
h(u) \leq \left[\max_{|u|\leq 1} \frac{h''(u)}{2}\right] \cdot u^2 \leq C \cdot |u|^p \quad \for |u|\leq 1
\end{equation} where $C = \max_{|u|\leq 1} \frac{h''(u)}{2} \geq \frac{h''(0)}{2} >0$. Thus, it holds that \begin{equation}
h(u) \leq C_p \cdot |u|^p \where C_p \geq 2p \lor C \textand 1 < p < 2.
\end{equation}

\textbf{Upper bound on the variance.} For the variance, note that \begin{align}
H(u) = V\left[|Z+u|^p\right] = V\left[|Z + u|^p - |u|^p\right].
\end{align} By Corollary \ref{cor:cr_up}, the following upper-bound holds \begin{align}
H(u) &\leq  E\left[|Z + u|^p - |u|^p\right]^2\\
&\leq \begin{cases}
\mu_{2p} &\textif 0<p\leq 1\\
2p \cdot \left[\mu_{2p} + \mu_2\cdot|u|^{2(p-1)}\right] &\textif 1<p\leq 2
\end{cases}
\end{align}

\end{proof}

\begin{corollary}\label{cor:bound_p_less_2} For $0<p\leq 2$, it holds that the debiased plug-in statistic \eqref{eq:debiased_plugin_test_lp} satisfies \begin{equation}
T_p(X) = \sigma^p \cdot \sum_{i=1}^d g_p\left(\sigma^{-1} \cdot X_i\right) = \sum_{i=1}^d |X_i|^p - \sigma^p \cdot d\cdot\mu_p
\end{equation} where $X_i \sim \mathcal{N}(v_i,\sigma^2)$ and $\mu_p = E|Z|^p$ such that $Z \sim \mathcal{N}(0,1)$. The expectation of $T_p$ is bounded as follows \begin{equation}\label{eq:Tp_expectation_bound_p_less_than_2}
L_p \leq E[T_p(X)]  \leq \norm{v}_p^p
\end{equation} where \begin{align}
L_p = \begin{cases}
C \cdot \left( \frac{\sigma^{p-2}}{d^{2/p-1}}\cdot \norm{v}_p^2 \wedge \norm{v}_p^p \right) &\textif 0<p<2\\
\norm{v}_2^2 &\textif p=2
\end{cases}\end{align} and $C \in(0,1/8)$. Furthermore, the variance of $T_p$ is bounded by \begin{equation}\label{eq:Tp_variance_bound_p_less_than_2}
V\left[T_p(X)\right] \leq \begin{dcases}
d\cdot \mu_{2p} \cdot \sigma^{2p} &\textif 0<p\leq 1\\
2p \cdot \left[d\cdot \mu_{2p} \cdot \sigma^{2p} + \mu_2\cdot d^{\frac{2}{p}-1} \cdot\norm{v}_{p}^{2(p-1)} \cdot \sigma^2 \right] &\textif 1<p< 2\\
2 \cdot d \cdot \sigma^4 + 4 \cdot \norm{v}_2^2 \cdot \sigma^2 &\textif p=2
\end{dcases}.\end{equation} Finally, for $1 \leq p \leq 2$, the bias is bounded:
\begin{equation}\label{eq:Tp_bias_bound_p_less_than_2}
\left|\ \norm{v}_p^p - E[T_p(X)]\ \right| \leq \norm{v}_p^p\ \wedge\ d \cdot \sigma^p \cdot \mu_p\period
\end{equation}
\end{corollary}
\begin{proof}[Proof of Corollary \ref{cor:bound_p_less_2}]

It holds that \begin{equation}
E[T_p] = \sigma^p \cdot \sum_{i=1}^d h(\sigma^{-p}\cdot v_i) \textand V[T_p]= \sigma^{2p} \cdot \sum_{i=1}^d H(\sigma^{-p}\cdot v_i)
\end{equation} From Lemma \ref{lemma:Ingsterlemma3.1}, the upper bounds on the expectation and variance follow immediately. Regarding, the lower bound on the expectation, by Lemma \ref{lemma:Ingsterlemma3.1}, we have that \begin{equation}
E[T_p] \geq C_1 \cdot \left(\sum_{i\not\in I} |v_i|^p + \sigma^{p-2} \cdot \sum_{i\in I} v_i^2\right)
\end{equation} where $I = \left\{i : |v_i|\leq \sigma \right\}$. Since \begin{equation}
\norm{v}_p^p = \sum_{i\not\in I} |v_i|^p + \sum_{i\in I} |v_i|^p\comma
\end{equation} it must hold that \begin{equation}
(A)\quad \sum_{i\in I} |v_i|^p \geq \frac{\norm{v}_p^p}{2} \quad \textor\quad (B)\quad \sum_{i\not\in I} |v_i|^p \geq \frac{\norm{v}_p^p}{2}
\end{equation} Under (B), we have that \begin{equation}\label{eq:mean_lb_1}
E[T_p] \geq C_1 \sum_{i\not\in I} |v_i|^p \geq \frac{C_1}{2}\cdot \norm{v}_p^p\period
\end{equation} Alternatively, under (A), it follows that  \begin{equation}\label{eq:mean_lb_2}
E[T_p] \geq C_1\sigma^{p-2} \sum_{i\in I} v_i^2 \geq C_1 \cdot \frac{\sigma^{p-2}}{|I|^{2/p-1}}\cdot \left(\sum_{i\in I} v_i^p\right)^{2/p} \geq \frac{C_1}{4^{1/p}} \cdot \frac{\sigma^{p-2}}{d^{2/p-1}}\cdot \norm{v}_p^2.
\end{equation} Hence by \eqref{eq:mean_lb_1} and \eqref{eq:mean_lb_2}, it holds that \begin{equation}\label{eq:lowerbound_vp}
E[T_p] \geq C \cdot \left( \frac{\sigma^{p-2}}{d^{2/p-1}}\cdot \norm{v}_p^2 \wedge \norm{v}_p^p \right)
\end{equation} where $C = C_1 \cdot \left(\frac{1}{2}\wedge\frac{1}{4^{1/p}}\right)$. Since $0 < C_1 \leq \frac{1}{2A_p}$ where $A_p=2^{(p-1)\lor0}$, it follows that $0 < C < 1/8$.

\textbf{Bias.} For $1 \leq p \leq 2$, $u\to|u|^p$ is convex. Thus, by Jensen's inequality, we have that \begin{equation}
\norm{v}_p^p - d\sigma^p\mu_p = \sum_{i=1}^d|EX_i|- d\sigma^p\mu_p\leq E[T]\period
\end{equation} Together with the fact that $E[T]\leq \norm{v}_p^p$, it holds that the bias of $T_p$ is capped \begin{equation}
|\norm{v}_p^p - E[T]| \leq d\sigma^p\mu_p
\end{equation} Using \eqref{eq:lowerbound_vp}, it also holds that \begin{equation}
|\norm{v}_p^p - E[T]| \leq \norm{v}_p^p - C \cdot \left( \frac{\sigma^{p-2}}{d^{2/p-1}}\cdot \norm{v}_p^2 \wedge \norm{v}_p^p \right) \leq (1-C)\cdot \norm{v}_p^p
\end{equation} Thus, \eqref{eq:Tp_bias_bound_p_less_than_2} follows.
\end{proof}

\subsection{Upper bound on the critical separation under \texorpdfstring{$\ell_p$}{lp} for p even}\label{sec:upper_bound_lp_even}

In the following, we provide a proof of Lemma \ref{lemma:upper_bound_lp_even}.

\begin{proof}[Proof of Lemma \ref{lemma:upper_bound_lp_even}]
It follows from Corollary \ref{cor:chebyshev_upperbound} and Corollary \ref{cor:bound_p_gtr_2} that the debiased plug-in test \eqref{eq:debiased_plugin_test_lp} controls the   type-II by $\beta$ error whenever \begin{equation} \label{eq:chebyshev_cond_lp_even}
\sup_{\norm{v}_p \leq \epsilon_0}E[T_p] + \sqrt{\frac{V[T_p]}{\alpha}} \leq \inf_{\norm{v}_p \geq \epsilon_1} E[T_p] - \sqrt{\frac{V[T_p]}{\beta}}.
\end{equation} where \begin{equation}
E_{P_v}[T_p]=\norm{v}_p^p \textand V_{P_v}[T_p]\leq C\cdot \left(d \cdot \sigma^{2p} + d^{2/p-1} \cdot \sigma^{2} \cdot \norm{v}_p^{2(p-1)}\right)
\end{equation} for some positive constant $C$ that depends only on $p$. The RHS of \eqref{eq:chebyshev_cond_lp_even} is non-negative whenever \begin{equation}\label{eq:non_neg_cond}
\epsilon_1 \geq C_0 \cdot d^{1/2p}\cdot \sigma
\end{equation} where $C_1$ is a positive constant. Henceforth, assume that the above condition holds. Then, it holds that \eqref{eq:chebyshev_cond_lp_even} is implied by \begin{equation}\label{eq:chebyshev_cond_lp_even_1}
\epsilon_1^p - \epsilon_0^p \geq C \cdot \left[d^{1/2}\cdot \sigma^p + d^{1/p-1/2}\cdot \sigma \cdot \left(\epsilon_0^{p-1}+\epsilon_1^{p-1}\right)\right]
\end{equation} where $C>0$, and it depends only on $p$, $\alpha$ and $\beta$.

\textbf{Free tolerance regime.} Note that \eqref{eq:non_neg_cond} and \eqref{eq:chebyshev_cond_lp_even_1} hold under the following conditions \begin{equation}
\epsilon_1\geq C_1 \cdot d^{1/2p} \cdot \sigma \textand \epsilon_0 \leq \frac{\epsilon_1}{2}
\end{equation} for $C_1$ large enough that, where $C_1$ is a positive constant that depends on $p$, $\alpha$ and $\beta$. Thus, it also holds for \begin{equation}\label{eq:free_tolerance_lp_even}
\epsilon_1-\epsilon_0 \geq \frac{C_1}{2}\cdot d^{1/2p} \cdot \sigma  \textand \epsilon_0 \leq \frac{C_1}{2}\cdot d^{1/2p}
\end{equation}

\textbf{Beyond the free tolerance regime.} Let $\epsilon_0>0$. Using the fact that \begin{equation}
\frac{\epsilon_1^p-\epsilon_0^p}{\epsilon_1^{p-1}+\epsilon_0^{p-1}}=(\epsilon_1-\epsilon_0)\cdot \frac{\epsilon_1^{p-1}+\epsilon_1^{p-2}\epsilon_0+\dots+\epsilon_1\epsilon_0^{p-2}+\epsilon_0^{p-1}}{\epsilon_1^{p-1}+\epsilon_0^{p-1}} \geq \epsilon_1-\epsilon_0,
\end{equation} it follows that \eqref{eq:non_neg_cond} and \eqref{eq:chebyshev_cond_lp_even_1} are implied by \begin{equation}
\epsilon_1 \geq C_0 \cdot d^{1/2p}\cdot \sigma \textand \epsilon_1 - \epsilon_0 \geq 2C\cdot \sigma \cdot \max\left( d^{1/2}\cdot \frac{\sigma^{p-1}}{\epsilon_0^{p-1}}\ ,\ d^{1/p-1/2} \right)
\end{equation} from which the last two regimes of the lemma follow.

\end{proof}

\subsection{Upper bound on the estimation error under \texorpdfstring{$\ell_p$}{lp} for p even}\label{sec:EstimationLpEven}

In this section, we prove Lemma \ref{lemma:EstimationLpEven}, and after it, we provide a proof of the auxiliary result Proposition \ref{prop:lpgeometry}.

\begin{proof}[Proof of Lemma \ref{lemma:EstimationLpEven}]

Let $$
\gamma_p = (d^{1/2} \cdot \sigma^{p}) \lor (d^{1/p-1/2} \cdot \sigma \cdot \norm{v}_p^{p-1}).$$ By Corollary \ref{cor:bound_p_gtr_2}, it holds that \begin{equation}
E_{X\sim P_v}|T_p(X)-\norm{v}_p^p| \leq \sqrt{V_{X\sim P_v}\left[T_p(X)\right]}\lesssim  \gamma_p.
\end{equation} By Proposition \ref{prop:lpgeometry}, it holds that $$
|a^{1/p}-b^{1/p}|\leq |a-b|^{1/p} \quad \since |a-b|\leq |a^p-b^p|^{1/p}.
$$ Hence, the following bound on the estimation error holds \begin{equation}
E_{X\sim P_v}|T^{1/p}_p(X)-\norm{v}_p| \leq E_{X\sim P_v}|T_p(X)-\norm{v}_p^p|^{1/p} \leq \left(E_{X\sim P_v}|T_p(X)-\norm{v}_p^p|\right)^{1/p}
\end{equation} where we used Jensen's inequality in the last step. Additionally, by the difference of powers formula, we have the following relationship \begin{equation}
a - b = (a^{1/p})^p - (b^{1/p})^{p} \geq (a^{1/p}-b^{1/p})\cdot (a^{(p-1)/p}+b^{(p-1)/p}).
\end{equation} Thus, the following bound on the estimation error holds \begin{equation}
E_{X\sim P_v}|T^{1/p}_p(X)-\norm{v}_p| \leq E_{X\sim P_v}\left|\frac{T_p(X)-\norm{v}_p^p}{T_p^{(p-1)/p}(X)+\norm{v}_p^{p-1}}\right|
\end{equation} In summary, the following bound holds on the estimation error \begin{equation}
E_{X\sim P_v}|T^{1/p}_p(X)-\norm{v}_p| \lesssim \frac{\gamma_p}{\norm{v}_p^{p-1}} \wedge \gamma_p^{1/p} = \begin{cases}
\gamma_p^{1/p}, &\textif \norm{v}_p^p \leq \gamma_p\\
\frac{\gamma_p}{\norm{v}_p^{p-1}}, &\otherwise
\end{cases}
\end{equation} The claim follows by a case-by-case analysis. \end{proof}

We proceed to prove Proposition \ref{prop:lpgeometry}.

\begin{proposition}\label{prop:lpgeometry} Let $p$ be an integer, and  $a,b\geq 0$. It holds that
$$
|a-b|\leq |a^p-b^p|^{1/p}
$$
\end{proposition}\begin{proof}
Let $m=a \wedge b$ and $u=|a-b|$. By the Binomial theorem , it follows that \begin{equation}
|a^p-b^p|=(m+u)^p-m^p = \sum_{k=1}^p \binom{p}{k} m^{p-k}\cdot u^k\geq u^p=|a-b|^p
\end{equation}
\end{proof}

\subsection{Upper bound on the critical separation under \texorpdfstring{$\ell_p$}{lp} for \texorpdfstring{$2k< p < 2(k+1)$}{2k< p < 2(k+1)} where \texorpdfstring{$k \in \mathbb{Z}_+$}{k in Z+}}\label{sec:upper_bound_lp_odd_gtr_2}

In this section, we prove Lemma \ref{lemma:upper_bound_lp_odd_gtr_2}. After it, we provide a proof of the auxiliary result Corollary \ref{cor:bound_p_gtr_2}.

\begin{proof}[Proof of Lemma \ref{lemma:upper_bound_lp_odd_gtr_2}]

It follows from Corollary \ref{cor:chebyshev_upperbound} and Corollary \ref{cor:bound_p_gtr_2} that the debiased plug-in test \eqref{eq:debiased_plugin_test_lp} controls the type-II by $\beta$ error whenever

\begin{equation}
\epsilon_1^p \geq C\cdot \left( \epsilon_0^p + \sigma^{p-2k}\cdot d^{1-2k/p}\cdot \epsilon_0^{2k} + \sigma^p\cdot d^{1/2} + \sigma\cdot d^{1/p-1/2} \cdot \left[\epsilon_0^{(p-1)}+\epsilon_1^{(p-1)}\right]\right),
\end{equation} where $C\geq 1$ and it depends on $p$, $\alpha$ and $\beta$. Using the fact that $\epsilon_1 \geq \epsilon_0$, the above equation is implied by \begin{equation}
\left(\epsilon_1^p-\epsilon_0^p\right)^{1/p} \gtrsim \epsilon_0 + \sigma^{1-2k/p}\cdot d^{1/p-2k/p^2}\cdot \epsilon_0^{2k/p} + \sigma\cdot d^{1/2p} + \sigma\cdot d^{1/p-1/2}.
\end{equation} Note that $1/p-1/2 \leq 1/2p$ for $p\geq 2$; Hence, \begin{equation}
\left(\epsilon_1^p-\epsilon_0^p\right)^{1/p}  \gtrsim  \epsilon_0 + \sigma^{1-2k/p}\cdot d^{1/p-2k/p^2}\cdot \epsilon_0^{2k/p} + \sigma\cdot d^{1/2p},
\end{equation} The claim of  Lemma \ref{lemma:upper_bound_lp_odd_gtr_2} follows by noting that \begin{equation}
(\epsilon_1^p-\epsilon_0^p)^{1/p} \asymp \epsilon_1 \asymp \epsilon_1 - \epsilon_0 \quad \for  \epsilon_0\leq \frac{\epsilon_1}{2}.\end{equation}
\end{proof}

In the following, we recap Lemma 3.2 of \citet{ingsterTestingHypothesisWhich2001} and prove Corollary \ref{cor:bound_p_gtr_2}, which allows us to control the mean and variance of the debiased plug-in statistic \eqref{eq:debiased_plugin_test_lp}.

\begin{lemma}[Lemma 3.2 of \citet{ingsterTestingHypothesisWhich2001}]\label{Ingsterlemma3.2} Let $2k\leq p< 2(k+1) \st k \in \mathbb{Z}_+$ , and let \begin{equation}
f_p(u) = E|Z+u|^p \textand g_p(u) = |u|^p - \sum_{j=0}^{k-1}\ \frac{f^{(2j)}_p(0)}{(2j)!} \cdot H_{2j}(u)
\end{equation} where $f^{(2j)}_p(0)$ denotes the $2j$th deriative of $f_p$ evaluated at zero, and $H_{2j}$ is the $2j$th Hermite polynomial. Consider the following expectation and variance \begin{equation}
h(u)=E[g_p(Z+u)] \textand H(u)=V[g_p(Z+u)] \where Z \sim \mathcal{N}(0,1).
\end{equation} It holds that \begin{equation}
C_1 \cdot |u|^p \leq h(u)\leq C_2 \cdot \left(\ |u|^p + u^{2k}\ \right)  \textand H(u) \leq  C_3 \cdot (1+|u|^{2(p-1)})
\end{equation} where $C_1$ and $C_3$ are positive constants that depends only on $p$, and \begin{equation}
C_2 = 2\cdot \frac{d_{2k}}{(2k)!}\cdot A_{p-2k} \cdot (\mu_{p-2k} + 2^{p-2k}) \st d_j = p \cdot (p-1)  \dots  (p-j+1).
\end{equation}
For the particular case of $p$ even, it holds that $h(u) = u^{p}$.

\end{lemma}\begin{proof}

\textbf{Upper and lower bound on the expectation.} By propagating the expectation, we have that \begin{equation}
h(u) = f_p(u) - r_{k-1}(u) \where r_{k-1}(u) = \sum_{j=0}^{k-1}\ \frac{f^{(2j)}_p(0)}{(2j)!} \cdot u^{2j}.
\end{equation} Note that $h(u)$ is an even infinite differentiable function, and it is convex. It follows that $
h^{(j)}(0)=0 \for j \in\{0,\dots,2k-1\}
$ by definition and evenness. Additionally, $h$ is convex.

By Taylor expansion, it holds that \begin{equation}
h(u) = \frac{h^{(2k)}(\lambda \cdot u)}{(2k)!} \cdot u^{2k} = \frac{f_p^{(2k)}(\lambda \cdot u)}{(2k)!} \cdot u^{2k}  \where 0 \leq \lambda \leq 1
\end{equation} Since $|Z+u|^p$ is integrable on $\R$ for $|u|\leq \epsilon $ with $\epsilon>0$, $\pdv[2k]{|Z+u|^p}{u}$ exists on $\R\times[-\epsilon,\epsilon]$, and $|\pdv[2k]{|Z+u|^p}{u}|$ is upper-bounded by the integrable function $d_{2k} \cdot A_{p-2k} \cdot (|Z|^{p-2k}+|u|^{p-2k})$, by the Lebesgue dominated convergence theorem \citep{bartleElementsIntegrationLebesgue1995}[Corollary 5.9], it follows that \begin{equation}
f_p^{(2k)}(u) = E\left[\pdv[2k]{|Z+u|^p}{u}\right] = d_{2k} \cdot E|Z+u|^{p-2k}
\end{equation}  Furthermore $f_p^{(2k+1)}(0)=0$ since $f_p$ is an even function. Consequently \begin{equation}
h(u) = \frac{d_{2k}}{(2k)!} \cdot E|Z+\lambda u|^{p-2k} \cdot u^{2k}
\end{equation} For $|u|\leq b$ with $b>0$, it follows that \begin{equation}
\frac{d_{2k}}{(2k)!}  \cdot \mu_{p-2k} \cdot u^{2k} \leq h(u) \leq \frac{d_{2k}}{(2k)!} \cdot A_{p-2k} \cdot (\mu_{p-2k} + b^{p-2k}) \cdot u^{2k}
\end{equation} Note that from the lower-bound, it follows that \begin{equation}
\frac{d_{2k}}{(2k)!}  \cdot \mu_{p-2k} \cdot |u|^{p} \leq h(u) \quad \for |u| \leq 1
\end{equation} using the convexity of $h$, we get for arbitrary $b$ that  \begin{equation}
\frac{d_{2k}}{(2k)!}  \cdot \mu_{p-2k} \cdot b^{1-p} \cdot |u|^{p} \leq b\cdot h\left(\frac{u}{b}\right) \leq h(u) \quad \for |u| \leq b.
\end{equation} For $|u|\geq b$ and $b\geq2$, by Corollary \ref{cor:cr_lb}, it follows that \begin{equation}
h(u) \geq \frac{|u|^p}{A_{p-1}} - \mu_p - |r_{k-1}(u)| \geq C_b \cdot \left(|u|^p - 1 - u^{2(k-1)}\right) \geq C_b \cdot (1-2^{-p}-2^{2(k-1)-p}) \cdot |u|^p,
\end{equation} and for the upper-bound, we have that \begin{equation}
h(u) \leq \frac{d_{2k}\cdot A_{p-2k}}{(2k)!}  \cdot (\mu_{p-2k} + |u|^{p-2k}) \cdot u^{2k} \leq \frac{2\cdot d_{2k}\cdot A_{p-2k}}{(2k)!} \cdot |u|^p.
\end{equation} In summary, choosing $b=2$, it holds that \begin{equation}
C \cdot |u|^p \leq h(u) \leq \frac{2\cdot d_{2k}\cdot A_{p-2k}}{(2k)!} \cdot |u|^p + \frac{d_{2k} \cdot A_{p-2k}}{(2k)!}  \cdot (\mu_{p-2k} + 2^{p-2k}) \cdot u^{2k}
\end{equation} where $C>0$.

\textbf{Upper bound on the variance.} Note that the following chain of inequalities holds \begin{align}
H(u) &= V[|Z+u|^p - q_{k-1}(Z+u)^2] \quad \where  q_{k-1} \text{ is a $k-1$ polynomial}\\
&= V[|Z+u|^p - |u|^p - q_{k-1}(Z+u)]\\
&\leq 2\cdot E\left[|Z+u|^p - |u|^p\right]^2 + 2\cdot E\left[q_{k-1}(Z+u)^2\right]^2\\
&\leq 2\cdot E\left[|Z+u|^p - |u|^p\right]^2 + C \cdot E[Z+u]^{4(k-1)}\\
&\leq C_p \cdot \left(1 + |u|^{2(p-1)}\right) + C'_p \cdot \left(1 + |u|^{4(k-1)}\right)\\
&\leq C_p'' \cdot (1 + |u|^{2(p-1)} + |u|^{4(k-1)})
\end{align} For $|u|\leq 1$ the first term dominates, while for $|u|\geq 1$ the second term dominates since $2(p-1)>4(k-1)$. Thus, \begin{equation}
H(u) \leq C_p \cdot (1 + |u|^{2(p-1)}).
\end{equation}

\textbf{Expectation when $p$ is an even integer.} By the Binomial theorem, the following equality holds $$
E[Z+u]^{p} = E\left[\sum_{j=0}^{p} \binom{p}{j}\cdot Z^{p-2j} \cdot u^{j}\right].
$$ Then, using the fact that the odd moments of a standard Gaussian vanish, it follows that \begin{align}
E[Z+u]^{p} &= \sum_{j=0}^{p/2} \binom{p}{2j}\cdot \mu_{p-2j} \cdot u^{2j} \\
&= u^{p} + \sum_{j=0}^{p/2-1} \frac{d_{2j}}{(2j)!}\cdot \mu_{p-2j} \cdot u^{2j} &&\since \frac{d_{j}}{j!}=\binom{p}{j}\\
&= u^{p} + \sum_{j=0}^{p/2-1} \frac{f_p^{(2j)}(0)}{(2j)!} \cdot u^{2j} &&\since f_p^{(2j)}(0)=d_{2j} \cdot \mu_{p-2j}.
\end{align} Hence,  \begin{equation}
h(u) = E[Z+u]^{p} - \sum_{j=0}^{p/2-1} \frac{f_p^{(2j)}(0)}{(2j)!} \cdot u^{2j} = u^p.
\end{equation}

\end{proof}
\begin{corollary}\label{cor:bound_p_gtr_2} For $2k\leq p< 2(k+1) \st k \in \mathbb{Z}_+$, the debiased plug-in statistic satisfies \begin{equation}
T_p(X) = \sigma^{p} \cdot \sum_{i=1}^d g_p\left(\sigma^{-1} \cdot X_i\right)  \where X_i \sim \mathcal{N}(v_i,\sigma^2).
\end{equation} Hence it follows that \begin{align}
E[T_p] = \sigma^{p} \cdot \sum_{i=1}^d h(\sigma^{-1}v_i) \textand V[T_p] = \sigma^{2p} \cdot \sum_{i=1}^d H(\sigma^{-1}v_i).
\end{align} Consequently, the following bounds on the expectation and variance hold: \begin{align}
C_1 \cdot \norm{v}_p^p &\leq E[T_p]  \leq C_2 \cdot \left(\norm{v}_p^p + d^{1-2k/p}\cdot \sigma^{p-2k} \cdot \norm{v}_{p}^{2k}\right)\\
\textand V[T_p] &\leq C_3 \cdot \left(d \cdot \sigma^{2p} + d^{2/p-1} \cdot \sigma^{2} \cdot \norm{v}_p^{2(p-1)} \right)
\end{align} Furthermore, for $p$ even, it holds that $E[T_p] = \norm{v}_p^p$.

\end{corollary}

\subsection{Bounding the chi-squared distance by moment differences}\label{sec:chi2_bound}

In this section, we recap known results for bounding the chi-squared distance between Gaussian mixtures by moment differences between their mixing distributions.

\begin{lemma}[\citet{wuMinimaxRatesEntropy2016}]\label{lemma:chi2bound}
Let $\pi_0$ be a centered distribution, i.e. $m_1(\pi_0) = 0$, supported on $[-\delta,\delta]$ that matches $L$ moments of $\pi_1$ \begin{equation}
m_l(\pi_0) = m_l(\pi_1) \for 1\leq l \leq L\period
\end{equation}  The chi-squared distance between their mixture distributions is bounded by the moment difference of the mixing distributions \begin{equation}
\chi^2\left(P_{\pi_1},P_{\pi_0}\right) \leq e^{\frac{\delta^2}{2\sigma^2}} \cdot \sum_{j\geq L+1} \frac{\Delta_j^2}{j!\sigma^{2j}}
\end{equation} where $\Delta_j = m_j(\pi_1)-m_j(\pi_0)$.
\end{lemma}

\begin{proof}[Proof of Lemma \ref{lemma:chi2bound}]

Let $\phi$ be the density of $\mathcal{N}(0,1)$\begin{equation} \phi(x) = \sqrt{\frac{1}{2\pi}} e^{-\frac{x^2}{2}},
\end{equation} and let $H_k$ be the $k$-th probabilist Hermite polynomial, it holds that \begin{equation}
E_{X \sim N(0,\sigma^2)}\left[H_{k}(X/\sigma)H_{j}(X/\sigma)\right]=\int_{-\infty}^\infty H_{k}(x/\sigma)H_{j}(x/\sigma) \phi(\frac{x}{\sigma})\frac{1}{\sigma}\ dx = k! \delta_{k,j}
\end{equation} and \begin{equation}
\sum_{j\geq0} H_j(x/\sigma) \cdot \frac{u^j}{j!\sigma^j} = e^{-\frac{u^2}{2\sigma^2}+\frac{u x}{\sigma^2}} = \frac{\phi(\frac{x-u}{\sigma})}{\phi(\frac{x}{\sigma})}.
\end{equation} Thus \begin{equation}
dP_{\pi}(x) = \int \frac{1}{\sigma}\phi(\frac{x-v}{\sigma})\ d\pi(v) = \frac{1}{\sigma}\phi(\frac{x}{\sigma}) \cdot \sum_{j\geq 0} H_j(x/\sigma) \cdot \frac{m_j(\pi)}{j!\sigma^j}
\end{equation} Finally, if a mixing distribution is supported on $[-\delta,\delta]$ and is centered, its density is bounded away from zero \begin{align}
dP_{\pi}(x) &= \int \frac{1}{\sigma}\phi(\frac{x-v}{\sigma})\ d\pi(v) \\
&=  \frac{1}{\sigma}\phi(\frac{x}{\sigma}) \cdot \int e^{-\frac{v^2}{2\sigma^2}+\frac{xv}{\sigma^2}}\ d\pi(v)\\
&\geq  \frac{1}{\sigma}\phi(\frac{x}{\sigma}) \cdot e^{E_{v \sim \pi}\left[-\frac{v^2}{2\sigma^2}+\frac{xv}{\sigma^2}\right]} &&\text{ by Jensen's inequality}\\
&= \frac{1}{\sigma}\phi(\frac{x}{\sigma}) \cdot e^{-\frac{\delta^2}{2\sigma^2}} &&\since m_1(\pi)=0 \textand \supp(\pi)\subseteq[-\delta,\delta]
\end{align} It follows that
\begin{align}
\chi^2\left(P_{\pi_1},P_{\pi_0}\right) &= \int \frac{\left(dP_{\pi_1}(x)-dP_{\pi_0}(x)\right)^2}{dP_{\pi_0}(x)}\\
&= \int \frac{1}{\sigma^2}\phi^2(\frac{x}{\sigma}) \frac{\left(\sum_{j\geq 0} H_j(x/\sigma) \cdot \frac{\Delta_j}{j!\sigma^j} \right)^2}{dP_{\pi_0}(x)}\\
&\leq e^{\frac{\delta^2}{2\sigma^2}} \cdot \int \frac{1}{\sigma}\phi(\frac{x}{\sigma}) \left(\sum_{j\geq 0} H_j(x/\sigma) \cdot \frac{\Delta_j}{j!\sigma^j} \right)^2\\
&= e^{\frac{\delta^2}{2\sigma^2}} \cdot \sum_{j\geq 0} \frac{\Delta_j^2}{j!\sigma^{2j}} &&\text{by orthogonality}\\
&= e^{\frac{\delta^2}{2\sigma^2}} \cdot \sum_{j\geq L+1} \frac{\Delta_j^2}{j!\sigma^{2j}} &&\text{matching L moments}
\end{align}
\end{proof}

\begin{corollary}\label{cor:chi2bound_for_simple_null}
Let $\pi_0 = \delta_0$, and $\pi_1$ be a symmetric mixing distribution supported on $[-\delta,\delta]$. Then it follows that \begin{equation}
\chi^2\left(P^d_{\pi_1},P^d_{0}\right) \leq \exp\left\{\frac{d}{2}\cdot \left(\frac{\delta}{\sigma}\right)^4\right\}-1\period
\end{equation}
\end{corollary}
\begin{proof}[Proof of Corollary \ref{cor:chi2bound_for_simple_null}]
\begin{align}
\chi^2\left(P_{\pi_1},P_{0}\right) &\leq e^{\frac{0}{2\sigma^2}} \cdot \sum_{j\geq 2} \frac{m_j(\pi_1)^2}{\sigma^{2j} \cdot j!} &&\text{by lemma \ref{lemma:chi2bound}}\\
&= \sum_{j\geq 1} \frac{m_{2j}(\pi_1)^2}{\sigma^{4j} \cdot (2j)!} &&\text{by symmetry, all odd moments vanish}\\
&\leq \sum_{j\geq 1} \frac{(\delta^2/\sigma^2)^{2j}}{(2j)!} &&\since \supp(\pi_1) \subseteq [-\delta,\delta]\\
&= \cosh\left(\delta^2/\sigma^2\right)-1\\
&\leq e^{\left(\frac{\delta}{\sigma}\right)^4\cdot \frac{1}{2}}-1.
\end{align} By tensorization, it follows that
\begin{equation}
\chi^2\left(P^d_{\pi_1},P^d_{0}\right) = \left(1+\chi^2\left(P_{\pi_1},P_{0}\right)\right)^d-1 \leq e^{\frac{d}{2}\cdot \left(\frac{\delta}{\sigma}\right)^4}-1.
\end{equation}
\end{proof}

\begin{lemma}\label{lemma:moment_matching_priors}
For $0 < \eta < 1$, and any $\delta\geq 0$ and $L\geq 1$ that satisfy \begin{equation}
\delta^2 \leq C_\eta \cdot \sigma^2 \cdot \frac{L}{d^{1/(L+1)}},
\end{equation} where $C_\eta$ is a universal positive constant that depends only on $\eta$. If  $\pi_0$ and $\pi_1$ are centered distributions supported on $[-\delta,\delta]$ that share the first $L$ moments, then the chi-squared distance between their corresponding mixtures is bounded $
\chi^2\left(P_{\pi_1}^d,P_{\pi_0}^d\right)\leq \eta$.
\end{lemma}\begin{proof}[Proof of Lemma \ref{lemma:moment_matching_priors}]
Note that
\begin{align}
\chi^2\left(P_{\pi_1},P_{\pi_0}\right) &\leq e^{\frac{\delta^2}{2}\cdot \sigma^{-2}} \cdot \sum_{j\geq L+1} \frac{\Delta_j^2\cdot \sigma^{-2j}}{j!} &&\by Lemma \ref{lemma:chi2bound}\\
&\leq e^{\frac{\delta^2}{2}\cdot \sigma^{-2}} \cdot \sum_{j\geq L+1} \frac{(4\delta^2\sigma^{-2})^j}{j!} &&\since \supp(\pi_1)\cup\supp(\pi_0) \subseteq [-\delta,\delta]\\
&\leq e^{\frac{\delta^2}{2}\cdot \sigma^{-2}} \cdot e^{4\delta^2\cdot \sigma^{-2}} \cdot \frac{(4\delta^2\sigma^{-2})^{L+1}}{(L+1)!} 
\\
&\leq C \cdot e^{\frac{\delta^2}{\sigma^{2}}} \cdot \left(\frac{\delta^2}{L\sigma^{2}}\right)^{L+1}
\end{align} for some universal positive constant $C\geq 1$. Recall that \begin{equation}
\chi^2\left(P^d_{\pi_1},P^d_{\pi_0}\right) =(1+\chi^2\left(P_{\pi_1},P_{\pi_0}\right))^d-1 \leq \exp\left\{\chi^2\left(P_{\pi_1},P_{\pi_0}\right) \cdot d\right\}-1.
\end{equation} Thus, $\chi^2\left(P^d_{\pi_1},P^d_{\pi_0}\right) \leq \eta$ whenever \begin{equation}
\chi^2\left(P_{\pi_1},P_{\pi_0}\right) \leq \frac{\tilde{\eta}}{d},
\end{equation} where $\tilde{\eta} = \log(1+\eta)$. It is sufficient to choose $\delta$ and $L$ such that \begin{equation}
C\cdot e^{\frac{\delta^2}{\sigma^2} } \cdot \left(\frac{\delta^2}{\sigma^2L}\right)^{L+1} \leq \frac{\tilde{\eta}}{d}
\end{equation} which is implied by \begin{equation}
e^{\frac{\delta^2}{\sigma^2L}} \cdot \frac{\delta^2}{\sigma^2L} \leq \left(\frac{\tilde{\eta}}{C}\right)^{1/(L+1)} \cdot d^{-1/(L+1)}.
\end{equation} That condition is satisfied whenever \begin{equation}
\frac{\delta^2}{\sigma^2L} \leq 1 \wedge \left(\frac{\tilde{\eta}}{C}\right)^{1/(L+1)} \cdot d^{-1/(L+1)},
\end{equation} which is satisfied by the assumptions since $\left(\frac{\tilde{\eta}}{C}\right)^{1/(L+1)} \cdot d^{-1/(L+1)}  \leq 1$.
\end{proof}

\subsection{Lower bound on the critical separation under \texorpdfstring{$\ell_p$}{lp} norms for \texorpdfstring{$p \in (0,2]$}{p in (0,2]} in the free tolerance regime}\label{lb_p_less_than_2_small_e0}

In this section, we present a simple lower bound on the critical separation in the free-tolerance regime.

\begin{proposition}\label{prop:free_tolerance_lb} For $0 < p \leq 2$, the critical separation for hypotheses \eqref{eq:gaussian_testing_lp} is lower-bounded by \begin{equation}
\epsilon_1^*(\epsilon_0,\GS)-\epsilon_0 \gtrsim d^{1/p-1/4} \cdot \sigma \quad \for  0 \leq \epsilon_0 \lesssim d^{1/2p} \cdot \sigma.
\end{equation} 
\end{proposition}
\begin{proof}[Proof of Proposition \ref{prop:free_tolerance_lb}]
Consider the following symmetric mixing distributions \begin{equation}
\pi_0 = \delta_0 \textand \pi_1 = \frac{1}{2}\left(\delta_{\epsilon \cdot d^{-1/p}}+\delta_{-\epsilon \cdot d^{-1/p}}\right)
\end{equation} Note that they share their first moment, and their separation is given by $\epsilon$ \begin{equation}
\norm{X}_p^p = 0 \text{ a.s. under } \pi_0\quad   \textand \quad \norm{X}_p^p=\epsilon^p \text{ a.s. under } \pi_1.
\end{equation} By Corollary \ref{cor:chi2bound_for_simple_null}, it holds that \begin{equation}
\chi^2\left(P_{\pi_1}^d,P_{\pi_0}^d\right)\leq \exp\left[\frac{1}{2} \cdot \left(\frac{\epsilon}{\sigma \cdot d^{1/p-1/4}}\right)^{4}\right]-1.
\end{equation} Thus, choosing \begin{equation}
\epsilon^p = C_p \cdot  d^{1-p/4} \cdot \sigma^p \where C_p = \left[2\log\left(1+(C_\alpha/2)^2\right)\right]^{p/4}\comma
\end{equation} the chi-squared distance is bounded: $
\chi^2\left(P_{\pi_1}^d,P_{\pi_0}^d\right)\leq (C_\alpha/2)^2$. Hence by Lemma \ref{lemma:twofuzzypriors}, it follows that \begin{equation}
\epsilon_1^*(\epsilon,\GS)-\epsilon_0 \gtrsim  d^{1/p-1/4} \cdot \sigma \quad \for  0 \leq \epsilon_0 \lesssim d^{1/p-1/4} \cdot \sigma^p\comma
\end{equation} which implies the statement of Proposition \ref{prop:free_tolerance_lb} since $d^{1/p-1/4}\geq d^{1/2p}$ for $p \leq 2$.
\end{proof}

\subsection{Conjecture for the lower bound on the critical separation under non-smooth \texorpdfstring{$\ell_p$}{lp} norms in the interpolation regime}\label{sec:interpolation_lb_lp}

Theorem \ref{thm:unconstrained_moment_matching_lowerbound} can be generalized such that any lower bound on the constrained moment-matching problem \eqref{eq:Mpepsilon} leads to a lower bound on the critical separation. We defer the proof to Theorem \ref{thm:LB_via_opt} in Section \ref{sec:moment_matching_lower_bounds}.

\begin{theorem}\label{thm:constrained_moment_matching_lowerbound}
Choose $\delta\geq 0$ and $L\geq1$ satisfying \eqref{eq:chi2_condition}. The critical separation for hypotheses \eqref{eq:gaussian_testing_lp} is lower bounded by \begin{equation}
\epsilon_1^*(\epsilon_0,\GS)-\epsilon_0 \gtrsim b_p(\epsilon_0) \quad \where b_p(\epsilon_0) = d^{1/p}\delta\cdot M^{1/p}_p(\tilde{\epsilon}_0,L) \textand \tilde{\epsilon}_0 \asymp \frac{\epsilon_0}{d^{1/p}\delta}
\end{equation} for $0 < \epsilon_0 \lesssim b_p(\epsilon_0) \textand  d^{1/2} \gtrsim M^{-1}_p(\tilde{\epsilon}_0,L)$.
\end{theorem}

We conjecture that the constrained moment-matching problem satisfies the following lower bound whenever $\epsilon$ is not too large.

\begin{conjecture} For $0<\epsilon \lesssim 1/L$, it holds that:
\begin{equation}
M_p^{1/p}(\epsilon,L) \gtrsim \begin{dcases}
\frac{\epsilon^{p/2}}{L^{1-p/2}} &\textif 1 \leq p < 2\\
\frac{\epsilon^{2k/p}}{L^{1-2k/p}} &\textif 2k < p < 2(k+1) \text{ with } k \in \mathbb{Z}_+
\end{dcases}
\end{equation}
\end{conjecture} Under the above conjecture, the following lower-bound on the critical separation would hold via  Theorem \ref{thm:constrained_moment_matching_lowerbound}: \begin{equation}
\epsilon_1^*(\epsilon_0,\GS)-\epsilon_0\gtrsim \begin{dcases}
\left(d^{1/p} \cdot \frac{\delta}{L}\right)^{1-p/2} \cdot \epsilon_0^{p/2} &\textif 1 \leq p < 2\\
\left(d^{1/p} \cdot \frac{\delta}{L}\right)^{1-2k/p} \cdot \epsilon_0^{2k/p} &\textif 2k < p < 2(k+1) \text{ with } k \in \mathbb{Z}_+
\end{dcases}
\end{equation} Thus, matching rates for the interpolation regimes in Lemma \ref{lemma:upper_bound_lp_odd_less_2} and Lemma \ref{lemma:upper_bound_lp_odd_gtr_2}, would follow for any $d$ large enough by choosing
$L \asymp \log d$ and $\delta \asymp \sigma \cdot \sqrt{L}$.

\section{General arguments for lower bounding the critical separation}

\subsection{Two fuzzy hypotheses}

The following is a variant of the well-known fuzzy hypotheses techniques \citep{tsybakovIntroductionNonparametricEstimation2009}.

\begin{lemma}[Indistinguishable fuzzy hypotheses \citep{tsybakovIntroductionNonparametricEstimation2009}.]\label{lemma:twofuzzypriors}
Let the hypothesis sets be \begin{equation}
V_0 = \{v : \norm{v}_p\leq \epsilon_0\} \textand V_1 = \{v : \norm{v}_p\geq \epsilon_1\}
\end{equation} and $C_\alpha = 1-(\alpha+\beta)$. Consider mixing distributions concentrated in such sets \begin{equation}
\pi_{0}^d(V_0) \geq 1 - \frac{C_\alpha}{4} \textand \pi_{1}^d(V_1) \geq 1 - \frac{C_\alpha}{4}
\end{equation} whose
induced mixtures are close in total variation or chi-squared distance \begin{equation}
V(P_{\pi_0}^d,P_{\pi_1}^d)<\frac{C_\alpha}{2} \textor \chi^2(P_{\pi_0}^d,P_{\pi_1}^d)<\left[\frac{C_\alpha}{2}\right]^2\period
\end{equation} Then, the minimax risk \eqref{eq:minimax_risk} is lower bounded by $\beta$: \begin{equation}
R_*(\epsilon_0,\epsilon_1,\GS)> \beta.
\end{equation}
\end{lemma}

\begin{proof}[Proof of lemma \ref{lemma:twofuzzypriors}]

Let $\beta_0=\beta_1=C_\alpha/4$. Consider the restriction of the priors to their hypothesis sets \begin{equation}
\tilde{\pi}_{i}(B) = \frac{\pi_i(B \cap V_i)}{\pi_i(V_i)}.
\end{equation}

The risk of any valid test is at least \begin{align}
R(\epsilon_0,\epsilon_1,\psi,\GS) &= \sup_{v \in V_1} E_{P^d_v}[1-\psi]\\
&\geq \sup_{v \in V_0}E_{P^d_v}[\psi] + \sup_{v \in V_1} E_{P^d_v}[1-\psi] - \alpha\\
&\geq \int E_{P^d_v}[\psi]\ d\tilde{\pi}_0^d + \int E_{P^d_v}[1-\psi]\ d\tilde{\pi}_1^d- \alpha\\
&= \int \psi\ dP_{\tilde{\pi}_0}^d + \int [1-\psi]\ dP_{\tilde{\pi}_1}^d- \alpha\\
&\geq \int [\psi+1-\psi]\ d(P_{\tilde{\pi}_0}^d \wedge P_{\tilde{\pi}_1}^d)- \alpha\\
&= 1 - V(P_{\tilde{\pi}_0}^d,P_{\tilde{\pi}_1}^d) - \alpha
\end{align} where the last equality is by Sheffe's theorem. Furthermore, by the triangle inequality and the data processing inequality, it follows by the triangle inequality
\begin{align}
V(P_{\tilde{\pi}_0}^d,P_{\tilde{\pi}_1}^d) \leq V(P_{\tilde{\pi}_0}^d,P_{\pi_0}^d) + V(P_{\tilde{\pi}_1}^d,P_{\pi_1}^d) + V(P_{\pi_0}^d,P_{\pi_1}^d) .\end{align} Furthermore, by the data processing inequality, we have that
\begin{align}
V(P_{\tilde{\pi}_0}^d,P_{\tilde{\pi}_1}^d)  &\leq V(\tilde{\pi}_0^d,\pi_0^d) + V(\tilde{\pi}_1^d,\pi_1^d) + V(P_{\pi_0}^d,P_{\pi_1}^d) \\
&\leq V(P_{\pi_0}^d,P_{\pi_1}^d) + \beta_0 + \beta_1,
\end{align} where in the last step, we used the fact that \begin{equation}
V(\tilde{\pi}_0^d,\pi_0^d) = \pi_0^d(V_0^c)=1-\pi_0^d(V_0) \leq \beta_0.
\end{equation} Consequently, the risk of any valid test is lower-bounded by $\beta$ \begin{equation}
R_*(\epsilon_0,\epsilon_1,\GS) = \inf_{\psi} R(\epsilon_0,\epsilon_1,\psi,\GS) \geq 1 - V(P_{\pi_0}^d,P_{\pi_1}^d) - (\beta_0+\beta_1)- \alpha>\beta.
\end{equation} Thus, the statement of the lemma follows.
\end{proof}

\begin{lemma}[Difference in means lower bounds critical separation]\label{lemma:lb_by_two_fuzzy_prios}

Let $s^p_p(\epsilon_0,\GS)$ denote a critical separation for hypotheses \eqref{eq:gaussian_testing_lp} \begin{equation}\label{eq:spGS}
s^p_p(\epsilon_0,\GS) = \inf\left\{\epsilon_1^p-\epsilon_0^p : \epsilon_1\geq\epsilon_0 \textand R_*(\epsilon_0,\epsilon,\GS)<\beta\right\}.
\end{equation}
Let $\pi_0$ be a distribution that satisfies \begin{equation}\label{eq:cond_pi0}
E_{v\sim\pi^d_0}\norm{v}_p^p\leq C_0 \cdot \epsilon_0^p,
\end{equation} where $C_0=C_\alpha/4$ and $C_\alpha=1-(\alpha+\beta)$. Furthermore, let $\pi_1$ be a distribution that satisfies \begin{equation}\label{eq:cond_pi1}
E_{v\sim \pi_1^d}\norm{v}_p^p \geq  C_0^{-1/2}(1-C_1)^{-1} \cdot \sqrt{V_{v\sim \pi_1^d}\norm{v}_p^p}
\end{equation} for some $C_1 \in (0,1)$ and whose mixture measure $P_{\pi_1}$ is close in total variation or chi-squared distance to $P_{\pi_0}$: $
V(P_{\pi_0}^d,P_{\pi_1}^d)<C_\alpha/2$ or $\chi^2(P_{\pi_0}^d,P_{\pi_1}^d)<\left[C_\alpha/2\right]^2$.

Define the hypotheses \begin{equation}
H_0: v \in V_0 \vs H_1: v \in V_1
\end{equation} where \begin{equation}
V_0 = \{v : \norm{v}_p \leq \epsilon_0\},\ V_1 = \{v : \norm{v}_p\geq \epsilon_1\} \textand \epsilon_1^p = C_1 \cdot E_{\pi_1^d}\norm{v}_p^p
\end{equation} It follows that $s^p_p(\epsilon_0,\GS)$ is lower bounded \begin{equation}
s^p_p(\epsilon_0,\GS)\geq \epsilon_1^p-\epsilon_0^p\geq (C_1-\tilde{C}_1)\cdot E_{v\sim \pi_1^d} \norm{v}_p^p \geq  (C_1-\tilde{C}_1)\cdot \left(E_{v\sim \pi_1^d}\norm{v}_p^p-E_{v\sim \pi_0^d}\norm{v}_p^p\right)
\end{equation} for $
C_0^{-1}\cdot E_{v\sim\pi^d_0}\norm{v}_p^p \leq \epsilon_0^p \leq \tilde{C}_1\cdot E_{v\sim\pi_1^d}\norm{v}_p^p
$ where $0 \leq \tilde{C}_1 < C_1$.

\end{lemma}

\begin{proof}

The probability that $\pi_0$ is supported on $V_0$ is bounded from below:
\begin{align}
\pi_0^d(V_0) &= \pi_0^d(\norm{v}_p \leq \epsilon_0 )\\
&= \pi_0^d(\norm{v}_p^p \leq \epsilon_0^p )\\
&\geq 1 - \frac{E_{\pi_0^d}\norm{v}_p^p}{\epsilon_0^p} &&\text{By Markov's inequality}\\
&\geq 1 - C_0 &&\by \eqref{eq:cond_pi0}\\
&\geq 1 - \frac{C_\alpha}{4}.
\end{align}

Additionally, the probability that $\pi_1$ is supported on $V_1$ is also bounded from below
\begin{align}
\pi_1^d(V_1) &= \pi_1^d(\norm{v}_p \geq \epsilon_1) \\
&= \pi_1^d(\norm{v}_p^p \geq \epsilon_1^p) \\
&= 1 - \pi_1^d(\norm{v}_p^p \leq \epsilon_1^p )\\
&= 1 - \pi_1^d(\norm{v}_p^p - E_{\pi^d}\norm{v}_p^p  \leq \epsilon_1^p - E_{\pi_1^d}\norm{v}_p^p )\\
&\geq  1 - \pi_1^d(|\norm{v}_p^p - E_{\pi_1^d}\norm{v}_p^p|  \geq  E_{\pi_1^d}\norm{v}_p^p - \epsilon_1^p )\\
&= 1 - \pi_1^d(|\norm{v}_p^p - E_{\pi_1^d}\norm{v}_p^p|  \geq  (1-C_1)\cdot E_{\pi_1^d}\norm{v}_p^p) &&\since \epsilon_1^p = C_1 \cdot E_{\pi_1^d}\norm{v}_p^p\\
&\geq 1 - \frac{V_{\pi_1^d}\norm{v}_p^p}{\left((1-C_1)\cdot E_{\pi_1^d}\norm{v}_p^p\right)^2} &&\text{by Markov's inequality}\\
&\geq 1 - \frac{C_\alpha}{4}  &&\by \eqref{eq:cond_pi1}
\end{align} The claim of Lemma \ref{lemma:lb_by_two_fuzzy_prios} follows by Lemma \ref{lemma:twofuzzypriors}.
\end{proof}

\subsection{Moment-matching lower-bounds}\label{sec:moment_matching_lower_bounds}

The following theorem enables using a lower bound for a moment-matching problem to derive a lower bound on the critical separation. The theorem depends on Lemma \ref{lemma:centering} and Corollary \ref{cor:scaling}, which are stated and proved at the end of the section.

\begin{theorem}[Moment-matching problems lower bound critical separation]\label{thm:LB_via_opt}
Let $s^p_p(\epsilon_0,\GS)$ denote a critical separation for hypotheses \eqref{eq:gaussian_testing_lp}, see \eqref{eq:spGS} for a definition.

Choose $\delta\geq 0$ and $L\geq1$ such that for any centered distributions $\pi_0$ and $\pi_1$ supported on $[-\delta,\delta]$ that share the first $L$ moments, the total variation between their mixtures is bounded $V(P_{\pi_0}^d,P_{\pi_1}^d)\leq C_\alpha$. It holds that \begin{equation}\label{eq:lb1}
s_p^p(\epsilon_0,\GS) \gtrsim d\delta^p\cdot M_p(\tilde{\epsilon}_0,L) \quad \where \tilde{\epsilon}_0 \asymp \frac{\epsilon_0}{d^{1/p}\delta}
\end{equation} for $0 < \epsilon_0^p \lesssim d\delta^p\cdot M_p(\tilde{\epsilon}_0,L)$ and $d^{1/2} \gtrsim M^{-1}_p(\tilde{\epsilon}_0,L)$. Furthermore, if $M_p(L)>0$, then \begin{equation}\label{eq:lb2}
s_p^p(\epsilon_0,\GS) \gtrsim \delta^p d \cdot M_p(L)
\end{equation} for $\epsilon_0^p \asymp \delta^p\cdot d \cdot M_p(L)$ and $d^{1/2} \gtrsim M^{-1}_p(L)$. Finally, we remark that for hypotheses \eqref{eq:gaussian_testing_lp} it holds that \begin{equation}
s_p^p(\epsilon_0,\GS) \asymp \left(\epsilon_1^*(\epsilon_0,\GS)\right)^p \asymp \left(\epsilon_1^*(\epsilon_0,\GS)-\epsilon_0\right)^p \quad \for \epsilon_0 \leq \frac{\epsilon_1}{2}.
\end{equation} Therefore we can replace $s_p^p(\epsilon_0,\GS)$ for $\left(\epsilon_1^*(\epsilon_0,\GS)-\epsilon_0\right)^p$ in \eqref{eq:lb1} and \eqref{eq:lb2}.
\end{theorem}
\begin{proof}

\textbf{Constrained moment-matching.} Let $\pi_0$ and $\pi_1$ be the solutions of  $\tilde{M}_p(\tilde{\epsilon}_0,L,\delta)$ \eqref{eq:MpEpsilonCentredScaled} where $\tilde{\epsilon}_0=\epsilon_0 \cdot C_0^{1/p}$. Let $\tilde{C}_1=C_1/2$ for some $0 < C_1 < 1$. Using the facts that $\tilde{M}_p(\tilde{\epsilon}_0,L,\delta)= E_{\pi_1^d}\norm{v}_p^p$ and $V_{\pi_1^d}\norm{v}_p^p \leq \delta^{2p}\cdot d$, Lemma \ref{lemma:lb_by_two_fuzzy_prios} implies \begin{equation}
s_p^p(\epsilon_0,\GS) \geq \frac{C_1}{2} \cdot \tilde{M}_p(\tilde{\epsilon_0},L,\delta)
\end{equation} for \begin{equation}
\epsilon_0^p \leq \frac{C_1}{2} \cdot \tilde{M}_p(\tilde{\epsilon}_0,L,\delta)
\end{equation} and \begin{equation}
\tilde{M}_p(\tilde{\epsilon_0},L,\delta) \geq  C_0^{-1/2}(1-C_1)^{-1} \cdot d^{1/2}\delta^p
\end{equation} The statement follows by Corollary \ref{cor:scaling}.

\textbf{Unconstrained moment-matching.} Let $\pi_0$ and $\pi_1$ be the solutions of  $\tilde{M}_p(L,\delta)$ \eqref{eq:MpCentredScaled}. If $\tilde{M}_p(L,\delta)>0$, since $\tilde{M}_p(L,\delta)=E_{\pi_1^d}\norm{v}_p^p-E_{\pi_0^d}\norm{v}_p^p$, it follows that \begin{equation}
\exists\ 0<C<1 \st E_{\pi_0^d}\norm{v}_p^p = C \cdot E_{\pi_1^d}\norm{v}_p^p\period
\end{equation} Let $\tilde{C}_1=C$ and $C_1=C + \gamma$ for some $0 < \gamma < 1-C$. Using the facts that $\tilde{M}_p(L,\delta)\leq E_{\pi_1^d}\norm{v}_p^p$ and $V_{\pi_1^d}\norm{v}_p^p \leq \delta^{2p}\cdot d$, Lemma \ref{lemma:lb_by_two_fuzzy_prios} implies \begin{equation}
s_p^p(\epsilon_0,\GS) \geq \gamma \cdot \tilde{M}_p(L,\delta)
\end{equation} for \begin{equation}
\epsilon_0^p = C_0^{-1} \cdot E_{\pi_0^d}\norm{v}_p^p = C_0^{-1} \cdot C\cdot (1-C)^{-1} \cdot \tilde{M}_p(L,\delta)
\end{equation} and \begin{equation}
\tilde{M}_p(L,\delta) \geq  C_0^{-1/2}(1-C_1)^{-1} \cdot d^{1/2}\delta^p.
\end{equation} The statement follows by Corollary \ref{cor:scaling}.

\end{proof}

We finish the section by proving the auxiliary Lemma \ref{lemma:centering} and Corollary \ref{cor:scaling}.

\begin{lemma}[Centered moment-matching problems]\label{lemma:centering}

Define the centered versions of problem \eqref{eq:Mpepsilon} \begin{align}\label{eq:Mpepsiloncentred}
\tilde{M}_p(\epsilon,L) =\quad \sup_{\pi_0,\pi_1} E_{\pi_1}|v|^p \st & E_{\pi_0}|v|^p \leq \epsilon^p \\
& m_1(\pi_0)=m_1(\pi_1)=0\\
& m_l(\pi_0)=m_l(\pi_1) \for 1\leq l\leq L\\
& \pi_0 \textand \pi_1 \text{ are prob. measures supp. on } [-1,1]
\end{align} and \eqref{eq:Mp} \begin{align}\label{eq:Mpcentred}
\tilde{M}_p(L) =\quad \sup_{\pi_0,\pi_1} E_{\pi_1}|v|^p-E_{\pi_0}|v|^p \st & m_1(\pi_0)=m_1(\pi_1)=0\\
& m_l(\pi_0)=m_l(\pi_1) \for 1\leq l\leq L\\
& \pi_0 \textand \pi_1 \text{ are prob. measures supp. on } [-1,1]
\end{align} It holds that \begin{equation}
\tilde{M}_p(\epsilon,L) = M_p(\epsilon,L) \textand \tilde{M}_p(L) = M_p(L)
\end{equation}
\end{lemma}
\begin{proof}[Proof of Lemma \ref{lemma:centering}]

We do the proof for $M_p(\epsilon, L)$, since the proof for $M_p(L)$ is analogous. The solution of \eqref{eq:Mpepsiloncentred} is upper bounded by \eqref{eq:Mpepsilon} \begin{equation}
\tilde{M}_p(\epsilon,L) \leq M_p(\epsilon,L)
\end{equation} since we are removing a constraint.

Let $(\pi_0,\pi_1)$ be any feasible solution of \eqref{eq:Mpepsilon}, we can define the symmetrized measures $(\tilde{\pi}_0,\tilde{\pi}_1)$ as \begin{equation}
X_i \cdot \epsilon_i \sim \tilde{\pi}_i \where X_i \sim \pi_i \textand \epsilon_i \sim \text{Unif}\{-1,1\}
\end{equation} Note that due to the symmetrization, $(\tilde{\pi}_0,\tilde{\pi}_1)$ are centred, and consequently, all their odd moments vanish. Furthermore, all the absolute and even moments remain untouched \begin{equation}
E_{\pi_i}X^{2k}=E_{\tilde{\pi}_i}X^{2k} \for k \in \Zplus \textand E_{\pi_i}|X|^p=E_{\tilde{\pi}_i}|X|^p
\end{equation} Thus, any feasible solution of \eqref{eq:Mpepsilon} can be transformed into a feasible solution of \eqref{eq:Mpepsiloncentred}. Ergo, \begin{equation}
M_p(\epsilon,L)  \leq \tilde{M}_p(\epsilon,L)
\end{equation} and the lemma is proved.
\end{proof}

\begin{corollary}\label{cor:scaling}

Define the high-dimensional, centered, and scaled versions of problems \eqref{eq:Mpepsilon} and \eqref{eq:Mp}.

\begin{align}\label{eq:MpEpsilonCentredScaled}
\tilde{M}_p(\epsilon,L,\delta) =\quad \sup_{\pi_0,\pi_1}& E_{\pi_1^d}\norm{v}_p^p\\
\st & E_{\pi_0^d}\norm{v}_p^p \leq \epsilon^p \\
& m_1(\pi_0)=m_1(\pi_1)=0\\
& m_l(\pi_0)=m_l(\pi_1) \for 1\leq l\leq L\\
& \pi_0 \textand \pi_1 \text{ are probability measures supported on } [-\delta,\delta]
\end{align}

\begin{align}\label{eq:MpCentredScaled}
\tilde{M}_p(L,\delta) =\quad \sup_{\pi_0,\pi_1}& E_{\pi_1^d}\norm{v}_p^p-E_{\pi_0^d}\norm{v}_p^p\\
\st & m_1(\pi_0)=m_1(\pi_1)=0\\
& m_l(\pi_0)=m_l(\pi_1) \for 1\leq l\leq L\\
& \pi_0 \textand \pi_1 \text{ are probability measures supported on } [-\delta,\delta]
\end{align}

It holds that \begin{equation}
\tilde{M}_p(\epsilon,L,\delta) = d\cdot\delta^p\cdot M_p\left(\frac{\epsilon}{d^{1/p}\delta},L\right) \textand \tilde{M}_p\left(L,\delta\right) = d\cdot\delta^p\cdot M_p(L)
\end{equation}
\end{corollary}\begin{proof}
The statement follows by re-scaling the measures in Lemma \ref{lemma:centering} and using the fact that every dimension has the same distribution.
\end{proof}

\subsubsection{Unconstrained moment-matching bounds}\label{sec:unconstrained_moment_matching}

In the following, we prove Lemma \ref{lemma:MpOdd}, which follows from section 5.3.2 of \cite{hanEstimationL_Norms2020}.

\begin{lemma*}[Restated Lemma \ref{lemma:MpOdd}] Let $p\geq 1$ non-even and $L \geq 1$, then \begin{equation}
C_p \cdot L^{-p} \leq M_p(L) \leq \tilde{C}_p \cdot L^{-p}
\end{equation} where $C_p$ and $\tilde{C}_p$  are positive constants that depends only on $p$.
\end{lemma*}\begin{proof}[Proof of Lemma \ref{lemma:MpOdd}]

\textbf{Upper bound.} The upper bound follows directly as a corollary of the duality between moment matching and polynomial approximation in Lemma \ref{lemma:moment_matching_duality} together with the following result on polynomial approximation. \begin{lemma*}[\cite{bernsteinOrdreMeilleureApproximation1912},\cite{vargaBernsteinConjectureApproximation1985}]
Let $p > 0$ and $L \geq 1$, it holds that
$
A_p(L) \leq \beta_p \cdot L^{-p}$, where $\beta_p$ is a positive constant that depends only on $p$.
\end{lemma*}

\textbf{Lower bound.} By Lemma 5.6 of \citet{hanEstimationL_Norms2020}, there exists $(\nu_0,\nu_1)$ non-negative measures  that match $\ceil{L}+q$ moments, where $q=\ceil{p/2}$, are supported on $[c/n^2,1]$, and they satisfy \begin{equation}
\int f(x)\ d\nu_1(x)-\int f(x)\ d\nu_0(x) \geq  c' \cdot L^{2q-p} \quad \for f(x)=x^{-q+r/2},
\end{equation} where $c \in (0,1)$ and $c'$ are positive constants that depend only on $(q,p)$.

Define the measures  \begin{equation}
\tilde{v}_i(dx) = \left(1 - E_{v_i}\left[\frac{c}{L^2 \cdot X}\right]^q\right) \cdot \delta_0(dx) +  \left[\frac{c}{L^2 \cdot x}\right]^q \cdot v_i(dx)
\end{equation} $(\tilde{v}_0,\tilde{v}_1)$ are probability measures supported on $[0,1]$ since \begin{equation}
E_{v_i}\left[\frac{c}{L^2 \cdot X}\right]^q \leq 1
\end{equation} due to $X \geq c/\ceil{L}^2$ almost surely for $X \sim v_i$. Furthermore, the distributions match moments \begin{align}
m_l(\tilde{v}_1) = \frac{c^{q}}{L^{2q}} \cdot E_{X \sim v_i}[X^{l-q}] = m_l(\tilde{v}_0) \for 1 \leq l \leq \ceil{L}+q.
\end{align}  They are far away when measured by $f$ \begin{equation}
E_{\tilde{v}_1}[X^{p/2}]-E_{\tilde{v}_0}[X^{p/2}] = \frac{c^{q}}{L^{2q}} \cdot \left(E_{v_1}[f(X)]-E_{v_0}[f(X)]\right) \geq \frac{c^{q}c'}{L^{p}},
\end{equation} and they have bounded q-moments \begin{equation}
m_{q}(\tilde{v}_i) = \frac{c^{q}}{L^{2q}} \cdot m_{0}(v_i) = \frac{c^{q}}{L^{2q}}.
\end{equation} We symmetrize and scale the measures. Let $\mu_i$ be such that \begin{equation}
\epsilon_i \cdot X_i^{1/2} \sim \mu_i \where X_i \sim \tilde{v}_i \textand \epsilon_i \sim \text{Unif}\{-1,1\}
\end{equation} Then $(\mu_1,\mu_2)$ are symmetric probability distributions supported on the $[-1,1]$ interval. It follows that all their odd moments vanish, and their even moments match up to $\ceil{L}+q$  \begin{equation}
m_l(\mu_1) =  m_{l/2}(\tilde{v}_i) = m_l(\mu_0) \for 1\leq l/2 \leq \ceil{L}+q \textand l \text{ even}.
\end{equation} They are far away when measured by $x\to|x|^p$ \begin{equation}
E_{\mu_1}|X|^p - E_{\mu_0}|X|^p =  \left(E_{\tilde{v}_1}X^{p/2} - E_{\tilde{v}_0}X^{p/2}\right) \geq c^{q}c' \cdot L^{-p}\period
\end{equation} Thus, the statement of Lemma \ref{lemma:MpOdd} follows.

\end{proof}

\subsubsection{Constrained moment-matching bounds}\label{sec:constrained_moment_matching}

In this section, we provide a simplified proof of \cite{canonnePriceToleranceDistribution2021}'s main contribution.

\begin{proof}[Proof of Lemma \ref{lemma:M1Epsilon}]

\textbf{The dual of $M_1(\epsilon_0,L)$.} Consider the dual of $M_1(\epsilon_0,L)$ \eqref{eq:Mpepsilon} \begin{align}\label{eq:dual}
\inf_{\alpha,z_1,z_2,,\lambda} \epsilon_0 \cdot \alpha + z_1 + z_2 \st & z_1 + \sum_{l=1}^L \lambda_l x^l \geq |x| &&\for |x|\leq 1\\
&\alpha |x| \geq \sum_{l=1}^L \lambda_l x^l-z_2  &&\for |x|\leq 1\\
& \alpha \geq 0
\end{align} where the $\lambda$'s correspond the moment-matching restrictions, the $z$'s guarantee that the $(\pi_0,\pi_1)$ integrate to one, and $\alpha$ corresponds to the upper-bound on $\pi_0$. Lemma 3.8 of \cite{canonnePriceToleranceDistribution2021} states that the above dual and the primal have the same value. We rewrite the dual for simplicity as \begin{align}
\inf_{\alpha\geq 0,z_1,z_2,\lambda} \epsilon_0 \cdot \alpha + z_1 + z_2 \st &\alpha |x| \geq \sum_{l=1}^L \lambda_l x^l-z_2 \geq |x| - (z_1+z_2)  &&\for |x|\leq 1.
\end{align} The condition of the dual indicates that $z_1+z_2\geq 0$. We can further lower-bound the problem by considering all polynomials of order $L$. \begin{align}\label{eq:simplified_dual}
\inf_{\alpha\geq 0,z\geq 0,p \in P_L} \epsilon_0 \cdot \alpha + z \st &\alpha |x| \geq p(x) \geq |x| - z  &&\for |x|\leq 1.
\end{align} Let $(\alpha_*,z_*,p_*)$ denote the optimal solution of the above optimization. Note that \begin{equation}
(\tilde{\alpha},\tilde{z},p_*) \where \tilde{\alpha}\epsilon_0 = \tilde{z} = \alpha_*\epsilon_0 \lor z_*
\end{equation} is a feasible solution whose achieved value is at most $2$ times the optimal solution of \eqref{eq:simplified_dual} \begin{equation}
\epsilon_0 \cdot \tilde{\alpha} + \tilde{z} = 2 \cdot \tilde{z} \leq 2\cdot\left[\epsilon_0 \cdot \alpha_* + z_*\right].
\end{equation} Consequently, the optimal solution of the following restricted optimization
\begin{align}\label{eq:restricted_dual}
\inf_{\alpha\geq 0,p \in P_L(\alpha)} \alpha\epsilon_0 \where
P_L(\alpha)=\left\{p \in P_L : |x| - \alpha\epsilon_0 \leq p(x) \leq \alpha|x|\quad  \for |x|\leq 1\right\}
\end{align} is at most $2$ times the optimal solution of \eqref{eq:simplified_dual} since  \begin{align}
\alpha'\epsilon_0
\leq \tilde{\alpha}\epsilon_0
\end{align} where $(\alpha',p')$ be the optimal solution of \eqref{eq:restricted_dual}. Putting it all together, the restricted dual problem \eqref{eq:restricted_dual} is at most $2$ times the value of the solution to the original primal problem $M_1(\epsilon_0,L)$ \eqref{eq:Mpepsilon} .

\textbf{$\alpha \geq 1 + C/L$ for $C$ small enough.} Let $(\alpha,p)$ be the solution of \eqref{eq:restricted_dual}. Note that under the assumption $
\epsilon_0 \leq C/L$ it must hold that $\alpha$ is strictly positive since by the constraints of the dual we have that \begin{equation}
\alpha \geq \sup_{|x|\leq 1}\frac{|x|}{|x| + \epsilon_0} = \frac{1}{1+\epsilon_0} \geq \frac{1}{1+C/L}
\end{equation} Furthermore, consider $L \geq C$, then evaluating the ratio at $x = \epsilon_0 \cdot \frac{L-C}{C}$, we get \begin{equation}\label{eq:alpha_lb}
\alpha \geq 1 - C/L
\end{equation} Consider any $0 < C \leq \bar{\beta}/(1+\bar{\beta}) $ where \begin{align}
\bar{\beta} &= \inf\left\{\beta : \inf_{p \in P_L}\sup_{|x|\leq 1}|p(x)-|x|| \leq \frac{\beta}{L} \right\},
\end{align} it follows that $\alpha \geq 1 + \bar{\beta}/L$. We argue by contradiction, that $\alpha \not\in (1-\bar{\beta}/L,1 + \bar{\beta}/L)$. If there exists $\alpha \in (1-\bar{\beta}/L,1 + \bar{\beta}/L) $ optimal solution, then there exists \begin{align}
p \in P_L(\alpha) &\implies \sup_{|x|\leq 1}|p(x)-|x||\leq \alpha\epsilon_0 \lor |\alpha-1| < \left(1+\frac{\bar{\beta}}{L}\right)\epsilon_0 \lor \frac{\bar{\beta}}{L}\\
&\implies  \sup_{|x|\leq 1}|p(x)-|x|| < \frac{\bar{\beta}}{L} \label{eq:low_poly_error}
\end{align} where in the last line we used the fact that $
\epsilon_0 \leq \frac{\bar{\beta}}{1+\bar{\beta}/L}\cdot \frac{1}{L} $. However, \eqref{eq:low_poly_error} is absurd since no polynomial of order $L$ can achieve an error smaller than $\frac{\bar{\beta}}{L}$.

\textbf{\cite{canonnePriceToleranceDistribution2021}'s proof.}  Let $(\alpha,p)$ be the solution of \eqref{eq:restricted_dual}. It's easy to see that $p(\pm2\alpha\epsilon_0)\geq\alpha\epsilon_0$. Thus, if we find $c$ such that \begin{equation}
|x| \leq c \implies p(x)<\alpha\epsilon_0,
\end{equation} that would imply that $2\alpha\epsilon_0 \geq c$. In order to get that result, one does a local expansion of $p$ around zero. In the following, we cancel the zeroth and first derivatives of $p$ in order to get the right rate.

From the restrictions of $p$, it follows that $- \epsilon_0 \leq \lambda_0 \leq 0$ and \begin{equation}
|x|-\alpha\epsilon_0-\lambda_1 \cdot x \leq |x|-\alpha\epsilon_0-\lambda_0-\lambda_1 \cdot x \leq g(x) \leq \alpha|x| +\alpha\epsilon_0-\lambda_1 \cdot x \quad \for |x|\leq 1
\end{equation} where $g(x) = \sum_{l=2}^L \lambda_l \cdot x^l$. Consequently, $
g(\pm 2\alpha\epsilon_0) \geq \alpha\epsilon_0 \mp \lambda_1 \cdot 2\alpha\epsilon_0
$, that is,\begin{align}\label{eq:set}
2\alpha\epsilon_0 \in \{|x| : g(x) \geq \ell \} \where \ell = \alpha\epsilon_0 \cdot (1+2|\lambda_1|).
\end{align} We proceed to find $c$ such that \begin{equation}
|x|\leq c \implies g(x) < \ell.
\end{equation} Since the first two derivatives of $g$ are zero, we expect $|g(x)|\asymp x^2$ when $x$ is close to zero. The following argument exploits this property.

Let's upper bound $|g|$ \begin{equation}
|g(x)| \leq R \cdot |x| + \alpha\epsilon_0 \quad\where R = \alpha\lor1+|\lambda_1|.
\end{equation} Hence, we have that \begin{equation}
|g(x)| \leq R \cdot |x| + \alpha\epsilon_0 \comma g \in P_L \textand g(0)=g'(0)=0.
\end{equation} By Lemma \ref{lemma:linear_upperbound}, we get \begin{equation}
g(x) \leq |g(x)| \leq C \cdot R \cdot |x| \quad\for |x|\leq \frac{C}{L} \textand \epsilon_0 < \frac{R}{\alpha} \cdot \frac{C}{L}.
\end{equation} Furthermore, by Lemma \ref{lemma:derivative_bound}, it follows that \begin{equation}
g(x) \leq |g(x)| \leq C \cdot R \cdot L \cdot x^2 \quad\for |x|\leq \frac{C}{L} \textand \epsilon_0 < \frac{\alpha}{R} \cdot \frac{C}{L}.
\end{equation} Thus, \begin{equation}
g(x) < \ell \quad \for |x| \leq C \cdot \sqrt{ \frac{\alpha(1+2|\lambda_1|)}{R} \cdot \frac{\epsilon_0}{L}} \textand \epsilon_0 \leq  \frac{R}{\alpha}\cdot \frac{C}{L}.
\end{equation} Using the fact that $\alpha \geq 1-C/L$, see \eqref{eq:alpha_lb}, we get that  \begin{equation}
\frac{\alpha(1+2|\lambda_1|)}{R} = \frac{\alpha(1+2|\lambda_1|)}{\alpha\lor1+|\lambda_1|} \geq 1-C/L \textand \frac{R}{\alpha}=\frac{\alpha\lor1+|\lambda_1|}{\alpha} \geq 1.\end{equation} Furthermore, the above condition reduces to \begin{equation}
g(x) < \ell \quad \for |x| \leq C \cdot \sqrt{ \frac{\epsilon_0}{L}} \textand \epsilon_0 \leq  \frac{C}{L}.
\end{equation} Consequently by \eqref{eq:set} \begin{equation}
\alpha\epsilon_0 \geq C' \cdot \sqrt{ \frac{\epsilon_0}{L}} \quad \for \epsilon_0 < C/L
\end{equation} This implies that the solutions of the primal problem $M_1(\epsilon_0,L)$ \eqref{eq:Mpepsilon} satisfy \begin{equation}
E_{\pi_1}|X| \geq C' \cdot \sqrt{\frac{\epsilon_0}{L}} \geq C' \cdot \epsilon_0 \geq E_{\pi_0}|X| \quad \textif \epsilon_0 < C''/L
\end{equation} in addition to match $L$ moments and be supported on $[-1,1]$.

\end{proof}

We now prove the auxiliary results needed for Lemma \ref{lemma:M1Epsilon}.

\begin{theorem}[Generalized Bernstein's inequality. Theorem 8.1 of \cite{totikReflectionsTheoremAndrievskii2022}. Equation 37 of \cite{kalmykovBernsteinMarkovtypeInequalities2021}]\label{thm:bernstein_ineq}

Let $g \in P_L$ be a $L$th order polynomial, and let $\norm{g}_C$ denote its maximum value over the $C$ set: \begin{equation}
\norm{g}_C = \sup_{x \in C}|g(x)|.
\end{equation} It follows that it $k$th derivative is bounded by. \begin{equation}
|g^{(k)}(x)| \leq C_L(x)\cdot \left[\frac{L}{1-x^2}\right]^k \cdot \norm{g}_{[-1,1]} \quad \for x \in (-1,1)
\end{equation} where \begin{equation}
\lim_{L\to\infty}\sup_{x \in (-1,1)} C_{L}(x) = 1.
\end{equation} Consequently \begin{equation}
\norm{g^{(k)}}_{[-\delta/\sqrt{2},\delta/\sqrt{2}]} \leq (1+o(1)) \cdot 4^k \cdot \left[\frac{L}{\delta}\right]^k \cdot \norm{g}_{[-\delta,\delta]}.
\end{equation}

\end{theorem}

\begin{lemma}[\citet{canonnePriceToleranceDistribution2021}]\label{lemma:derivative_bound}
Let $g \in P_{L}$ be a $L$th order polynomial such that \begin{align}
	g(0) = g'(0) = 0 \textand |g(x)| \leq \alpha |x| \quad \for |x| \leq \delta.
\end{align} It holds that \begin{equation}
|g(x)| \leq C \cdot \alpha \cdot \frac{L}{\delta} \cdot x^2 \quad \for |x|\leq C \cdot \frac{\delta}{L}
\end{equation}
\end{lemma} \begin{proof}
Since $g(0) = g'(0) = 0$, $g$ admits the following representation \begin{equation}
g(x) = x \cdot h(x) \where h(x)=\sum_{l=1}^{L-1} \lambda_l \cdot x^l
\end{equation} where $h \in P_{L-1}$ , $h(0) = 0$ and \begin{align}
\max_{|x|\leq\delta}|h(x)|&\leq\max_{|c|\leq\delta,x\not=0}|h(x)| &&\since h(0)=0\\
&=\max_{|x|\leq\delta,x\not=0}|g(x)/x|\\
&\leq\alpha &&\for |x|\leq\delta \since |g(x)|\leq\alpha|x|  \label{eq:max_upperbound}
\end{align} By Taylor expansion, it holds
\begin{align}
|g(x)| &\leq x^2 \cdot \max_{|c|\leq|x|}|g^{(2)}(c)|   && \since g(0)=g'(0)=0
\end{align} Note that \begin{equation}
|g^{2}(x)|=|x\cdot h^{(2)}(x) + 2h^{(1)}(x)|
\end{equation} Thus for $|x|\leq \delta/\sqrt{2}$ and Bernstein's inequality (Theorem \ref{thm:bernstein_ineq}) \begin{equation}
\max_{|x|\leq \delta/\sqrt{2}}|g^{2}(x)| \leq C \cdot \left[|x| \cdot \frac{L^2}{\delta^2} + \frac{L}{\delta} \right] \cdot \max_{|x|\leq \delta}|h(x)| \leq C \cdot \frac{L}{\delta} \cdot \alpha \quad \for |x|\leq C \cdot \left[\frac{\delta}{\sqrt{2}}\wedge\frac{\delta}{L}\right]
\end{equation} Consequently \begin{equation}
|g(x)| \leq C \cdot x^2 \cdot \frac{L}{\delta} \cdot \alpha \quad \for |x|\leq C' \cdot \delta/L \textand L > C''.
\end{equation}

\end{proof}

\begin{lemma}[\citet{canonnePriceToleranceDistribution2021}]\label{lemma:linear_upperbound}
Let $g \in P_{L}$ and $\alpha>1$ be a $L$th order polynomial such that \begin{align}
g(0) = g'(0) = 0 \textand
|g(x)| \leq \alpha|x|+\beta \quad \for |x| \leq \delta
\end{align} It holds that \begin{equation}
|g(x)| \leq 2\alpha|x| \quad\for |x|\leq \delta/L \quad \textif \beta < C \cdot \alpha \cdot \frac{\delta}{L}
\end{equation}
\end{lemma} \begin{proof}
The claim is trivial if $|x| > \frac{\beta}{\alpha}$. The claim is also true for $x=0$ since $g(0)=0$. Henceforth, we implicitly assume that $x\not=0$ to simplify the notation. For $|x| \leq \frac{\beta}{\alpha}$, assume that the claim is not true. Then \begin{equation}
\exists y \st |y| \leq \frac{\beta}{\alpha} \textand |g(y)| > 2\alpha|y|
\end{equation} Define \begin{align}
r = \max_{|x|\leq\frac{\beta}{\alpha}} \frac{|g(x)|}{2\alpha|x|} \textand \tilde{g}(x)=\frac{g(x)}{r}
\end{align} It holds that $r>1$ and $\tilde{g} \in P_L$, additionally \begin{align}
|\tilde{g}(x)| &= \frac{|g(x)|}{2\alpha|x|} \cdot \frac{1}{r} \cdot 2\alpha|x| \leq 2\alpha|x| \for |x|\leq\frac{\beta}{\alpha} &&\since \frac{|g(x)|}{2\alpha|x|} \cdot \frac{1}{r} \leq 1 \for |x|\leq\frac{\beta}{\alpha}
\end{align} Furthermore, since $\frac{|g(x)|}{2\alpha|x|}$ is a continuous function, it achieves a maximum in the closed interval $|x|\leq\frac{\alpha}{\beta}$. Thus, \begin{equation}
\exists z \st |z|\leq\frac{\alpha}{\beta} \textand |\tilde{g}(z)| = 2\alpha|z|
\end{equation} However, we show that such $z$ cannot exists leading to a contradiction.

We have established that \begin{equation}
\tilde{g}(x) \in P_L \textand \tilde{g}(0)=\tilde{g}'(0)=0 \textand \tilde{g}(x)\leq2\alpha|x| \for |x|\leq\frac{\beta}{\alpha}
\end{equation} By assumption $\frac{\beta}{\alpha} < C \cdot \frac{\delta}{2}$ and lemma \ref{lemma:derivative_bound}, it follow that \begin{equation}
|\tilde{g}(x)|\leq
\frac{4\sqrt{2}\alpha L}{\delta} \cdot x^2 \quad \for |x|\leq \delta/L
\end{equation} The desired contraction immediately follows \begin{equation}
\for |x|\leq\frac{\beta}{\alpha}\quad |\tilde{g}(x)|<2\alpha|x|
\impliedby C \cdot \frac{L}{\delta}\cdot|x|<1
\impliedby \frac{\beta}{\alpha} < C \cdot \frac{\delta}{L}
\end{equation}

\end{proof}

\begin{corollary}[Of Lemma \ref{lemma:linear_upperbound} and Lemma \ref{lemma:derivative_bound}]\label{polynomial_bound}

Let $g \in P_{L}$ and $\alpha>1$ such that \begin{align}
g(0) = g'(0) = 0 \textand
|g(x)| \leq \alpha\cdot |x|+\beta \quad \for |x| \leq \delta
\end{align} It holds that \begin{equation}
|g(x)| \leq 4\sqrt{2} \cdot \alpha \cdot \frac{L}{\delta} \cdot x^2 \quad\for |x|\leq \frac{\delta}{\sqrt{2}}\quad \textand \beta <  \frac{\alpha}{2\sqrt{2}} \cdot \frac{\delta}{L}.
\end{equation}
\end{corollary}

\section{Critical separation for the Gaussian white noise model}\label{sec:GaussianWhiteNoiseEquivalence}

In the following we prove Lemma \ref{lemma:GaussianWhiteNoiseEquivalence}, which provides lower and upper bounds on the critical separation for the Gaussian white noise model. The arguments are based on sections 4.1 and 4.2 of \citet{ingsterTestingHypothesisWhich2001}. Finally, we note that the lower bound depends on Lemma \ref{lemma:discretization_lb}, which is stated and proved at the end of this section.

\begin{proof}[Proof of Lemma \ref{lemma:GaussianWhiteNoiseEquivalence}]

\textbf{Upper bound.} Assume that we observe a realization of a Gaussian white noise model \eqref{eq:gaussian_white_noise} \begin{equation}
dX(t) = f(t)\ dt + \sigma \cdot dW(t) \quad \for t \in [0,1].
\end{equation} Our goal is to test the hypotheses \eqref{eq:testing_gaussian_white_noise} \begin{equation}
H_0: f \in \GWnull \vs H_1 : f \in \GWalt,
\end{equation} by using a tolerant test designed fora Gaussian sequence model \eqref{eq:gaussian_testing_lp}.

We note that observing the Gaussian white noise model \eqref{eq:gaussian_white_noise} corresponds to having access to any projection in $L_2[0,1]$: \begin{equation}\label{eq:gaussian_white_noise_projection}
X(g) = v(g) + \sigma \cdot W(g) \quad \forall\ g \in L_2[0,1]
\end{equation} where \begin{equation}
v(g) = \int_{0}^1 g(t)\cdot f(t)\ dt \textand W(g) = \int_{0}^1 g(t)\ dW(t) \sim \mathcal{N}(0,\norm{g}_2^2).
\end{equation}

Let $\Delta_d$ be a uniform partition of $[0,1]$ into $d$ intervals of length $1/d$ \begin{equation}
I_i = \left[\frac{i-1}{d},\frac{i}{d}\right] \for 1\leq i \leq d
\end{equation} and define \begin{equation}
\varphi_i = d^{1/2} \cdot \delta_{I_i}\period
\end{equation} It follows that $\varphi=\{\varphi_i\}_{i=1}^d$ is an orthonormal basis in in $L_2[0,1]$: \begin{equation}
\inner{\varphi_i}{\varphi_j}_{L_2[0,1]} = d \int \delta_{I_i} \cdot \delta_{I_j} =  \delta_{i=j}\period
\end{equation} Define the following $d$ observation by projecting onto the orthonormal basis: \begin{equation}
X_i = X(\varphi_i) \sim \mathcal{N}(v_i,\sigma^2) \quad \for 1\leq i \leq d.
\end{equation} For any $f \in L_2[0,1]$, let $f_{\varphi}$ be its projection onto the above basis: \begin{equation}
f_\varphi = \sum_{i=1}^d \varphi_i \cdot v_i  \where v_i = \inner{f}{\varphi_i}_{L_2[0,1]}.
\end{equation} Since $\varphi=\{\varphi_i\}_{i=1}^d$ is a basis, under the null hypothesis, $H_0: \norm{f}_p \leq \epsilon_0$, it holds that \begin{equation}
\epsilon_0 \geq \norm{f}_p \geq \norm{f_\varphi}_p = h(0,p) \cdot \norm{v}_p
\end{equation} where $h(0,p)=d^{1/2-1/p}$. Consequently, \begin{equation}
\norm{v}_p \leq \tilde{\epsilon}_0 \where \tilde{\epsilon}_0  = h(0,p)^{-1} \cdot \epsilon_0.
\end{equation}

Under the alternative hypothesis $H_1: \epsilon_1 \leq \norm{f}_p$ and $\norm{f}_{s,p,q}\leq L$, it follows by Lemma \ref{lemma:projection_error} that \begin{equation}
\norm{f_\varphi}_p \geq C_2 \cdot \epsilon_1 - C_3 \cdot C \cdot d^{-s} = \frac{C_2}{2} \epsilon_1 \quad \for  d = \left[\frac{C_3C}{C_2/2}\right]^{1/s} \cdot \epsilon_1^{-1/s}.
\end{equation} Finally, by Equation 4.2 of \citet{ingsterTestingHypothesisWhich2001}, it holds that \begin{equation}
\norm{v}_p \geq \tilde{\epsilon}_1  \where \tilde{\epsilon}_1 = \frac{C_0C_2}{2} \cdot h(0,p)^{-1} \cdot \epsilon_1.
\end{equation}

Consequently, we have reduced the Gaussian white noise problem to the following Gaussian sequence:
\begin{align}\label{eq:reduces_gs_model}
&\Given X_i \sim \mathcal{N}(v_i,\sigma^2) \for 1\leq i \leq d\\
&\Test H_0: \norm{v}_p \leq \tilde{\epsilon}_0 \vs H_1: \norm{v}_p \geq \tilde{\epsilon}_1.
\end{align}

Let $\psi$ be the minimax test statistic for \eqref{eq:reduces_gs_model}. We define the following test for hypotheses \eqref{eq:testing_gaussian_white_noise} \begin{equation}
\psi(X,\epsilon_0) = \psi(\tilde{X},\tilde{\epsilon}_0) \where \tilde{X} = \left\{X_i\right\}_{i=1}^d.
\end{equation} It follows that the hypotheses \eqref{eq:testing_gaussian_white_noise} can be consistently distinguished for any \begin{equation}
\epsilon_1 \st \epsilon_1^*(\tilde{\epsilon}_0,\GS)-\tilde{\epsilon}_0 \leq \tilde{\epsilon}_1 - \tilde{\epsilon}_0 \asymp h(0,p)^{-1} \cdot \left(\epsilon_1-\epsilon_0\right)
\end{equation} where $d \asymp \epsilon_1^{-1/s}$. Hence, we have that \begin{align}
\epsilon_1^*(\epsilon_0,\GW)-\epsilon_0 \lesssim \left(\epsilon_1^*(\tilde{\epsilon}_0,\GS)-\tilde{\epsilon}_0\right) \cdot h(0,p)
\end{align}where $d^{-s} \asymp
\epsilon_1^*(\tilde{\epsilon}_0,\GS) \cdot h(0,p)$.

\textbf{Lower bound.} Using the null and alternative hypotheses sets in Lemma \ref{lemma:discretization_lb}, tolerant testing under the Gaussian white noise model \eqref{eq:testing_gaussian_white_noise} is at least as hard as tolerant testing under the following Gaussian sequence model: \begin{align}
&\Given X_i= X(\varphi_i)\sim \mathcal{N}(v_i,\sigma^2) \quad \for1\leq i \leq d\\
&\Test H_0: \norm{v}_p \leq \tilde{\epsilon}_0 \vs H_1: \norm{v}_p = \tilde{\epsilon}_1
\end{align} where \begin{equation}
\tilde{\epsilon}_0 = C_0  \cdot  h(0,p)^{-1} \cdot \epsilon_0
,\  \tilde{\epsilon}_1 = C_1 \cdot h(0,p)^{-1} \cdot \epsilon_1
,\  d = \left[\frac{C_0}{C_1}\cdot L\right]^{1/s} \cdot \epsilon_1^{-1/s},
\end{equation} and $C_0$ and $C_1$ are positive constants.

Fixing $\epsilon_0$, the hypotheses cannot be consistently distinguished for any \begin{equation}
\epsilon_1 \st \epsilon_1^*(\tilde{\epsilon}_0,\GS)-\tilde{\epsilon}_0 \geq \tilde{\epsilon}_1 - \tilde{\epsilon}_0 \asymp h(0,p)^{-1} \cdot \left(\epsilon_1-\epsilon_0\right)
\end{equation} where $d \asymp \epsilon_1^{-1/s}$. Hence, we have that \begin{align}
\epsilon_1^*(\epsilon_0,\GW)-\epsilon_0 \gtrsim (\epsilon_1^*(\tilde{\epsilon}_0,\GS)-\tilde{\epsilon}_0) \cdot h(0,p)
\end{align}where $d^{-s} \asymp \epsilon_1^*(\tilde{\epsilon}_0,\GS)   \cdot h(0,p)$.

\end{proof}

We conclude the section by proving Lemma \ref{lemma:discretization_lb}.

\begin{lemma}\label{lemma:discretization_lb} Let $\varphi$ be any $(s+2)$-continuously differentiable function supported on $(0,1)$ with $\norm{\varphi}_2=1$, and define the following orthonormal system in $L_2[0,1]$ with disjoint supports \begin{equation}
\varphi_i(t) = d^{1/2} \cdot \varphi(d\cdot t-i+1)\cdot \delta_{I_i} \where I_i = \left[\frac{i-1}{d},\frac{i}{d}\right] \for 1\leq i \leq d.
\end{equation} Henceforth, for any $v\in \R^d$, let \begin{equation}
f_{\varphi,v} = \sum_{i=1}^{d} \varphi_i \cdot v_i\period
\end{equation} For $s\geq 0$, $\epsilon_1 \geq \epsilon_0 \geq0$, define the sets \begin{align}
\GWnullTilde &= \left\{f_{\varphi,v} : v\in \mathbb{R}^d \textand  \norm{v}_p  \leq \tilde{\epsilon}_0   \right\}\\
\GWaltTilde &= \left\{f_{\varphi,v} : v\in \mathbb{R}^d \textand  \norm{v}_p=\tilde{\epsilon}_1 \right\}
\end{align} where \begin{equation}
\tilde{\epsilon}_0 = C_0  \cdot  d^{1/p-1/2} \cdot \epsilon_0
\comma \tilde{\epsilon}_1 = C_1 \cdot d^{1/p-1/2} \cdot \epsilon_1
\textand d = \left[\frac{C_0}{C_1}\cdot L\right]^{1/s} \cdot \epsilon_1^{-1/s},
\end{equation} and $C_0$ and $C_1$ are positive constants.

Finally, consider hypotheses \eqref{eq:testing_gaussian_white_noise}, it follows that \begin{equation}
\GWnullTilde  \subseteq \GWnull\textand
\GWaltTilde \subseteq \GWalt
\end{equation}
\end{lemma}\begin{proof}
Let $h(s,p)=d^{s+1/2-1/p}$, then Equation (4.2) of \citet{ingsterTestingHypothesisWhich2001} states that \begin{equation}
C_0 \cdot \norm{f_{\varphi,v}}_p \leq \norm{v}_p \cdot h(0,p) \leq C_1 \cdot \norm{f_{\varphi,v}}_p \label{eq:Lp_ineq}.
\end{equation} Thus, it holds that \begin{equation}
\norm{v}_p  \leq \tilde{\epsilon}_0 \implies \norm{f_{\varphi,v}}_p \leq \epsilon_0,
\end{equation} which implies that $\GWnullTilde \subseteq \GWnull$. Furthermore, we have that \begin{equation}
\norm{v}_p  \geq \tilde{\epsilon}_1 \implies \norm{f_{\varphi,v}}_p \geq \epsilon_1 \period
\end{equation} By Equation (4.1) of \citet{ingsterTestingHypothesisWhich2001}, it holds that \begin{align}
C_0 \cdot \norm{f_{\varphi,v}}_{s,p,q} \leq \norm{v}_{p} \cdot h(s,p) \leq C_1 \cdot \norm{f_{\varphi,v}}_{s,p,q}  \period\label{eq:besov_ineq}
\end{align}Thus, we guarantee that $\norm{f_{\varphi,v}}_{s,p,q} \leq L$ and $\GWaltTilde \subseteq \GWalt$ whenever \begin{align}
\norm{v}_p \leq \tilde{\epsilon}_1 = CC_0 \cdot h(s,p)^{-1},
\end{align} which is implied by \begin{equation}\label{eq:alternative_constraints}
C_1 \cdot h(0,p)^{-1} \cdot \epsilon_1 =   C_0C \cdot h(s,p)^{-1},
\end{equation} or equivalently \begin{equation}
d^s = \frac{C_0}{C_1} \cdot L \cdot \epsilon_1^{-1}.
\end{equation}
\end{proof}

\section{Critical separation for the Density model}\label{sec:DensityEquivalence}


In this section, we prove Lemma \ref{lemma:DensityEquivalence}, which provides upper and lower bounds for the critical separation of the density model. The proof of the lower bound depends on Lemma \ref{lemma:discretization_lb_multinomial}, which is stated and proved at the end of this section.

\begin{proof}[Proof of Lemma \ref{lemma:DensityEquivalence}]

\textbf{Upper bound.} Assume that we have access to $n$ observations from a density \eqref{eq:density_model} \begin{equation}
X_1,\dots,X_n \iid f \where f \in \D.
\end{equation} Our goal is to test the hypotheses \eqref{eq:density_testing} \begin{equation}
H_0: f \in \Dnull \vs H_1 : f \in \Dalt,
\end{equation} by using a tolerant test for Multinomial distributions.

Let $\{I_i\}_{i=1}^d$ be a uniform partition of $[0,1]$ into $d$ intervals of length $1/d$: \begin{equation}
I_i = \left[\frac{i-1}{d},\frac{i}{d}\right] \for 1\leq i \leq d
\end{equation} and define the functions \begin{equation}
\varphi_i = d^{1/2} \cdot \delta_{I_i} \for 1\leq i \leq d.
\end{equation} It follows that $\{\varphi_i\}_{i=1}^d$ is an orthonormal basis in $L_2[0,1]$: \begin{equation}
\inner{\varphi_i}{\varphi_j}_{L_2[0,1]} = d \int \delta_{I_i} \cdot \delta_{I_j} =  \delta_{i=j}
\end{equation} Define the following $d$ observations:
\begin{equation}
Y_i = \sum_{j=1}^n \delta_{I_i}(X_j)\quad  1\leq i\leq d
\end{equation} For any $f \in L_2[0,1]$, let $f_\varphi$ be its projection into $\{\varphi\}_{i=1}^d$. Namely, \begin{equation}
f_\varphi = \sum_{i=1}^d \varphi_i \cdot v_i(f) \where v_i(f) = \inner{f}{\varphi_i}_{L_2[0,1]}=d^{1/2} \cdot F_i \textand F_i = \int_{I_i} f(u)\ du.
\end{equation} Consequently, we have that \begin{equation}
(Y_1,\dots,Y_d) \sim \Multi_d(n,F).
\end{equation} Since $\varphi=\{\varphi_i\}_{i=1}^d$ is a basis, under the null hypothesis $H_0: \norm{f-g}_p\leq \epsilon_0$, it holds that \begin{equation}
\epsilon_0 \geq \norm{f-g}_p \geq \norm{f_\varphi-g_\varphi}_p = h(0,p) \cdot \norm{v(f)-v(g)}_p = h(0,p) \cdot d^{1/2} \cdot \norm{F-G}_p.
\end{equation} where $h(0,p)=d^{1/2-1/p}$. Thus, under the null hypothesis, it holds that \begin{equation}
\norm{F-G}_p \leq
\tilde{\epsilon}_0 \where \tilde{\epsilon}_0 = d^{-1/2} \cdot h(0,p)^{-1} \cdot \epsilon_0
\end{equation}

Under the alternative hypothesis $H_1: \epsilon_1 \leq \norm{f-g}_p$ and $\norm{f-g}_{s,p,q}\leq L$. By Lemma \ref{lemma:projection_error}, it follows that \begin{equation}
\norm{f_\varphi-g_\varphi}_p \geq C_2 \cdot \epsilon_1 - C_3 \cdot L \cdot d^{-s} = \frac{C_2}{2} \epsilon_1 \quad \for  d = \left[\frac{C_3}{C_2/2}\cdot L\cdot \frac{1}{\epsilon_1}\right]^{1/s}
\end{equation}
Furthermore, by Equation (4.2) of \citet{ingsterTestingHypothesisWhich2001}, it holds that \begin{equation}
C_0 \cdot \norm{f_{\varphi}-f_\varphi}_p \leq d^{1/2}\cdot \norm{F-G}_p \cdot h(0,p).
\end{equation} Consequently, under the alternative hypothesis, it holds that \begin{equation}
\norm{F-G}_p \geq \tilde{\epsilon}_1  \where \tilde{\epsilon}_1 = \frac{C_0C_2}{2} \cdot d^{-1/2} \cdot h(0,p)^{-1} \cdot \epsilon_1.
\end{equation}

In summary, we have reduced the tolerant testing under the density model to tolerant testing under the Multinomial model
\begin{align}\label{eq:multinomial_testing_lp}
&\Given Y \sim \Multi_d(n,F)\\
&\Test H_0: \norm{F-G}_p \leq \tilde{\epsilon}_0 \vs H_1: \norm{F-G}_p \geq \tilde{\epsilon}_1
\end{align}

Let $\tilde{\psi}$ be the minimax test statistic for \eqref{eq:multinomial_testing_lp}. We define the following test for hypotheses \eqref{eq:density_testing} \begin{equation}
\psi(X,\epsilon_0) = \tilde{\psi}(Y,\tilde{\epsilon}_0)
\end{equation} It follows that the hypotheses \eqref{eq:density_testing} can be consistently distinguished for \begin{equation}
\epsilon_1 \st \epsilon_1^*(\tilde{\epsilon}_0,\M_d)-\tilde{\epsilon}_0 \leq \tilde{\epsilon}_1 - \tilde{\epsilon}_0 \asymp d^{-1/2}\cdot h(0,p)^{-1} \cdot \left(\epsilon_1-\epsilon_0\right)
\end{equation} where $d \asymp \epsilon_1^{-1/s}$. Hence, we have that \begin{equation}
\epsilon_1^*(\epsilon_0,\D)-\epsilon_0 \lesssim \left(\epsilon_1^*(\tilde{\epsilon}_0,\M_d)-\epsilon_0\right) \cdot h(0,p) \cdot d^{1/2}
\end{equation} where $d^{-s} \asymp \epsilon_1^*(\tilde{\epsilon}_0,\M_d) \cdot d^{1/2} \cdot h(0,p)$.

\textbf{Lower bound.} Using the null and alternative hypotheses sets in Lemma \ref{lemma:discretization_lb_multinomial}, tolerant testing under the density model \eqref{eq:density_testing} is at least as hard as tolerant testing under the following Multinomial model: \begin{align}
&\Given (Y_1,\dots,Y_d) \sim \Multinomial_d(n,F) \quad \\
&\Test H_0: \norm{F-G}_p \leq \tilde{\epsilon}_0 \vs H_1: \norm{F-G}_p = \tilde{\epsilon}_1
\end{align} where \begin{align}
\tilde{\epsilon}_0 = C_0  \cdot  h(0,p)^{-1} \cdot \epsilon_0
\comma \tilde{\epsilon}_1 &= C_1 \cdot h(0,p)^{-1} \cdot d^{-1/2} \cdot \epsilon_1
\textand d = \left[\frac{C_0}{C_1}\cdot L\right]^{1/s} \cdot \epsilon_1^{-1/s},
\end{align} $h(0,p)=d^{1/2-1/p}$, and $C_0$ and $C_1$ are positive constants. Fixing $\epsilon_0$, the hypotheses cannot be consistently distinguished for any \begin{equation}
\epsilon_1 \st \epsilon_1^*(\tilde{\epsilon}_0,\Multinomial_d)-\tilde{\epsilon}_0 \geq \tilde{\epsilon}_1 - \tilde{\epsilon}_0 \asymp h(0,p)^{-1} \cdot d^{-1/2} \cdot \left(\epsilon_1-\epsilon_0\right)
\end{equation} where $d \asymp \epsilon_1^{-1/s}$. Hence, we have that \begin{align}
\epsilon_1^*(\epsilon_0,\D)-\epsilon_0 &\gtrsim \left(\epsilon_1^*(\tilde{\epsilon}_0,\Multinomial_d)-\tilde{\epsilon}_0\right) \cdot h(0,p)
\end{align}where $d^{-s} \asymp \epsilon_1^*(\tilde{\epsilon}_0,\Multinomial_d) \cdot h(0,p) \cdot d^{1/2}$.

\end{proof}

We conclude this section by proving Lemma \ref{lemma:discretization_lb_multinomial}.

\begin{lemma}\label{lemma:discretization_lb_multinomial}
Let $\varphi$ be any $(s+2)$-continuously differentiable function supported on $(0,1)$ with $\norm{\varphi}_2=1$ and $\norm{\varphi}_1=1$, and define the following orthonormal system in $L_2[0,1]$ with disjoint supports  \begin{equation}
\tilde{\varphi}_i(t) = d^{1/2} \cdot \varphi(d\cdot t-i+1)\cdot \delta_{I_i}(t)\where I_i = \left[\frac{i-1}{d},\frac{i}{d}\right] \for 1\leq i \leq d.
\end{equation} Henceforth, for any $F,G\in \Delta^d$, let \begin{equation}
f_\varphi = \sum_{i=1}^d \varphi_i \cdot F_i \textand g_\varphi = \sum_{i=1}^d \varphi_i \cdot G_i \where \varphi_i =  \tilde{\varphi}_i \cdot d^{1/2}.
\end{equation}

For $s\geq 0$, $\epsilon_1 \geq \epsilon_0 \geq0$, define the sets \begin{align}
\DnullTilde &= \left\{f_{\varphi} \in \D : F\in \Delta^d \textand  \norm{F-G}_p  \leq \tilde{\epsilon}_0  \right\} \\
\textand \DaltTilde &= \left\{f_{\varphi} \in \D : F\in \Delta^d \textand  \norm{F-G}_p=\tilde{\epsilon}_1 \right\}
\end{align} where \begin{align}
\tilde{\epsilon}_0
= C_0  \cdot d^{1/p-1}  \cdot \epsilon_0, \tilde{\epsilon}_1
= C_1 \cdot d^{1/p-1} \cdot \epsilon_1,
\textand d = \left[\frac{C_0}{C_1}\cdot L\right]^{1/s} \cdot \epsilon_1^{-1/s},
\end{align}
and $C_0$ and $C_1$ are positive constants. Consider hypotheses \eqref{eq:density_testing}, it follows that \begin{equation}
\DnullTilde  \subseteq \Dnull \textand \DaltTilde \subseteq \Dalt.
\end{equation}
\end{lemma}\begin{proof}
Note that \begin{align}
\int f_\varphi = 1 &\impliedby 1=\sum_{i=1}^d \varphi_d \cdot F_i\\
&\impliedby F \in \Delta^d \textand \int_{I_i} \varphi_i(t)\ dt = 1\\
&\impliedby F \in \Delta^d \textand \int_{I_i} \varphi(d\cdot t-i+1)\ dt = \frac{1}{d}\\
&\impliedby F \in \Delta^d \textand \int_{0}^1 |\varphi(t)|\ dt = 1\period
\end{align} Analogously, we have that $\int g_\varphi = 1$. Following a similar argument to the proof of Lemma \ref{lemma:discretization_lb}, it follows that $
\norm{f_\varphi-g_\varphi}_p \leq \epsilon_0$ for $f_\varphi \in \DnullTilde$. Thus, $\DnullTilde \subseteq \Dnull$. Furthermore, $\norm{f_\varphi-g_\varphi}_p = \epsilon_1$ and $\norm{f_\varphi-g_\varphi}_{s,p,q} \leq L$ for $f_\varphi \in \DaltTilde$. Thus, $\DaltTilde \subseteq \Dalt$.
\end{proof}

\section{Critical separation for the Poisson sequence model}\label{sec:PoissonSequenceEquivalence}

This section reviews the known fact that Multinomial and Poisson sequence models are equivalent when restricted to the simplex. The equivalence is based on the Poissonization trick  \citep{canonneTopicsTechniquesDistribution2022,kimConditionalIndependenceTesting2023}, and the testing algorithm proposed by \citet{neykovMinimaxOptimalConditional2021}.

Henceforth, restrict $\lambda$ and $\lambda_0$ to be a positive $d$-dimensional probability vectors \begin{equation}
\lambda,\lambda_0 \in \Delta^d,\ \min_{1\leq i\leq d}\lambda_i >0  \textand \min_{1\leq i\leq d}\lambda_{0,i} >0.
\end{equation} Furthermore, let $\epsilon_1^*(n,\Poisson_d)$ denote the critical separation for tolerant testing under the Poisson sequence model \begin{align}\label{eq:poisson_testing}
&\Given X_i \sim \Poi(n \cdot \lambda_i) \quad \for 1\leq i \leq d \\
&\Test H_0: (\lambda-\lambda_0) \in V_0 \vs
H_1: (\lambda-\lambda_0) \in V_1
\end{align} and let $\epsilon_1^*(n,\Multi_d)$ denote the critical separation for tolerant testing under the Multinomial model \begin{align}\label{eq:multinomial_testing}
&\Given X_1,\dots,X_n \iid \Multi(1,\lambda) \where \Multi(1,\lambda) = \sum_{i=1}^d\delta_{i}\cdot \lambda_i\\
&\Test H_0: (\lambda-\lambda_0) \in V_0 \vs
H_1: (\lambda-\lambda_0) \in V_1.
\end{align} It holds that \begin{equation}
\epsilon_1^*(n,\Poisson_d) \asymp \epsilon_1^*(n,\Multinomial_d)
\end{equation}

\textbf{From Multinomial to Poisson.} Let $\psi_n$ be a test for problem \eqref{eq:poisson_testing} that controls type-I and type-II error by $\alpha$ and $\beta$ respectively. Furthermore, for $k \leq n$, define \begin{equation}
X^k = \sum_{i=1}^{k}X_i \sim \Multi(k,\lambda),
\end{equation} and note that \begin{equation}
X^{\tilde{n}}| \tilde{n}\sim \Multi(\tilde{n},\lambda) \textand X^{\tilde{n}} \sim \Poi(n\cdot \lambda) \for \tilde{n} \sim \Poi(n/2).
\end{equation} We proceed as follows: sample $\tilde{n} \sim \Poi(n/2)$, $\tilde{n} > n$, accept the null hypothesis. Otherwise, return $\psi_{\tilde{n}}(X^{\tilde{n}})$. That is, we construct the following randomized test \begin{equation}
\psi(X) = \psi_{\tilde{n}}(X^{\tilde{n}})\cdot I( \tilde{n} \leq n) \quad \where \tilde{n} \sim \Poi(n/2).
\end{equation} First note that the type-I error is bounded by $\alpha$: \begin{align}
P_{X,\tilde{n}}(\tilde{\psi}(X)=1)\leq P_{X^{\tilde{n}} \sim \Poi(n \cdot \lambda)}(\psi_{\tilde{n}}(X^{\tilde{n}})=1)
\leq \alpha
\end{align} where the last inequality follows from the facts that $X^{\tilde{n}} \sim \Poi(n\cdot \lambda)$ and $\psi_{\tilde{n}}$ is a valid test. Furthermore, the type-II error is bounded by $2\beta$. To see that, first note that by the concentration of measure of Poisson random variables, and the fact that $\psi_{\tilde{n}}$ is a powerful test, we have that
\begin{align}
P_{X,\tilde{n}}(\tilde{\psi}(X)=0)=P_{X^{\tilde{n}} \sim \Poi(n \cdot \lambda)}(\psi_{\tilde{n}}(X^{\tilde{n}})=0) + P_{\tilde{n} \sim \Poi(n/2)}(\tilde{n} > n)
\leq \beta + e^{-n/8}.
\end{align} Consequently, \begin{equation}
P_{X,\tilde{n}}(\tilde{\psi}(X)=0) \leq 2\beta \quad \for n \geq 8\log(1/\beta).
\end{equation} In summary, we have that \begin{equation}
\epsilon_1^*(n,\Multinomial_d) \gtrsim \epsilon_1^*(n,\Poisson_d) \for n \gtrsim 1.
\end{equation}

\textbf{From Poisson to Multinomial.} Let $
X = (X_1,\dots,X_d)$, it holds that \begin{equation}
X | K \sim \Multi(K,\lambda) \where K = \sum_{i=1}^d X_i\sim \Poi(n).
\end{equation} Let $\psi_k$ be a test that controls type-I and II errors in \eqref{eq:multinomial_testing} for $n=k$, and define the test  \begin{equation}
\psi(X) = \psi_{K}(X)\cdot I(\mathcal{A}) \where \mathcal{A} = \left\{\frac{n}{2} \leq K \leq \frac{3}{2}n\right\}.
\end{equation} For type-I error, it follows that \begin{align}
P_{X}(\tilde{\psi}(X)=1) \leq E_K \left[I(\mathcal{A})\cdot E_{X|K}\left[\psi_K(X)\right]\right]
\leq \alpha,
\end{align} where in the last inequality we used the fact that $X|K\sim \Multi(K,\lambda)$ and $\psi_K(X)$ is a valid test. Finally, the type-II error is at most doubled \begin{align}
P_{X}(\psi(X)=0) &= P_{X}(\psi(X)=1|\mathcal{A}) + P(\mathcal{A}^c)\\
&\leq E_K \left[I(\mathcal{A})\cdot E_{X|K}\left[1-\psi_K(X)\right]\right] + 2\exp(-n/12) \\
&\leq 2\beta \for n \geq 12\cdot\log(2/\beta),
\end{align} where we used the concentration of measure for the Poisson random variable $K$ and the fact that $\psi_K(X)$ is powerful.

Note that we have mapped one Poisson hypothesis testing problem to a set of Multinomial hypothesis testing problems. Thus, the overall difficulty is dominated by the hardest of the problems. Namely, \begin{equation}
\epsilon_1^*(n,\Poisson) \gtrsim \max_{\frac{n}{2}\leq k \leq \frac{3}{2}n} \epsilon_1^*(k,\Multinomial_d) = \epsilon_1^*(\frac{n}{2},\Multinomial_d) \quad \for n \gtrsim 1
\end{equation}

\section{Besov norm}\label{sec:besov_norm}

We refer the reader to section 4.3 of \citet{gineMathematicalFoundationsInfinitedimensional2016} for an introduction to the Besov norm. In the following, we state the definition of the Besov norm, followed up by a lemma that indicates that projections of functions in a Besov ball are well behaved, which is needed for our lower-bound arguments for the Gaussian white-noise model, see Section \ref{sec:GaussianWhiteNoiseEquivalence}, and the density model, see  Section \ref{sec:DensityEquivalence}.

The Besov norm is defined as
\begin{equation}
\norm{f}_{s,p,q} = \norm{f}_p + \begin{cases}
\left[\int_0^\infty\ \left(\frac{w_{r,p}(f,u)}{u^{s}}\right)^{q}\ \frac{du}{u} \right]^{1/q} &\textif 1\leq q <\infty \\
\sup_{0<u<1} \frac{w_{r,p}(f,u)}{u^{s}} &\textif q = \infty
\end{cases}
\end{equation} where $\norm{f}_p$ is the norm in $L_p(0,1)$ and $w_{r,p}(f,u)$ is the $L_p$ norm of the $r$th difference of $f$: \begin{equation}
w_{r,p}(f,u)=\int_{0}^{1-ru} |\left(\Delta_{u}^{(r)}f\right)(t)|^p\ dt \where  \left(\Delta_{u}^{(r)}f\right)(t)=\sum_{k=0}^rC_r^k(-1)^kf(t+ku).
\end{equation}

The previous definitions allow us to state that low-frequency projections of functions that live in a Besov space preserve the size of the $L_p$ norm. The following lemma can be found in inequality 5.16 of part III of \citet{ingsterAsymptoticallyMinimaxHypothesisI1993} or inequality 4.10 of \citet{ingsterTestingHypothesisWhich2001}. A simple proof is provided in proposition 2.16 of \citet{ingsterNonparametricGoodnessofFitTesting2003}.

\begin{lemma}[Projection of functions on a Besov ball]\label{lemma:projection_error} Let $p\geq1$ and $s > 0$. For any $f\in L_2(0,1)$ such that $\norm{f}_{s,p,q} \leq L$. Furthermore, let $\varphi=\{\varphi_i\}_{i=1}^d$ be an $d$-dimensional orthonormal system in $L_2[0,1]$, and define the $L_2$ projection of $f$ onto $\varphi$: \begin{equation}
f_\varphi = \sum_{i=1}^d \varphi_i \cdot v_i  \where v_i = \inner{f}{\varphi_i}_{L_2[0,1]}.
\end{equation} It holds that \begin{equation}\label{eq:projection_error}
\norm{f_\varphi}_p \geq C_2 \cdot \norm{f}_p - C_3 \cdot L \cdot d^{-s}
\end{equation} where $C_2$ and $C_3$ are positive constants that depend only on $(p,s)$.\end{lemma} Note that \eqref{eq:projection_error} is weaker than guaranteeing a bound on the approximation error \begin{equation}
\norm{f-f_\varphi}_p\leq C' \cdot L \cdot d^{-s}
\end{equation} which doesn't hold for $s > 1$. In other words, $f_\varphi$ might be a bad approximation for $f$ but still have good projection properties, which allow us to construct consistent tests.


\etocdepthtag.toc{mtreferences}
\addcontentsline{toc}{section}{References}
\bibliography{paper}

\end{document}